\documentclass[11pt,,leqno,twoside]{article} %12.3.2012, pekka
\usepackage{bbm}
\usepackage{mathrsfs}
\usepackage{amssymb}
\usepackage{amsmath}
\usepackage{amsthm}
\usepackage{amsfonts}
\usepackage{color}
\usepackage{graphicx}
\usepackage[active]{srcltx}
\usepackage{CJK}

\usepackage[cp1252]{inputenc}

\usepackage{mathrsfs}
\usepackage{graphicx}

\usepackage[active]{srcltx}

\allowdisplaybreaks

\usepackage{titletoc}
\titlecontents{section}[0pt]{\addvspace{2pt}\filright}
              {\contentspush{\thecontentslabel\ }}
              {}{\titlerule*[8pt]{.}\contentspage}

\def\rr{{\mathbb R}}
\def\rn{{{\rr}^n}}

\def\cf{{\mathcal F}}

\def\cl{{\mathcal L}}

\def\cp{{\mathcal P}}

\def\fz{\infty}
\def\az{\alpha}

\def\dist{{\mathop\mathrm{\,dist\,}}}
\def\loc{{\mathop\mathrm{\,loc\,}}}
\def\weak{{\mathop\mathrm{ \,weak\,}}}
\def\lip{{\mathop\mathrm{\,Lip}}}

\def\lz{\lambda}
\def\dz{\delta}

\def\ez{\epsilon}

\def\bz{\beta}

\def\gz{{\gamma}}

\def\sz{\sigma}

\def\wz{\widetilde}

\def\osc{{\textrm{osc}}}

\def\bint{{\ifinner\rlap{\bf\kern.35em--}
\int\else\rlap{\bf\kern.45em--}\int\fi}\ignorespaces}

\def\bbint{{\ifinner\rlap{\bf\kern.35em--}
\hspace{0.078cm}\int\else\rlap{\bf\kern.45em--}\int\fi}\ignorespaces}

\def\osc{ \mathop \mathrm{\, osc\,} }

\def\diam{{\mathop\mathrm{\,diam\,}}}

\def\usc{\mathop\mathrm{\,USC\,}}
\def\lsc{\mathop\mathrm{\,LSC\,}}
\def\cica{\mathop\mathrm{\,CICA\,}}
\def\cicb{\mathop\mathrm{\,CICB\,}}

\def\r{\right}
\def\lf{\left}

\newtheorem{thm}{Theorem}[section]
\newtheorem{lem}[thm]{Lemma}%[section]     %@@!!@@!!
\newtheorem{rem}[thm]{Remark}%[section]    %@@!!@@!!
\newtheorem{cor}[thm]{Corollary}%[section]    %@@!!@@!!
\newtheorem{defn}[thm]{Definition}%[section]    %@@!!@@!!
\numberwithin{equation}{section}

\begin{document}
%\begin{CJK*}{GBK}{song}

\arraycolsep=1pt

\title{\Large\bf
Uniqueness of absolute minimizers
for $L^\fz$-functionals  involving Hamiltonians $H(x,p)$
%\footnotetext{\hspace{-0.35cm}
%%\noindent{2000 {\it Mathematics Subject Classification:}}
%\noindent  {\it Key words and phases:}  Dirichlet form, Diffusion process, intrinsic distance,
% differential structure, $L^\fz$-variational problem, absolute minimizer
%\endgraf Pekka Koskela was supported by the Academy of Finland
%grant 131477. Nageswari Shanmugalingam was partially supported by grant \#200474 from
%the Simons Foundation and by NSF grant \# DMS-1200915.
%Yuan Zhou was supported by
%Program for New Century Excellent Talents in University of Ministry of Education of China
%and National Natural Science Foundation of China (Grant No. 11201015).
%%\endgraf $^\ast$ Corresponding author.
%}
}
\author{Qianyun Miao, Changyou Wang, and Yuan Zhou
}
\date{ }
\maketitle

\begin{center}
\begin{minipage}{13.5cm}\small
{\noindent{\bf Abstract.}\quad
For a bounded domain $U\subset\rn$, consider the
 $L^\fz$-functional  involving a nonnegative Hamilton function $H:\overline U\times\rn\to [0,\fz)$.
In this paper, we will establish  the uniqueness of  absolute minimizers $u\in W^{1,\fz}_\loc(U)\cap C(\overline U)$
for $H$,
under the  Dirichlet boundary value $g\in C(\partial U)$, provided

\noindent (A1) $H$ is lower semicontinuous in $\overline U\times\rn$, and $H(x,\cdot)$ is convex for any $x\in\overline U$.

\noindent (A2)  $\displaystyle H(x,0)=\min_{p\in \rn}H(x,p)=0$  for any $  x\in \overline U$,
and $\displaystyle\bigcup_{x\in \overline U}\big\{p: H(x,p)=0\big\}$  is contained in a hyperplane of $\rn$.

\noindent (A3) For any $\lz>0$, there exist  $\displaystyle 0<r_\lz\le R_\lz<\fz$, with $\displaystyle\lim_{\lz\to\fz}r_\lz=\fz$,
such that
$$B(0,r_\lz)\subset  \Big\{p\in\rn\ |\  H(x,p)< \lz\Big\}\subset B(0,R_\lz)\ \forall\  \lz> 0\ \mbox{and}\ x\in \overline U.$$
This generalizes the uniqueness theorem by \cite{j93, jwy, acjs} and \cite{ksz}
to a large class of Hamiltonian functions $H(x,p)$ with $x$-dependence.
As a corollary, we confirm an open question on the uniqueness
of absolute minimizers posed by {\cite{jwy}}.
The proofs rely on geometric structure of the action function $\mathcal L_t(x,y)$ induced by $H$, and the
identification of the absolute subminimality of $u$ with  convexity
of the Hamilton-Jacobi flow $t\mapsto T^tu(x)$.}
\end{minipage}
\end{center}

\tableofcontents
\contentsline{section}{\numberline{ } References}{51}

\section{Introduction}\label{s1}

The  study of calculus of variations in $L^\fz$, initiated by Aronsson \cite{a1,a2,a3,a4} in 1960's,
has attracted great attentions by analysts in the past decades, see \cite{j93,bjw,ceg,acj,c01, CGW, jwy,y06,y07,c2008,wy11,s05,es08, BEJ, es11a,es11b,acjs,cf,swz, as, as1} and the references therein.
For a bounded domain $ U\subset\rn$  and a Hamiltonian function $H:\overline U \times \rn\to\mathbb R_+$,
the $L^\fz$-functional is defined by
$$\cf_\fz(u,U):=\big\|H(x, Du(x))\big\|_{L^\fz(U)}, \ u\in W_\loc^{1,\fz}(U).$$
From the studies on the model case $H(x,  p)=|p|^2$ (see \cite{a1,a2,a3,a4,j93,ce,ceg}),
it becomes clear that absolute minimizer is the correct notion of minimizer
for the minimization of $\cf_\fz(u,U)$
  over $W^{1,\fz}_\loc(U)$.
A function $u\in W^{1,\fz}_\loc(U)$  is said to be an absolute minimizer
 for $H$ in $U$,
 if for all open set $V\Subset U$, it holds that
$$\cf_\fz(u,V)\le \cf_\fz(v,V),  \ \ \mbox{whenever}\ v\in W^{1,\fz}(V)\cap C(\overline V)\ \mbox{and} \ u\big|_{\partial V}
=v\big|_{\partial V}.$$
To simplify the presentation, an absolute minimizer always refers to an absolute minimizer for $H$ in $U$,
unless there is a confusion.
%We remark that in the model case $H(x,  p)=|p|^2$, an absolute minimizer $u\in W^{1,\fz}_\loc(U)\cap C(\overline U)$  is
%a minimizer of  $\cf$ in $U$, that is,
% $$\cf(u,U)\le \cf(v,U)\quad \mbox{whenever}\ v\in W_\loc^{1,\fz}(U)\cap C(\overline U)\ \mbox{and} \ u|_{\partial U}=v|%_{\partial U};$$
% but the converse is not necessarily correct.
 %Moreover, given arbitrary $g\in\lip(\partial U)$, an absolute minimizer admitting boundary $g$, with respect to $|p|^2$ in $U$,
%  is exactly an absolute minimizing Lipschitz extension of $g$.

%Among several To understand the existence and uniqueness of absolute minimizers in $U$ admitting any given boundary value $g\in C(\partial U)$

%
%   with $u|_{\partial U}=g$,
%for  arbitrarily given $g\in C(\partial U)$.
%
% the absolute minimizer plays a correct notion of minimizers for such $L^\fz$-functional.

The existence of absolute minimizers with a given boundary value  has been extensively studied,
see  \cite{bjw,cpp} and the references therein.
In particular, for any $g\in C(\partial U)$, Barron, Jensen and  Wang \cite{bjw}
have obtained the existence of   absolute minimizers with  boundary value $g$,
under the natural assumption that $H(x,p)$ is lower semicontinuous in $\overline U\times\rn$, and
$H(x,\cdot)$ is quasi-convex  for all $x\in\overline U$, that is, the set
$\{p\in\rn | H(x,p)\le\lz\}$ is convex for $\lz\in\mathbb R$ and $x\in\overline U$.

The issue of uniqueness of  absolute minimizers with a given boundary value is much more subtle. It has
been established when the Hamilton function $H(x,p)$ takes the forms $|p|^2$, $H(p)$, or some special  type
$H(x,p)$, see \cite{j93,acj,jwy,acjs,ksz}.
It is a challenging problem to study the uniqueness of absolute minimizers
for general Hamiltonian functions $H(x,p)$.

When $H(x,p)=|p|^2$, Jensen  obtained in a seminal paper \cite{j93} the uniqueness of absolute minimizers with any given boundary value $g\in C(\partial U)$ by  identifying   absolute minimizer
with  $\fz$-harmonic function, that is, viscosity solution of the $\fz$-Laplacian equation:
\begin{equation}\label{e1.x0}
\Delta_\fz u:=\sum_{i,j=1}^n u_{x_i} u_{x_j} u_{x_ix_j} =0\quad {\rm in}\quad  U.
\end{equation}
The comparison principle is shown by \cite{j93}
for  absolute minimizers $u,v\in  W_\loc^{1,\fz}(U)\cap C(\overline U)$:
\begin{equation}\label{e1.x2}
\max_{x\in U}\big(u(x)-v(x)\big)= \max_{x\in\partial U}\big(u(x)-v(x)\big).
\end{equation}
Alternative proofs of \eqref{e1.x2}  were later found by Peres et al \cite{pssw}, and Armstrong and Smart \cite{as, as1}.
It is readily seen that the uniqueness of absolute minimizers follows from \eqref{e1.x2}.

%The first one is given by Peres et al \cite{pssw} based on  a stochastic game, that is, tug-of-war;
%while another one   by Armstrong-Smart \cite{as}
%based on the equivalence of absolute minimizing with the comparison property with cones of \cite{ceg} and with the convexity of the map $r\to S^ru(x)$, where
%$$S^ru(x)=\sup_{y\in B(x,r)}\frac{u(x)-u(y)}{r}.$$

For general Hamiltonians $H(p)$ with no $x$-dependence,  the comparison principle \eqref{e1.x2} for absolute minimizers was first established  in \cite{jwy} for $H\in C^2$
and later by \cite{acjs} for $H\in C^0$, under the assumption that
$H$ is convex, coercive (i.e., $\displaystyle\lim_{|p|\to\fz}H(p)=\fz$), and
$\displaystyle\big\{p\in\rn\ |\ H(p) = \min_\rn H\big\}$  has no interior points.
The approach by \cite{jwy} is based on the identification  of absolute minimizers with  viscosity solutions of
the  Aronsson equation:
\begin{equation}
\mathcal A_\infty[u]:= \sum_{i,j=1}^n  H_{p_i} (Du)  H_{ p_j} (Du)  u_{x_ix_j}=0 \quad {\rm in}\quad  U.
\end{equation}
While the approach by Armstrong, Crandall, Julin and Smart \cite{acjs} utilizes convexity
of the Hamilton-Jacobi flow, first proven by \cite{js} for $H\in C^2$, and some new ideas by \cite{as}.
Without  assuming $H\in C^2( \rn)$,  the approach based on the Aronsson equation from \cite{j93,jwy} does not seem
to work.
Moreover, as pointed out by \cite{jwy} and \cite{acjs}, it is crucial for the uniqueness
that $\{p\in\rn\ |\ H(p) = 0\}$  has no interior points.
We should point out that when $H(p)$ is some norm on $\rn$, the uniqueness
has  been established by \cite{CGW} and \cite{acj}.

The uniqueness of absolute minimizers for general Hamiltonian functions $H(x,p)$ with $x$-dependence may
fail: non-uniqueness of absolute minimizers was constructed by \cite{jwy} and \cite{y07}
when $H(x,p)=|p|^2+w(x)$, and $w$ satisfies that
 $\{x\in U\ | \  w(x) =\max_{\overline U} w\}\not=\emptyset$ consists of  finitely many points.
In general, when $H\in C^2(\rn\times\rn)$, $\displaystyle\lim_{|p|\to\fz}H(x,p)=\fz$ uniformly
in $x$, and $H(x,\cdot)$ is convex for $x\in\rn$,
set
$$c_0:=\inf_{\phi\in C^1( U)\cap C(\overline U)}\sup_{x\in U}H(x,D\phi(x)),$$
it was shown by \cite{jwy} that for $g\in C(\partial U)$, there is a 1-1 map from the set of viscosity solutions to
the Hamilton-Jacobi equation:
\begin{equation}  \label{e1.x3}
H(x,Du(x))=c_0\ \ {\rm in} \ \  U, \quad u\big |_{\partial U}=g,
\end{equation}
 to the set of absolute minimizers for $H$ in $U$. While equation \eqref{e1.x3} can have
multiple solutions for certain Hamilton functions $H(x,p)$ and  $g\in C(\partial U)$ (see Lions \cite{l}).

% the Hamilton-Jacobi equation
% \begin{equation}\label{e1.x3} H(z,Du(z))=c_0\quad {\rm in} \quad  U;\quad u|_{\partial U}=g\end{equation}
%  may have
 %Notice that in this case, by \cite{bjw},  an absolute minimizer is exactly a   viscosity solutions to
%   the Aronsson equation
%  \begin{equation}\label{e1.x4}\ca[u]=\sum_{i,j =1}^n  H_{p_i} (x,Du )   H_{ p_j} (x,Du ) u_{x_ix_j}+\sum_{i  =1}^n   H_{ p_i} (x,Du )  H_{ x_i} (x,Du )=0\quad {\rm in}\quad  U.\end{equation}

In order to obtain uniqueness of absolute minimizers, with any boundary value $g\in C(\partial U)$,
for general Hamiltonian functions $H(x,p)$,  additional assumptions on $H(x,p)$ to rule out the
possibility of multiple solutions to \eqref{e1.x3} seem to be necessary.
The  following conjecture was proposed by \cite{jwy}:

\smallskip
\noindent {\it Assume
that $H\in C^2(U\times\rn)\cap C(\overline U\times\rn)$ satisfies\\
 (1) $H(x,\cdot)$ is convex for any $x\in\overline U$, and \\
 (2) $0=H(x,0)<H(x,p)$ for any $0\ne p\in\rn$ and $x\in\overline U$.\\
Then, for any $g\in C(\partial U)$, there is  a unique absolute minimizer $u$ for $H$, with $u\big|_{\partial U}=g$.}

  %a unique
%  viscosity solution to  the Aronsson  equation
%  \begin{equation}\lf\{\label{e1.x1} \begin{array}{rl}
% \ca[u] =0&\quad {\rm in}\quad  U\\
%u=g &\quad {\rm on}\quad  \partial U.
%\end{array}\r.
%\end{equation}
\medskip

In this paper, we are able to confirm this conjecture. In fact, we will establish the uniqueness
of absolute minimizers for a large class of Hamiltonian functions $H(x,p)$.

We will consider a class of nonnegative Hamiltonian functions $H:\overline U\times\rn\to\mathbb R_+$ satisfying:

\noindent (A1) $H$ is lower semicontinuous on $\overline U\times\rn$, and $H(x,\cdot)$ is convex for all
$x\in\overline U$.

\noindent (A2) $\displaystyle H(x,0)=\min_{ p\in\rn}H(x,p)=0$  for every $  x\in \overline U$,
and $\displaystyle\bigcup_{x\in \overline U}\big\{p\in\rn\ |\ H(x,p)=0\big\}$ is contained in a hyperplane $P\subset\rn$.

\noindent (A3) There exist  $0<r_\lz\le R_\lz<\fz$, with $\displaystyle\lim_{\lz\to\fz}r_\lz=\fz$,
such that
$$B(0,r_\lz)\subset  \Big\{p\in\rn\ |\ H(x,p)< \lz\Big\}\subset B(0,R_\lz)
\ \ \mbox{for}\ \lz> 0 \ \mbox{and}\ x\in \overline U.$$

Now we state our main theorem.
\begin{thm}\label{t1.x2}
Assume that $H:\overline U\times\rn\to\mathbb R_+$ satisfies (A1), (A2), and  (A3).
Then \eqref{e1.x2} holds for every pair of absolute subminimizer $u\in W^{1,\fz}_\loc(U)\cap C(\overline U)$
and absolute superminimizer $v\in W^{1,\fz}_\loc(U)\cap C(\overline U)$.
Thus, for any $g\in C(\partial U)$,  there exists a unique absolute minimizer $u\in W^{1,\fz}_\loc(U)\cap C(\overline U)$,
with $u=g$ on $\partial U$. If, in addition, $H\in C^2(U\times\mathbb R^n)\cap C(\overline U\times\mathbb R^n)$,
then there exists a unique viscosity solution $u\in C(\overline U)$ to Aronsson's equation:
\begin{equation} \label{AE1}
\sum_{i=1}^nD_{x_i}(H(x,Du))\cdot H_{p_i}(x, Du)=0 \ {\rm{in}}\ U; \ u=g \ {\rm{on}}\ \partial U.
\end{equation}
\end{thm}

The assertion of uniqueness of viscosity solutions to Aronsson's equation (\ref{AE1}) in Theorem \ref{t1.x2}
follows from that of absolute minimizers, since it has been proven by Yu \cite{y06} that if $H(x,p)\in C^2(U\times\mathbb R^n)\cap C(\overline U\times \mathbb R^n)$ is convex and coercive in $p$-variable, uniformly in $x\in \overline U$,
then any viscosity solution $u\in C(\overline U)$ to (\ref{AE1}) is an absolute minimizer for $H$.

Theorem 1.1
extends the main result of \cite{acjs}, since  (A1), (A2), and (A3) reduce to
the same assumptions as in \cite{acjs} when $H=H(p)$ (see Remark 1.4 below).
When $H(x,p)=\langle A(x) p, p\rangle$,
with $A:\overline U\to\rr^{n\times n}$ a symmetric uniformly elliptic matrix-valued measurable function,
the comparison principle \eqref{e1.x2} and the uniqueness of absolute minimizers, with $u=g$ on $\partial U$,
has been established by \cite{bbm} and \cite{ksz}; while when $H(x,p)$ satisfies $H(x,p)=|p|H(x,\frac{p}{|p|})$, 
\eqref{e1.x2} has been recently obtained by \cite{gxy} by extending the approach in \cite{as,ksz}.  

As a corollary, we prove the following result, which answers the above question by \cite{jwy}.

\begin{cor}\label{t1.x1}
Assume that $H\in C(\overline U\times \rn)$ satisfies\\
(i) $H(x,\cdot)$ is convex for all $x\in\overline U$, and \\
(ii) $H(x,0)=0<H(x,p)$ for any $0\not=p\in\rn$ and $x\in\overline U$.\\
Then there exists a unique absolute minimizer $u\in W^{1,\fz}_{\loc}(U)\cap C(\overline U)$
with  $u=g$ on $\partial U$. If, in addition, $H\in C^2(U\times\rn)$, then there exists a unique viscosity
solution $u\in C(\overline U)$ to  Aronsson's equation (\ref{AE1}).
\end{cor}

% \begin{thm}\label{t1.x1}
%Under the assumptions of Conjecture 1.3, if $u,v\in C(\overline U)$ are viscosity solutions of the equation \eqref{e1.x4}, then
%\eqref{e1.x2} holds.
%Consequently, there exists a unique viscosity solution to \eqref{e1.x1}.
%\end{thm}

%We will see that $H(x, p)$ considered in Theorem 1.1 (Conjecture 1.3) satisfies (A1)  through (A3); see the proof of Theorem 1.1.
%If $a(x)$ is bounded away from $0$ and is lower semicontnuous on $\overline U$,
%and $H(p)$ is as considered by ACJS, then $a(x)H(p)$ satisfies above  (A1)  through (A3).
%Moreover, if $A:\overline U\to\rr^{n\times n}$ is semicontinuous and uniformly elliptic, then $A(x)p\cdot p$ also satisfies  (A1)  through (A3).

\begin{rem}\rm
%{\color{red}
%(i) Notice that the assumption (A1) guarantees the existence of absolute minimizer, and also plays a key role in the proof  %of Theorem 1.2.
%Even in the case $H(p)$, the convexity assumption cannot be reduced to quasiconvex.
%But we are not clear if it is possible to remove the lower semicontinuity when proving uniqueness. }

%{\color{red}  (ii) The assumptions (A1) and (A3)   are required to establish the
%comparison property with intrinsic cones of absolute minimizers in \cite{cp},
%which will be used to prove Theorem 1.2.

% \bf Question: (A1), (A3) sharp? }
This example suggests that
(A2) is optimal, since it  can't be replaced by\\
 (A2)$_{\rm{weak}}$: For any $x\in U$, {\it $H(x,0)=\min_{ \rn}H(x,\cdot)=0$,  and $\displaystyle\Big\{p\in\rn\ | \ H(x,p)=0\Big\}$
has no interior points. }
%\end{rem}

\smallskip
\noindent {\it Example}. Let $U=\Big\{x\in\rr^2\ |\ \frac12< |x|< 2\Big\}$ and
$H:\rr^2\to \mathbb R_+$ be a convex function satisfying
$$\Big\{p\in\rr^2\ |\ H(p)=0\Big\}=[-2, 2]\times\{0\}.$$
For $x\in\rr^2\setminus\{0\}$, let $O(x):\rr^2\to\rr^2$ be the rotation mapping $\displaystyle\frac{x}{|x|}$
to $(1,0)^T$.
Set $$\widetilde{H}(x,p)=H( O(x)p), \ \forall\ x\in \overline U\ {\rm{and}}\ p\in\rr^2.$$
Then we have
$$\Big\{p\in\rr^2\ \big|\ \widetilde{H}(x,p)=0\Big\}=\lf\{ \frac{t x}{|x|}\ \big|\  t\in[- 2, 2]\r\}.$$
Hence $\widetilde{H}$ satisfies (A1), (A2)$_\weak$ and (A3), but not (A2).
We want to show that the uniqueness fails for absolute minimizers
for $\widetilde{H}$ in $U$.
In fact, let
$$u(x)=|x|-\frac12
\  {\rm{and}}\ v(x)=  \frac25(|x|^2-\frac14)\ \ {\rm{for}}\ \ x\in\overline U.
$$
Then $u,v\in C(\overline U)\cap C^2(U)$ satisfy that
$u\big|_{\partial U}=v\big|_{\partial U}$, and
$$D u(x)= \frac{ x}{|x|} \ {\rm{and}}\  D v(x)=\frac45 x
\ \ {\rm{for}}\ \ x\in\overline U.
$$
Hence we have that
$$\widetilde{H}(x,Du(x))=0= \widetilde{H}(x,Dv(x)) \ \ {\rm{for}}\ \ x\in\overline U.$$
This yields that both $u$ and $v$ are absolute minimizers, with the same boundary value.
Thus the uniqueness for absolute minimizers fails.
\end{rem}

We would like to make a remark on the relationship between (A3) and the
uniform coercivity of $H$ in $p$-variable.

\begin{rem} {\rm
Assume that $H(x,p)\in\lsc(\overline U\times\rn)$ is convex in $p$ and $H(x,0)=\min_{p\in\rn}H(x,p)=0$.
Then \\
(a) The following statements are equivalent:

(a1) $\bigcup_{x\in\overline U}\{p\in\rn| H(x,p)=0\}$ is bounded.

(a2) For any $\lz>0$, $\exists\ R_\lz>0$ such that  $\{p\in\rn| H(x,p)<\lz\}\subset B(0,R_\lz)$ for all $x\in\overline U$.

(a3) $\displaystyle\lim_{|p|\rightarrow+\infty} H(x,p)=+\infty$ uniformly for $x\in\overline U$.

\noindent (b) If, in addition, $H\in C(\overline U\times \rn)$, then (A3) $\displaystyle\Leftrightarrow$ (a3).\\
\noindent (c) There exists $H\in\lsc(\overline U\times\rn)$, that is convex in $p$ and $H(x,0)=\min_{p\in\rn}H(x,p)=0$,  satisfying (a3) but not (A3).}
\end{rem}

\begin{proof}

\noindent(a) It is obvious that (a3) $\Rightarrow$ (a1).  We argue by contradiction that (a2) $\Rightarrow$ (a3):
Suppose that (a3) were false. Then there exist $N_0>0$, $p_i\in\rn$, with $|p_i|\rightarrow+\infty$, and  $x_i\in\overline U\rightarrow x_\infty\in\overline U$ such that $H(x_i, p_i)< N_0$ for all $i$.
By (a2), there exists $R_{N_0}<+\infty$ such that $|p_i|\le R_{N_0}$ for all $i$. This is impossible. We can also
argue by contradiction that (a3) $\Rightarrow$ (a2): Suppose that (a2) were false. Then $\exists\ \lz_0>0$, $p_i\in \rn$ with
$\lim_{i\rightarrow\fz}|p_i|=+\infty$, and $x_i\in\overline U$ such that
$H(x_i, p_i)<\lz_0$. This clearly contradicts (a3).  To see (a1) $\Rightarrow$ (a3), observe that (a1) yields
that there exists $R_0>0$ such that
$\bigcup_{x\in\overline U}\big\{p: \ H(x,p)=0\big\}\subset B(0, R_0-1).$
Thus  $H(x,p)>0$ for all $x\in\overline U$ and $p\in \partial B(0, R_0)$.
Since $H\in \lsc(\overline U\times \rn)$, there exists $(x_0,p_0)\in\overline U\times\partial B(0,R_0)$ such that
$0<c_0=H(x_0, p_0)\le H(x,p)$ for all $(x,p)\in\overline U\times\partial B(0,R_0).$
This, combined with the convexity of $H(x,\cdot)$ and $H(x,0)=0$ for $x\in\overline U$, implies that
$$H(x,p)\ge \frac{|p|}{R_0} c_0, \ {\rm{for\ all}}\ p\in \rn \ {\rm{with}}\ |p|\ge R_0\ {\rm{and}}\ x\in\overline U.$$
Therefore (a3) holds.

\noindent(b) It is obvious that (A3) implies (a3).  To see that (a3) implies (A3),
set $$E(x,\lz)=\big\{p:\ H(x,p)<\lz\big\} \ {\rm{for}}\ x\in\overline U \ {\rm{and}}\ \lz>0.$$
Since $H\in C(\overline U\times\rn)$ is convex in $p$ and $H(x,0)=0$, we see that
$E(x,\lz)\subset\rn$ is a bounded,  convex open set,  containing $0$, for all $x\in\overline U$.
Then
$$0<r(x,\lz)=\sup\Big\{r>0\ | \ B(0,r)\subset E(x,\lz)\Big\}<+\infty, \ \forall\ x\in\overline U.$$
Similar to the proof of Corollary 1.2 in section 6, we can prove that for any $\lz>0$,
$r(\cdot,\lz)\in \lsc(\overline U)$. Hence $\displaystyle
r_\lz=\min_{x\in\overline U} r(x,\lz)$ exist and $0<r_\lz<+\infty$. It is clear that
$B(0,r_\lz)\subset E(x,\lz)$ for all $x\in\overline U$.
To see that $\displaystyle\lim_{\lz\rightarrow +\infty}r_\lz=+\infty$,
observe that for any $\lz>0$ there exist $x_\lz\in\overline U$ and $p_\lz\in \partial B(0,r_\lz)$ such that
$r_\lz=r(x_\lz,\lz)$ and $H(x_\lz, p_\lz)=\lz.$
If $\displaystyle\lim_{\lz_i\rightarrow\infty}r_{\lz_i}=r_\infty<+\infty$, then,
by assuming $x_{\lz_i}\rightarrow x_\infty\in\overline U$ and $p_{\lz_i}\rightarrow p_\infty\in\partial B(0, r_\fz)$
and applying $H\in C(\overline U\times\rn)$, we would have that
$+\fz=\lim_{\lz_i\rightarrow\fz} H(x_i, p_i)=H(x_\infty, p_\infty)<+\infty.$
This is impossible. Hence $\displaystyle\lim_{\lz\rightarrow +\infty}r_\lz=+\infty$ and (A3) holds.

\noindent(c) Let $a\in\lsc(\overline U)$ be such that
$$a(x)\ge 1, \ \forall\ x\in\overline U; \ a(x)=1, \ \forall \ x\in\partial U; \lim_{x\in U\rightarrow x_0\in\partial U} a(x)
=+\infty,$$
and $H(x,p)=a(x)|p|$ for $x\in\overline U$ and $p\in \rn$. It is readily seen that $H\in \lsc(\overline U\times\rn)$
is convex in $p$ and $0=H(x,0)=\min_{\rn}H(x,p)$ for all $x\in \overline U$, and satisfies (a3).
However, for $\lz>0$, we have that
$$\bigcap_{x\in\overline U}\Big\{p\ | \ H(x,p)<\lz\Big\}=
\bigcap_{x\in\overline U}\Big\{p\ | \ |p|<\frac{\lz}{a(x)}\Big\}=\big\{0\big\},
$$
so that (A3) doesn't hold.
\end{proof}

Corollary 1.2 follows directly from Theorem 1.1,  any $H(x,p)$ given by Corollary 1.2 satisfies  the conditions (A1), (A2), and (A3) of Theorem 1.1, see section 6 for the details.

To prove Theorem 1.1, we will employ some ideas by \cite{acjs} and \cite{as,js,ksz}.
The first crucial step to prove Theorem 1.1 is to establish the convexity criteria for absolute subminimizers,
namely, there is $\delta>0$ such that for any $x\in U$ the map $t\mapsto T^tu(x)$ is convex
on $[0, \delta]$ for any absolute subminimizer $u\in
W^{1,\infty}_{\rm{loc}}(U)\cap C(\overline U)$, here $T^tu(x)$ is the Hamilton-Jacobi flow induced by $H(x,p)$.
Some fundamental properties on absolute subminimizers are stated in the following Theorem.

\begin{thm}\label{t1.x4}
For a bounded domain $U\subset\rn$ and $u\in L^\infty(U)$, assume that
$H$ satisfies  (A1), (A2)$_\weak$, and (A3). If, in addition,
$H(x,p)$ is uniformly super-linear in $p$-variable:
\begin{equation}\label{u-coercive}
\lim_{|p|\rightarrow +\infty}\frac{H(x,p)}{|p|}=+\infty, \ \mbox{uniformly in}\ x\in\overline U,
\end{equation}
 then the following statements are equivalent:\\
(i) $u$ is an absolute subminimizer for $H$ in U.\\
(ii) $u$ satisfies the comparison property with intrinsic cones from above in $U$.\\
(iii) u satisfies the convexity criteria in U.\\
(iv) u satisfies the pointwise convexity criteria in U.
\end{thm}

Definitions of  absolute subminimizers, comparison property with intrinsic cones from above (or CICA),
and (pointwise) convexity criteria will be given in section 2.

A few remarks on Theorem 1.5 are in order:

When $H=H(p)\in C^2(\rn)$,  (i)$\Rightarrow$(ii) is proven in \cite{gwy} and (ii)$\Rightarrow$(iii) is given by \cite{js}.
When $H=H(p)\in \rm{Lip}_\loc(\rn)$,  Theorem 1.5 was proven by \cite{acjs}.
The idea of \cite{acjs} relies heavily on the concrete geometric structure  of  the induced action  function
$\mathcal L_t(x,y)=tL\big(\frac{y-x}t\big)$ (see also Lemma 3.11 below).

When $H(x,p)$ satisfies  (A1), (A2)$_{\weak}$, (A3), and (\ref{u-coercive}),
 (i)$\Rightarrow$(ii) is essentially proved by \cite{cp},
 and we will prove  (ii)$\Rightarrow$(iii)$\Rightarrow$(iv)$\Rightarrow$(i)  in section 5.
To this end, we need to study the geometric structures of $H(x,p)$  and the corresponding action  function
$\cl_t(\cdot,\cdot)$ via
 the family of intrinsic pseudo-distances $\{d_\lz\}_{\lz\ge0}$,
 which may have its own interest, see Theorem \ref{t3.x6} and subsection 3.1.
Such geometric structure plays important roles in the proof of  (ii)$\Rightarrow$(iii)$\Rightarrow$(iv)$\Rightarrow$(i).
Comparing Theorem \ref{t3.x6} for general $H(x,p)$ with  Lemma 3.11 for $H=H(p)$ and Lemma 3.12
 for $ H=\langle A(x)\cdot p, p\rangle$ in section 3.3,
we will see that the geometric structure of the action  function induced by general $H(x,p)$
is  much more complicated.
%This raises several new difficulties that we need to overcome
% during the proof of (ii)$\Rightarrow$(iii)$\Rightarrow$(iv)$\Rightarrow$(i) in Theorem 1.4.

With section 3 at hand, we will establish the localization, semigroup and Lipschitz properties of the
Hamilton-Jacobi flow $T^t u(x)$ under (A1), (A2)$_{\weak}$, (A3), and (\ref{u-coercive}), see section 4
for the details.

In section 5, we prove Theorem 1.5 by establishing (i) $\Rightarrow$ (ii) $\Rightarrow$ (iii) $\Rightarrow$ (iv) $\Rightarrow$ (i). We have to overcome several new difficulties due to the complicate geometric structures of the action  function
$\cl_t(\cdot,\cdot)$ induced by $H(x,p)$.
Especially, in order to establish (ii)$\Rightarrow$(iii), we have to prove the key inequality \eqref{e5.x3}, with $s\ge 0$,
in order to obtain convexity of the map $t\mapsto T^tu$
for $u\in\cica(U)$.  Since $u\in\cica(U)$ doesn't necessarily imply $T^su\in\cica(U)$,
 \eqref{e5.x3} for $s>0$ doesn't follow from  \eqref{e5.x3} for $s=0$.
%A new proof without using $T^su\in\cica(U)$ for \eqref{e5.x3} with $s>0$ is required.
With the help of Theorem \ref{t3.x6} and a careful analysis of the Hamilton-Jacobi flow $T^t u$,
we manage to give a direct proof of \eqref{e5.x3} for $s>0$,
 which seems to be new even in the case  $H=H(p)$,  see section 5.2 for the details.
 The proof of (iii)$\Rightarrow$(iv) relies mainly on Theorem \ref{t3.x6}, see section 5.2.
 To prove (iv)$\Rightarrow$(i), we need to establish Lemma \ref{l5.x5} on slope increasing estimate.
It is worthwhile mentioning that proofs for slope increasing estimate by \cite{acjs,jwy} rely
on the linearity of  action function $\mathcal L_t(x,y)=tL\big(\frac{y-x}{t}\big)$.
Since such a linearity may not be available for the action function $\cl_t(\cdot,\cdot)$ associated with
general $H(x,p)$,  a new argument for the slope increasing estimate based on Theorem
\ref{t3.x6}  was developed  in section 5.3.

The second crucial step to prove Theorem 1.1  is to establish the patching lemma, Lemma 6.2, under  (A1),
(A2), (A3), and (\ref{u-coercive}),  for $H(x,p)$.
To prove Lemma 6.2, we need  to establish the approximation Lemma 6.3,
whose proof relies on  (A2) and a careful analysis based on the lower semicontinuity of $H$. Notice that the example in Remark 1.3 shows that (A2) in Lemma 6.3  can not be relaxed to  (A2)$_\weak$.
 Thanks to Lemma 6.3, Theorem \ref{t3.x6}, and Lemma \ref{l5.x5},
 and the utilization of intrinsic pseudo-distances $\{d_\lz\}_{\lz>0}$, we
 can prove  Lemma 6.2 by modifying  the arguments by \cite{acjs},
 see section 6.2 for the details.

Since Theorem 1.5, Lemma 6.1, and Lemma 6.2
all requires the condition (\ref{u-coercive}),
we can't directly apply Theorem 1.5 and Lemmas 6.1 and 6.2 to prove Theorem 1.1.
However, as shown by Lemma \ref{l6.x2.0}, $H^a$ for any $a>1$ satisfies (A1), (A2), (A3), and (\ref{u-coercive});
while Theorem 1.1 remains to be same if we replace $H$ by $H^a$.
Thus we can apply Theorem 1.5 and Lemmas 6.1, 6.2 to $H^a$ for $a>1$ to establish
Theorem 1.1.

To end this section, we would like to make a few comments on regularity  of absolute minimizers or viscosity
solutions to Aronsson's equation.
When $H=|p|^2$, Savin \cite{s05} and Evans and Savin \cite{es08} have established
$C^1$ and $C^{1,\az}$-interior regularity of infinity harmonic functions
respectively, and Wang and Yu \cite{wy11} obtained the boundary $C^1$-regularity for $n=2$;
while Evans and Smart \cite{es11a,es11b} have proved the interior everywhere differentiability
of infinity harmonic functions, and Wang and Yu \cite{wy11} have shown their boundary differentiability
provided $g\in C^1(\partial U)$ for $n\ge 3$. However,  the $C^1$-regularity for infinity harmonic functions
remains largely open.
When $H(p)\in C^2(\rn)$ is convex, Wang and Yu \cite{wy}
have showed  the $C^1$-regularity of absolute minimizers for $n=2$;
while the corresponding $C^1$-regularity for $n\ge 3$ remains open.
When $H(x,p)=\langle A(x)\cdot p, p\rangle$, with $A\in C^{1,1}(\Omega, \mathbb R^{n\times n})$
symmetric and uniformly elliptic, the interior everywhere differentiability of absolute minimizers was recently obtained
by \cite{swz} for $ n\ge2$; while the corresponding $C^1$-regularity is also open.
We believe that the properties established in the paper, such as Theorem 1.5,
may be useful for the investigation of regularity of the Aronsson equation for general
Hamiltonian functions $H(x,p)$.

\section{Definitions and notions}

In this section, we will assume that $H(x,p)$ satisfies that assumptions (A1), (A2)$_ \weak $  and (A3).
Let $\lip( U)$ denote the space of Lipschitz functions $u:U\to\rr$,
that is,
$$\lip(u,U):=\sup_{x,y\in U,\,x\ne y}\frac{|u(x)-u(y)|}{|x-y|}<\fz.$$
Recall that $u\in \lip( U)$ iff $u\in W^{1,\,\fz}(U)$, that is, $u$ is differentiable almost everywhere in $U$
and its gradient $Du$ is  bounded in $U$.  We say $u\in \lip_\loc(U)$ (or $u\in W^{1,\infty}_\loc(U)$ equivalently)
if $u\in \lip(V)$ (or $u\in W^{1,\infty}(V)$ equivalently) for any open subset $V\Subset U$.

\begin{defn} \label{d2.x1}

(i) A function $u \in\lip_{\loc}( U)$ is called an  absolute  subminimizer
  in $U$ for $H$,
if for each $V\Subset U$,  $v\in\lip(V)\cap C(\overline V)$ satisfies
$v\le u$ in $V$, and $v=u$ on $\partial V$,  then
$$\mathcal F_\infty(u,V)\le \mathcal F_\infty(v,V).$$
(ii) A function $u \in\lip_{\loc}( U)$ is called   an  absolute  superminimizer
  in $U$ for $H$,
if for each $V\Subset U$, $v\in\lip(V)\cap C(\overline V)$ satisfies $v\ge u$ in $V$, and
$v=u$ on $\partial V$, then
$$\mathcal F_\infty(u,V)\le \mathcal F_\infty(v,V).$$
(iii) A function $u \in\lip_{\loc}( U)$ is called  an absolute  minimizer in $ U$ for $H$,
if it is both an absolute subminimizer and an absolute superminimizer in $ U$ for $H$.
\end{defn}

%CP showed that the absolute  (sub-/sup-) minimizer be characterized by comparison property with intrinsic cones (from above/below).
To introduce the property of comparison  with  intrinsic cones (also called as
the comparison property with distance functions by \cite{cp}), we set, for every $\lz\ge 0$,
\begin{equation}\label{e2.x0}
L_\lz (x,q):=\sup_{\{p\in\rn:\ H(x,p)\le\lz\}} p\cdot q, \quad  \forall \ x\in\overline U \ \mbox{and }\  q\in\rn.
\end{equation}
For $0\le a<b\le +\infty$, let  $\gz:[a,\,b]\to \overline U$ be a Lipschitz curve, that is, there exists a constant $C>0$
such that $|\gz(s)-\gz(t)|\le C|s-t|$ whenever $s,t\in[a,b]$.
The $L_\lz$-length of $\gz$ is defined by
\begin{equation}\label{e2.x1}\ell_\lz(\gz):= \int_a^bL_\lz\lf(\gz(\theta),\gz'(\theta)\r)\,d\theta,\end{equation}
which is nonnegative, since $L_\lambda(x,q)\ge 0$ for any $x\in\overline U$ and $q\in\rn$.
For a pair of  points  $x,\,y\in\overline U$, the $d_\lambda$-distance from $x$ to $y$ is defined by
\begin{equation}\label{e2.x2}
d_\lz(x,y):= \inf\Big\{\ell_\lz(\gz)\ |\ \gamma\in\mathcal C(a,b;x,y; \overline U)\Big\},
\end{equation}
where $\mathcal C(a,b;x,y; \overline U)$ denotes the space of all Lipschitz curves $\gz:[a,\,b]\to  \overline U$
that joins $x$ to $y$, that is $\gamma(a)=x$ and $\gamma(b)=y$.
Since $L_\lambda(x,\cdot)$ is positively homogeneous of degree one for $x\in\overline U$,
a simple change of variables shows that $d_\lz(x,y)$ is independent of
the choices $0=a<b\le +\infty$ in $\mathcal C(a, b; x,y; \overline U)$.
It is not hard to verify that $d_\lz$ is a pseudo-distance on $\overline U$:\\
(i) $d_\lz(x,y)\ge0$ for all $x, y\in \overline U$, with the equality iff $x=y$; and \\
(ii) $d_\lz(x,y)\le d_\lz(x,z)+d_\lz(z,y)$ for all $x,y,z\in\overline U$.\\
However, since $H(x,p)$ is not assumed to be an even function in $p$-variable,
$d_\lz(\cdot, \cdot)$ may not be symmetric, in general.

A function is called an intrinsic cone on $\overline U$,
 if it is either $d_\lz(x_0,\cdot)+c$ or $d_\lz(\cdot,x_0)+c$ for some $\lz\ge 0, c\in\rr$
 and $x_0\in U$. Notice that if $U=\rn$ and $H(x,p)=|p|^2$ for $x, p\in\rn$, then
  $d_\lz(x,y)=\sqrt\lz |x-y|$ for all $x,y\in\rn$ becomes
 the standard round cone function, which was  introduced by \cite{ceg} (see also \cite{y07} and \cite{gwy}).

 Denote by $\usc(U)$ (or $\lsc(U)$ respectively)  the space of all upper semicontinuous
 (or lower semicontinuous, respectively) functions in $U$.  Notice that
 $C(U)=\usc(U)\cap\lsc(U)$. We also set $C_b(U)=C(U)\cap L^\infty(U)$.

 We now given the definition of comparison with intrinsic cones.
\begin{defn}\label{d2.x2}
(i) A function $u\in \usc( U)$  enjoys the comparison property with intrinsic cones from above in $U$
for $H$, written as $u\in\cica(U)$,  if for every $\lz\ge0$, $c\in\rr$, $x_0\in  U$ and $V\Subset  U$,
$$\max_{y\in \overline{V\setminus\{x_0\}}}\Big\{u(y) -(d_\lz(x_0,y) +c)\Big\}
=\max_{y\in \partial( V\setminus\{x_0\})}\Big\{u(y) -(d_\lz(x_0,y) +c)\Big\}.$$

\smallskip
\noindent (ii) A function $u\in \lsc( U)$  enjoys the comparison property with intrinsic cones from below
in $ U$ for $H$, written as  $u\in\cicb(U)$, if
for every $\lz\ge0$, $c\in\rr$, $x_0\in  U$ and $V\Subset  U$,
$$\min_{y\in \overline{V\setminus\{x_0\}}}\Big\{d_\lz(y,x_0) +c+u(y)\Big\}
= \min_{y\in \partial( V\setminus\{x_0\})}\Big\{d_\lz(y,x_0) +c+u(y)\Big\}.$$

\smallskip
\noindent (iii) A function $u\in C( U)$  enjoys the comparison property with intrinsic cones  in $ U$
for $H$, if $u\in\cica(U)\cap \cicb(U)$.

\end{defn}

Now, we introduce property of (pointwise) convexity (or concavity)  criteria, which is an extension of
that \cite{acjs} and \cite{js} where $H=H(p)$ was considered. To do it, we need to
recall the Hamilton-Jacobi flow.
Let $L$ be the Lagrangian corresponding to $H$ or, equivalently, the Legendre transform of $H$:
$$ L(x,q):=\sup_{p\in\rn}\Big\{p\cdot q-H(x,p)\Big\}, \ x\in\overline U \ {\rm{and}}\ q\in\rn.
$$
It follows directly from  (A1) and (A2)$_\weak$
that $L$ is upper semicontinuous on $\overline U\times\rn$;
$0=L(x,0)\le L(x,q)$ for all  $x\in\overline U$ and $q\in\rn$;
and $L(x,\cdot)$ is convex on $\rn$ for all $x\in\overline U$. If
$H$ satisfies (\ref{u-coercive}), then it is not hard to show that
$L(x,q)<+\infty$ for all $x\in\overline U$ and $q\in\rn$, and
\begin{equation}\label{u-coercive1}
\lim_{|q|\rightarrow\infty}\frac{L(x,q)}{|q|}=+\infty, \ {\mbox{uniformly in}}\ x\in\overline U.
\end{equation}

The action function $\cl_t(\cdot,\cdot)$, corresponding to $L$, is defined as follows:
for all $t>0$ and $  x,y\in \overline U$,
$$\cl_t(x,y):= \inf\Big\{\int_0^tL\lf(\gz(\theta),\gz'(\theta)\r)\,d\theta\ \big|\ \gamma\in \mathcal C(0,t; x, y; \overline U)\Big\}.$$
It is easy to see that for $t>0$, $\cl_t(x,x)=0$ for all $x\in\overline U$, and $0\le \cl_t(x,y)$ for all $x,y\in \overline U$.
For $t=0$, we set
$$\cl_0(x,y)=\begin{cases} 0 & \ {\rm{if}}\ x=y\in\overline U,\\
+\infty & \ {\rm{if}}\ x\not=y\in\overline U.
\end{cases}
$$
The action function $\cl_t(\cdot,\cdot)$ induces two Hamilton-Jacobi flows.
For every $u\in L^\fz(U)$, $x\in \overline U$ and $t\ge 0$, we define
$$T^tu (x):=\sup_{y\in  \overline U}\Big\{u(y)-\cl_t(x,y)\Big\} \ {\rm
and}\
 T_tu (x):=\inf_{y\in \overline U}\Big\{u(y)+\cl_t(y,x)\Big\}.$$
 %which guarantees the maps $t\mapsto T^t u(x)$ and $t\mapsto T_t u(x)$ continuous from the right at $t=0.$(see remark)
 It is clear that $T^0u(x) =T_0u(x) =u(x)$ for all $x\in \overline U$, and
\begin{eqnarray}\label{e2.x3}
\inf_{y\in \overline U}u(y)\leq T_t u(x)\leq u(x)\leq T^t u(x)\leq\sup_{y\in \overline U}u(y),\ \forall\;x\in \overline U
\ {\rm{and}}\ t\ge0.
\end{eqnarray}
%We also observe that $T_t$ and $T^t$ preserve the order and commute with constants, that is, if  $u\leq v$ in $U$, then for all $t>0$
%we have $
% T^t u \leq T^t v$ and $ T_t u \leq T_t v$ in $U$,
% and  for all $t>0$ and   $c\in\rr$, we have
% $
%T^t(u+c)=T^t u+c$ and $T_t(u+c)=T_t u+c$.

Employing Hamilton-Jacobi flows,  we can define the  convexity (or concavity)  criteria
and the pointwise convexity (or concavity)  criteria.

\begin{defn}\label{d2.x4}
(i) A function $u\in C_b(U)$  enjoys the convexity   criteria in $U$  for $H$,
 if for every $V \Subset U,$ there exists $\delta=\dz(V)>0$ such that
 the map $t\mapsto T^t u(x)$ is convex  on $ t\in[0,\delta(V)]$ for all $x\in V$.

\smallskip
\noindent (ii) A  function $u \in C_b(U)$  enjoys the concavity  criterion in $U$ for $H$,
 if for every $V \Subset U,$ there exists $\delta=\dz(V)>0$ such that
 the map $t\mapsto  T_t u(x)$ is concave  on $ t\in[0,\delta(V)]$ for all $x\in V$.
 \end{defn}

For every  $u\in L^\fz(U)$  and $x\in U$,
we set
\begin{equation}\label{e2.x6} S^+u(x) := \limsup_{s\to 0}\frac{T^su(x)-u(x)}{s}\quad {\rm and}\quad
S^-u(x) := \liminf_{s\to 0}\frac{u(x)-T_su(x)}{s}.
\end{equation}
It is clear that both $S^{-}u(x)$ and $S^+u(x)$ are non-negative for all $x\in U$.
 \begin{defn}\label{d2.x5}
 (i)  A function $u \in C_b(U)\cap\lip_\loc(U)$ enjoys the pointwise convexity   criteria in $U$ for $H$
 if  $S^+u\in\usc(U) $,   and   for each $x \in U$,  there exists $ \dz(x)>0$ such that
 the map $t\mapsto T^tu(x)$ is convex on $[0,\dz(x)]$.

\smallskip
\noindent (ii)  A function $u \in C_b(U)\cap\lip_\loc(U)$ enjoys the pointwise concavity   criteria in $U$ for $H$
 if   $S^-u\in\lsc(U) $,   and   for each $x \in U$,  there exists $ \dz(x)>0$ such that
 the map $t\mapsto T_tu(x)$ is concave on $[0,\dz(x)]$.
\end{defn}

As already pointed by \cite{acjs} for the case $H=H(p)$, the condition that $u \in \lip_\loc(U)$ and $S^+u\in\usc(U) $
is necessary to characterize the absolute subminimizers for $H=H(x,p)$.

To end this section,  we would like to point out that for any $H$ satisfying (A1), (A2) (or (A2)$_{\rm{weak}}$),
and (A3), if we define a Hamiltonian $\widehat H$ by
$$\widehat{H}(x,p):=H(x,-p), \ \mbox{for}\ x\in\overline U,\ p\in\rn,$$
then $\widehat{H}$ also satisfies A1), (A2) (or (A2)$_{\rm{weak}}$),
and (A3).  Furthermore, we have the following.
\begin{rem}\label{r2.x3}{\rm
(i)  $u\in\usc(U)$ is an absolute superminimizer for $H$ in $U$
 iff $v:=-u\in\lsc(U)$ is an absolute subminimizer for $\widehat{H}$ in $U.$\\
(ii) For all $x\in\overline U$,
$\lz\ge0$ and $q\in\rn$, it holds that
$$\widehat {L}_\lz(x,q)\big(\displaystyle:=\sup_{\{\widehat{H}(x,p)\le\lambda\}}p\cdot q\big)=L_\lz(x,-q),$$
and
$$\widehat\ell_\lz(\widehat\gz) \big(:=\int_a^b \widehat{L}_\lz(\widehat\gamma(t),\widehat\gamma'(t))\,dt\big)
= \ell_\lz( \gz) $$
where
$\widehat\gz(t)=\gz(a+b-t)$ for all $t\in[a,b]$ and $\gz: [a,b]\to \overline U$ is Lipschitz.
Let $\widehat {d}_\lz$ denote the pseudo-distance of
$\widehat H$ (defined by \eqref{e2.x2} with $H$ replaced by $\widehat H$).
Then direct calculations lead
$$\widehat {d}_\lz(x,y)=d_\lz(y,x), \ \forall\ \lz\ge0, \ x, y\in\overline U.
$$
Hence $u\in\cicb(U)$ for $ H$ iff $-u\in\cica(U)$ for $\widehat H$.
\\(iii)  Denote by
$\widehat{\cl}_t$, $\widehat {T}^t$ and $\widehat {T}_t$
the action function and the Hamilton-Jacobi flows  associated to $\widehat H$.
Then   it holds
$$\widehat{\cl}_t(x,y)=\cl_t(y,x),  \ \forall\ t\ge 0, \ x,y\in\overline U.$$
Hence we have that for any $u\in L^\fz(U)$,
$$\widehat T^t(-u)(x)=-T_tu(x) \quad \mbox{and}\quad  \widehat T_t(-u)(x)=-T^tu(x),
\ \forall \ t\ge 0, \ x\in\overline U.$$
Therefore, $u$ satisfies the concavity criteria for $H$ iff $-u$ satisfies
the convexity criteria for $\widehat H$.
\\
(iv) Similar to \eqref{e2.x6}, we can define $\widehat S^+u(x)$ (or $\widehat S^-u(x)$)
with $T^ s$ (or $T_s$) replaced by $\widehat T^s$ (or $\widehat T_s$).
Then
$$\widehat S^+(-u)(x)=-S^-u(x)\quad \mbox{and}\quad \widehat S^-(-u)(x)= -S^+u(x),\ \forall x\in U.$$
In particular, $S^-u\in\lsc(U)$  iff $\hat S^+(-u)\in\usc(U)$.
Hence  $u$ satisfies the pointwise concavity criteria for $H$
iff $-u$ satisfies  the pointwise convexity criteria for $\hat H$.}
\end{rem}
By Remark \ref{r2.x3},  we only need to work with absolute subminimizers,
$\cica(U)$,  the convexity criteria, and the pointwise convexity criteria for $H$.

\section{Structure of action function by pseudo-distance%$\cl_t$ and $d_\lz$
}

In this section, we will assume that $H(x,p)$ satisfies that  (A1), (A2)$_ \weak $, (A3), and (\ref{u-coercive}).
The main aim is to establish some geometric structure of $\cl_t(\cdot,\cdot)$ via the pseudo-distance $d_\lz$,
namely Theorem \ref{t3.x6}, which will play a crucial role in the proof of Theorems 1.5 and 1.1.
Theorem 3.6 may have its own interest.
In subsection 3.3, we also examize the geometric structure of $\cl_t(\cdot,\cdot)$ for two special
types of $H(x,p)$: i) $H(x,p)=H(p)$; and ii) $H(x,p)=\langle A(x)p, p\rangle$ for a uniformly elliptic
symmetric matrix-valued function $A:\Omega\to \mathbb R^{n\times n}$.

\subsection{Elementary properties of $\cl_t(\cdot,\cdot)$}
For all $x,y\in \overline U$, define the euclidean distance $d_U$, subject to $U$,
between $x$ and $y$ by
$$d_U(x,y):=\inf\Big\{\int_0^1 |\gz'(\theta)|\,d\theta\ \big|\ \gamma\in\mathcal C\big(0, 1; x, y; \overline U\big)\Big\}.$$
When $U$ is convex, $d_U(x,y)=|x-y|$ for all $x,y\in\overline U$;  when $U$ is non-convex, if
 the line segment joining $[x, y]\subset U$, then $d_U(x,y)=|x-y|$; and
$d_U(x,y)\ge |x-y|$ for general $x,y\in\overline U$.
For $x\in U$ and $\lz, r>0$, define $d_U$ and $d_\lz$ balls
$$B_{d_U}(x,r)=\Big\{y\in \overline U\ |\ d_U(x,y)<r\Big\}, \
B_{d_\lz}(x,r)=\Big\{y\in \overline U\ |\  d_\lz(x,y)<r\Big\},$$
and the euclidean ball $B(x,r)=\big\{y\in\rn: \ |y-x|<r\big\}$.

\begin{lem}\label{l2.w16} For all $\lz>0 $, the following holds:
\begin{eqnarray}\label{e3.x1}
 r_{\lambda}d_U(x,y)\leq d_{\lambda}(x,y)\leq   R_{\lambda}d_U(x,y), \ \forall\ x,y\in \overline U.
\end{eqnarray}
In particular, for every $\lz>0$ and $x\in U$, if $0<r\le \dist(x,\partial U)$, it holds that
\begin{equation}\label{e3.x1.0}
B_{d_U}(x,\frac{r}{R_\lz})\subset B_{d_\lz}(x,r)\subset B_{d_U}(x,\frac{r}{r_\lz}).
\end{equation}
\end{lem}
\begin{proof}
It follows from (A3) that for all $x\in\overline U$ and $p\in\rn$, we have that
$$r_{\lambda}|q|=\sup_{p\in B(0,r_{\lambda})}p\cdot q\le L_{\lambda}(x,q) =\sup_{\{H(x,p)\leq\lambda\}}p\cdot q
\le\sup_{p\in B(0,R_{\lambda})}p\cdot q=R_\lz |q|.
 $$
Thus, for every $\gz\in \mathcal C\big(0,1; x, y; \overline U\big)$, we have
that
$$r_\lz \int_0^1|\gamma'(t)|\,dt\le \int_0^1 L_\lz(\gamma(t), \gamma'(t))\,dt
\le R_\lz \int_0^1|\gamma'(t)|\,dt.$$
Taking infimum over all $\gz\in \mathcal C\big(0,1; x, y; \overline U\big)$, this yields \eqref{e3.x1}.
\end{proof}

 \begin{lem}\label{l3.x2}
 (i) For all $x\in\overline U$ and $q\in\rn$, we have that
 \begin{equation}\label{e3.x2.1}
 L(x,q )=\sup_{\lz\ge0} \Big\{L_\lz(x,q)-\lz\Big\}.
 \end{equation}
 (ii) For every   $ t>0$ and $x,y\in \overline U$, we have that
  \begin{equation}\label{e3.x2.2}
  \cl_t(x ,y)\ge \sup_{\lz\ge0}\Big\{d_\lz(x ,y)-\lz t\Big\}.
  \end{equation}
 \end{lem}

 \begin{proof}
 (i) follows directly from  the definitions of $L$ and $L_\lz$:
 \begin{eqnarray*}
L(x,q )&&=\sup_{p\in\rn}\Big\{p\cdot  q  -H(x,p)\Big\}
=\sup_{\lz\ge 0} \sup_{\{H(x,p)\le\lz\}}\Big\{p\cdot q-\lz\Big\}
=\sup_{\lz\ge0} \Big\{L_\lz(x,q)-\lz\Big\}.
\end{eqnarray*}
(ii) From (i), for every $\lz\ge 0$ and $\gz\in\mathcal C\big(0,t; x, y; \overline U\big)$
it holds
\begin{eqnarray*}
\int_0^t    L(\gz(\theta),\gz'(\theta) )
\,d\theta
 \geq  \int_0^t  \big(L_\lz(\gz(\theta), \gz'(\theta)) -\lz\big)\,d\theta
  \ge  d_\lz(x,y)-\lambda t.
\end{eqnarray*}
Taking infimum over $\gz\in \mathcal C\big(0, t; x, y; \overline U\big)$ and supermum
over $\lz\ge 0$,
this yields (\ref{e3.x2.2}).
 \end{proof}
Define  $M: \mathbb R_+\to \mathbb R_+$ by
\begin{equation}\label{e3.x3.0}
M(t)=\begin{cases} \displaystyle\sup_{\lz\ge0}\lf(r_\lz -\frac{\lz}{t}\r)  & t>0,\\
0 & t=0.
\end{cases}
\end{equation}
Then we have
\begin{lem}\label{l3.x3}
The function $M:\mathbb R_+\to\mathbb R_+$ is monotone increasing, and $\displaystyle
\lim_{t\rightarrow\infty}M(t)=\fz$.
There holds that
\begin{equation}\label{e3.x3.1}
L(x,q )\ge M(|q| ) |q|,  \ \forall\ x\in\overline U\  \mbox{and}\ q\in\rn,
\end{equation}
and hence
\begin{equation}\label{e3.x3.2}
 \cl_{t}(x,y) \geq M\left(\frac{d_U(x,y)}{t}\right)  {d_U(x,y)},
 \ \forall\ t>0, \ {\rm{and}}\ x,y \in\overline  U.
\end{equation}
\end{lem}
\begin{proof} It is easy to see that $M(t_1)\le M(t_2)$ when $0\le t_1\le t_2$. Moreover,
$$\displaystyle M(t)\ge r_t-\frac{t}{t}=r_t-1\rightarrow +\infty\  {\rm{as}} \ t\rightarrow\infty.$$
For $x\in\overline U$ and $q\in\rn\setminus\{0\}$, it follows from (A3) and (\ref{u-coercive}) that
\begin{eqnarray*}
+\infty>L(x,q )
 =\sup_{\lz\ge0}  \big(L_\lz(x,q)   -\lz \big)
 \ge\sup_{\lz\ge0}\big(r_\lz|q|-\lz\big)=|q|\sup_{\lz\ge0}\Big(r_\lz-\frac{\lz}{|q|}\Big)=M(|q|)|q|.
\end{eqnarray*}
In particular, $M(|q|)<+\infty$ for $q\in\rn$.
%Since $L(z,q)<\fz$, we  have $M(t)<\fz$ for all $t>0$ by taking $q$ such that $|q|=t$.
%Thanks to $r_\lz\to\fz$ as $\lz\to\fz$, we obtain $M(t)\to\fz$ as $t\to\infty$.
For $t >0$ and $ x,y\in \overline U$, it follows from (\ref{e3.x2.2}) that
\begin{eqnarray*}
 \cl_t(x,y)
 \geq  \sup_{\lz\ge0}  \big(d_\lz (x,y) -\lz t\big)
 \ge \sup_{\lz\ge0} \big(r_\lz d_U(x,y)-\lambda t\big)=M\lf(\frac{d_U(x,y)}{t}\r)d_U(x,y).
\end{eqnarray*}
This implies (\ref{e3.x3.2}).
\end{proof}

Without loss of generality, we may assume that both $R_\lz$ and $r_\lz$
are increasing functions of $\lz\in\mathbb R_+$. From Lemma \ref{l3.x3},  for every $\lz>0$
there exists $a_\lz>0$ such that $M(a_\lz)\ge  R_\lz$,  and we may assume that
$a_\lz$ is an increasing function of $\lz\in\mathbb R_+$.

For $x\in\overline U$ and $p\in \rn$,
denote the subdifferential of $H(x,\cdot)$ at $p$ by
$$\partial_pH(x,p)=\Big\{q\in\rn\ \big|  \ H(x, p')\ge H(x, p)+q\cdot(p'-p), \ \forall\ p'\in\rn\Big\}.$$

Now we have
\begin{cor}\label{c3.x4} For any $\lz>0$ and $x\in\overline U$,
if $H(x,p)=\lz>0$ and $q\in\partial_p H(x,p)$,
then $$r_\lz\le |p|\le R_\lz,\quad  p\cdot q\ge\lz\quad \mbox{and}\quad   \frac{\lz}{R_\lz}\le  |q| \le  a_\lz.$$

\end{cor}

\begin{proof}
Since $H(x,p)=\lz$, it follows from (A3) that
$$r_\lz\le |p|\le R_\lz.$$
Since $q\in\partial_pH(x,p)$, it follows that
\begin{equation}\label{e3.x4.1}
\lz\le \lz+L(x,q)=H(x,p)+L(x,q)=p\cdot q.
\end{equation}

Thus we have that
$|q|\ge \frac{\lz}{R_\lz}$.
By (\ref{e3.x4.1}), we have that
$$M(|q|)|q|\le L(x,q)\le p\cdot q\le |p||q|\le R_\lambda |q|.$$
This implies that $M(|q|)\le R_\lambda$ and hence $|q|\le a_\lambda$.
\end{proof}

\begin{lem}\label{l3.x5}

(i) For each $t>0$, $\cl_t(\cdot,\cdot)$ is locally bounded on $U\times U$.\\
(ii) For each pair of points $z,x\in U$,  the map $t\mapsto\cl_t(z,x)$ is monotone
decreasing on $\mathbb R_+$  and $\displaystyle\lim_{s\uparrow t^-}\cl_s(z,x)=\cl_t(z,x)$ for all $t>0$.\\
(iii) For each   $z \in U$ and $t>0$, the map $x\mapsto\cl_t(z,x)$ is upper semicontinuous on $\overline U$.
\end{lem}

\begin{proof} We divide the proof into several steps.
\\
\noindent{\it Step 1}.   First we want to show that for each pair of points $z,x\in U$,
the map $t\mapsto\cl_t(z,x)$ is monotone decreasing on $\mathbb R_+$. In fact,
for every $s<t$ and $\gz\in \mathcal C\big(0,s; z,x; \overline U\big)$,
set $\wz \gz(\theta)=\gz(\frac{\theta s}{t})$ for $\theta\in[0,t]$. Then
$\wz\gz\in \mathcal C\big(0,t; z,x; \overline U\big)$. Since
$L(x,\cdot)$ is convex, we obtain that
\begin{eqnarray*}
\int_0^sL(\gz(\theta),\gz'(\theta))\,d\theta&&
= \int_0^tL\big(\wz\gz(\theta  ),\frac{t}{s}\wz \gz'(\theta)\big)\frac{s}{t}\,d\theta\\
&&\ge \int_0^tL(\wz\gz(\theta  ),\wz \gz'(\theta  ) )  \,d\theta\ge \cl_t(z,x),
\end{eqnarray*}
this implies that  $\cl_s(z,x)\ge  \cl_t(z,x)$ by taking infimum over all
$\gz\in\mathcal C\big(0,s; z,x; \overline U\big)$.

{\noindent\it Step 2.}  For each $t>0, b>0$, and $z,x\in \overline U$ with $c:=d_U(z,x)\le b$,
there exists  $\gz\in\mathcal C\big (0,c; z, x; \overline U\big)$ such that
$|\gz'(s)|=1$ for a.e. $s\in [0, c]$.
Set $\wz \gz(s)=\gz\big(\frac{cs}t\big)$ for $s\in[0,\,t]$.
Then $\wz\gz\in \mathcal C\big(0, t; z, x; \overline U\big)$,
and $|\wz \gz'(s)|=\frac{c}t$ for a.e. $s\in[0,\,t]$.
Therefore, we obtain that
$$\cl_t(z,x)\le\int_0^t L\big(\wz\gz(s),\wz\gz'(s)\big)\,ds\le
t C(b,t),$$
where $$C(b,t):=\sup_{y\in \overline U}\sup_{|q|\le b}L\big(y,\frac{q}t\big)<+\infty.$$
This shows that
$\cl_t(\cdot,\cdot)$ is bounded on
$\big\{(z,x)\in \overline U\times\overline U\ | \ d_U(z,x)\le b\big\}$ and hence (i) holds.
\\
{\noindent\it Step 3.} For each pair of $z,x\in U$ and $t>0$,  since $\cl_s(z,x)$ is monotone decreasing
w.r.t. $s>0$,  for (ii) it suffices to show that for any $\epsilon>0$ there exists $\delta=\delta(\epsilon)>0$
such that
$$\cl_s(z,x)\le \cl_t(z,x)+ \ez,\ \forall \ s\in (t-\dz, t).$$
Let $\gz\in \mathcal C\big(0,t; z,x; \overline U\big)$ be such that
 $$\cl_t(z,x)\ge \int_0^t L(\gz (\theta),\gz '(\theta))\,d\theta-\frac{\ez}2.$$
 It is not hard to see that
 $$\cl_s(z,x)\le \int_0^t L\big(\gz (\theta),\frac{t}{s}\gz '(\theta)\big)\frac{s}{t}\,d\theta.$$
 Since   $\gz '\in L^\infty([0,t])$ and $L(w,q)$ is upper semicontinuous in $\overline U\times\rn$,
 we can find $\dz(w,q)>0$ such that
 $$L\big(\wz w,\frac{\wz qt}{s}\big)\frac{s}{t}\le L( w,q) +\frac{\ez}{8t}$$
whenever $s\in (t-\delta(w,q), t)$,
$|\wz w-w|\le\dz(w,q)$, and $\displaystyle|\wz q-q|\le \frac{2\dz(w,q)}{t} \|\gz'\|_{L^\infty([0,t])}$.

Since $\gz([0,1])\times \gamma'([0,t])\subset\overline U\times\rn$
is compact, by a simple covering argument we can find a $\dz>0$
such that for any $s\in (t-\delta, t)$, it holds that
$$L\big(\gz (\theta),\frac{t}{s}\gz_t'(\theta)\big)\frac{s}{t}\le L(\gz(\theta),\gz '(\theta) )
+\frac{\ez}{8t}, \ {\rm{a.e.}}\ \theta\in[0,\,t].$$
This implies that $\cl_s(z,x)\le \cl_t(z,x)+ \ez$ whenever $s\in (t-\delta, t)$. Hence (ii) holds.\\
{\noindent\it Step 4.} For $t>0$ and $z\in U$, we argue by
contradiction that $\cl_t(z,\cdot)$ is upper semicontinuous in $\overline U$:
For, otherwise, there exist $\ez_0>0$ and $\{x_i\}, \{x_0\}\subset\overline U$,
with $x_i\rightarrow x_0$, such that
$$\cl_t(z,x_i)\ge \cl_t(z,x_0)+\ez_0.$$
Since there exists $\delta_0>0$
such that  $\cl_t(z,x_0)+\ez_0\ge \cl_{t-\dz_0}(z,x_0)-\frac{\ez_0}{2}$,
we have that
 $$\cl_t(z,x_i)\ge\cl_{t-\dz_0} (z,x_0)+\frac{\ez_0}2.$$
Since $\cl_{t }(z,x_i)\le \cl_{t-\dz_0}(z,x_0)+\cl_{\dz_0}(x_0,x_i)$,
and
$$\cl_{\dz_0}(x_0,x_i)\le \delta_0\sup_{z\in\overline U}
\sup_{|q|\le d_U(x_i,x_0)}L(z,\frac{q}{\dz_0})\le
\frac{\ez_0}4$$
provided  $i$ is sufficiently large. Thus we obtain that
 $$\cl_{t-\dz_0}(z,x_0)\le\cl_{t-\dz_0} (z,x_0)-\frac{\ez_0}4,$$
 this is impossible.
\end{proof}

\subsection{A geometric structure theorem}

For a family of sets $\displaystyle\{F^t \}_{t\ge 0}$, we set, for $t>0$,
$$F^{< t }:=\bigcup_{0\le s< t} F^s; \ F^{\le t }:=\bigcup_{0\le s\le t} F^s;
\ F^{>t}:=\bigcup_{s>t} F^s; \ {\rm{and}}\ F^{\ge t }:=\bigcup_{s\ge t} F^s.$$
%Below, we always  set $E^0_\lz(x)=\{x\}$ and $E^\fz_\lz(x)=\{y\in\overline U: d_U(x,y)=\fz\}$ for all $x\in U$ and $\lz>0$.
Our main theorem on the geometric  structure of $\cl_t(\cdot,\cdot)$ is the following.
\begin{thm}\label{t3.x6}
For each $\lz>0$ and $x\in U$, there exists a family  $\big\{E^t_\lz(x) \big\}_{t\ge 0}$ of subsets of $\overline U$,
with $E_\lambda^0(x)=\{x\}$ and $E_\lambda^\infty(x)=\big\{y\in\overline U:\ d_U(x,y)=\infty\big\}$,
satisfying the following properties:\\
(i) For each $t\in[0,\,\fz)$ and $y\in E^t_\lz(x)$, it holds that
\begin{equation}
\label{e3.x2.3}\cl_t(x,y)=d_\lz(x,y)-\lz t.
\end{equation}
(ii) $U\subset E^{<\fz}_\lz(x)$, and for each $t\in[0,\,\fz)$, we have that
\begin{equation}\label{e3.x2.4}
B_{d_\lz}(x,\lz t)\subset E^{<t}_\lz(x)\subset \overline{B_{d_\lz}(x, a_\lz R_\lz t)}.
\end{equation}
(iii) For each $t\in (0,\,\fz)$,  $E^{t }_\lz(x)$ is closed, $E^{<t}_\lz(x)$ is  relatively open in $\overline U$,  and the relative boundary of $  E^{<t}_\lz(x)$ is contained in $E^{t }_\lz(x)$.\\
(iv) For all $s,t\in[0,\,\fz)$   and  every $z\in E_\lz^{t+s}(x)$,  there exists a point $y\in E_\lz^{t }(x)$
such that $z\in E_\lz^s(y)$, and
$$d_\lz(x,z)=d_\lz(x,y)+d_\lz(y,z).$$
\end{thm}

\bigskip
To prove Theorem \ref{t3.x6},  we first need to establish Lemma \ref{l3.x7} below.
Before stating it, we first point out that $(\overline U, d_\lambda)$ enjoys the length space property:\\
{\it For any $\lambda>0$ and $x,y\in\overline U$ with $d_\lz(x,y)<\fz$,
there exists at least one $\gamma\in \mathcal C(a, b; x, y; \overline U)$ such that
\begin{equation}\label{e3.length}
d_\lambda(x,y)=\int_a^b L_\lambda(\gamma(t), \gamma'(t))\,dt.
\end{equation}}
In fact, let $\big\{\gamma_i\big\}\subset \mathcal C(a, b; x, y; \overline U)$ be a minimizing sequence
for $d_\lambda(x,y)$, i.e.,
$$\lim_{i\rightarrow\infty}\ell_\lz(\gamma_i)=d_\lz(x,y).$$
After a possible change of parametrization, we can always assume that, for all $i\ge 1$,
\begin{equation} \label{e3.c.speed}
L_\lambda(\gamma_i(t),\gamma_i'(t)) =\frac{\ell_\lz(\gz_i)}{b-a}, \ {\rm{a.e.}}\ t\in [a,b].
\end{equation}
Combining (\ref{e3.c.speed}) with Lemma \ref{l2.w16} yields that, for all $i\ge 1$,
\begin{equation}\label{e3.lip}
|\gamma_i'(t)|\le \frac{\ell_\lz(\gz_i)}{(b-a)r_\lz}\Big(\le \frac{1+d_\lz(x,y)}{(b-a)r_\lz}\Big), \ {\rm{a.e.}}\ t\in [a,b].
\end{equation}
Therefore we may assume that there exists $\gamma\in \mathcal C(a, b; x, y; \overline U)$
such that after passing to a subsequence, $\gamma_i\rightarrow \gamma$ in $C([a,b])$ and
$\gamma_i'\rightarrow \gamma'$ weak$^*$ in $L^\fz([a,b])$. It follows directly from
the lower semicontinuity that $d_\lz(x,y)=\ell_\lz(\gamma)$.

With this length space property of $d_\lz$, we  can introduce
the notion of $d_\lz$-length minimizing geodesic rays.
A Lipschitz curve $\gz:[0,\,b]\to \overline U $, $b\in(0,\fz]$,
is a $d_\lz$-length minimizing geodesic curve, if
\begin{equation}\label{e3.geodesic}
d_\lz(\gz(s),\gz(t))=\int_s^t L_\lz(\gamma(\theta), \gamma'(\theta))\,d\theta=t-s, \ \forall\ 0\le s<t\le b.
\end{equation}
In particular, if $\gz:[0,\,b]\to \overline U $ is a $d_\lz$-length minimizing geodesic curve
then
$$\ell_\lz(\gz)=\begin{cases}b &\ {\rm{if}}\ b<+\infty\\
+\infty & \ {\rm{if}}\ b=+\infty.
\end{cases}
$$
%Below for simple of the notation,
%we always write $\gz:[0,\,\ell_\lz(\gz)]\to \overline U $  whenever $\gz$ is a $d_\lz$-length minimizing geodesic curve.

Given a Lipschitz curve $\gz:[0,\,b)\to \overline U $, $b\in[0,\fz)$,
if the restriction of $\gz$ on  each subinterval $[0,b']\subset [0,b)$ is $d_\lz$-length minimizing,
then $\gz(b)=\displaystyle\lim_{\ez\downarrow 0 }\gz(b-\ez)\in\overline U$ exists and
 $\gz$ can be extended to a $d_\lz$-length minimizing geodesic curve $\gz:[0,\,b]\to \overline U $.
Since $(U, d_\lz)$ enjoys the length space property,  the union of the images of all $d_\lz$-length minimizing geodesic curve with $\gz(0)=x$ equals to $U$.

Given $x\in U$ and $\lz>0$, we call a $d_\lz$-length minimizing geodesic curve
$\gz:[0,\,\ell_\lz(\gz)]\to \overline U$, with $\gz(0)=x$,
as a $d_\lz$-length minimizing geodesic ray starting at $x$, if either $\ell_\lz(\gz)=\fz$, or $\ell_\lz(\gz)<\fz$
and $\gz$ cannot be extended to a $d_\lz$-length minimizing geodesic curve
  $\wz \gz:[0,\,\ell_\lz(\gz)+\ez)\to \overline U$  for any $\ez>0$ (i.e., $\wz \gz(s)=\gz(s)$ for $s\in [0,\ell_\lz(\gz)])$.
We denote by $\Gamma_\lz(x )$  the collection of all $d_\lz$-length minimizing geodesic rays $\gz$  starting at $x$.
Since each $d_\lz$-length minimizing geodesic curve starting at $x$ is
contained in a $d_\lz$-length minimizing geodesic ray $\gz$  starting at $x$,
 the union of the images of all curves in $\Gamma_\lz(x)$  equals to $U$.

\begin{lem}\label{l3.x7} For $\lz>0$, $x\in U$, and $\gz\in\Gamma_\lz(x)$, there exist
 $C_\gz:[0,\,\ell_\lz(\gz)]\to \mathbb R_+$ and $p_{\gz}:[0,\,\ell_\lz(\gz)]\to\rn$ such that for a.e.
 $\theta\in [0,\,\ell_\lz(\gz)] $,
 it holds that\\
(i) $H(\gz(\theta),p_\gz(\theta))=\lz$ and   $C_\gz(\theta)\gz'(\theta)\in \partial_pH(\gz(\theta),p_\gz(\theta))$.
\\
(ii) $r_\lz\le |p_{\gz}(\theta)|\le R_\lz$ and $\lz\le C_\gz(\theta)\le  R_\lz a_\lz$.
\end{lem}

In order to prove Lemma \ref{l3.x7}, we recall Lemma \ref{l3.x8} and Lemma \ref{l3.x9}, whose proofs can
be found in \cite{c01} and \cite{cp} respectively.
\begin{lem}\label{l3.x8}
For every $u\in \lip(U)$ and Lipschitz curve $\gz:[a,b]\to U$, there exists a function $g :[a,b]\to\rn$
such that
$$\frac d{dt}u( \gz(t))=g (t)\cdot \gz'(t),$$
whenever $\gz$ is differentiable at $t$, and
$$g(t)\in\bigcap_{r>0}\mbox{closed convex hull of}\ Du(B(\gz(t),r)\setminus \mathcal N(u)),$$
where $\mathcal N(u)$ denotes the set of non-differentiable points of $u$.
\end{lem}

\begin{lem}\label{l3.x9}
(i) For $u\in \lip_\loc(U)$ and $\lz\ge \|H(\cdot,D u)\|_{L^\fz(U)}$, it holds that
\begin{equation}\label{e3.x3} u(y)-u(x)\leq d_{\lambda}(x,y),\ \forall\ x,y\in U.
\end{equation}
(ii) If $u:U\to\rr$ satisfies  \eqref{e3.x3},
then $u\in \lip (U)$ and $\|H(\cdot,D u)\|_{L^\fz(U)}\le\lz$.
In particular, for any $x\in U$ it holds that
\begin{equation}\label{e3.x3.3}
\big\|H(\cdot, Dd_\lz(x,\cdot)\big\|_{L^\infty(U)}\le\lambda.
\end{equation}
\end{lem}

\begin{rem}\label{r3.x10}\rm
For $u\in \lip_\loc(U)$ and $V\Subset U$, let $\lz\ge \|H(\cdot,D u)\|_{L^\fz(V)}$.
Then by the same argument as in Lemma \ref{l3.x9} (i), we also have that
$$u(y)-u(x)\leq d_{\lambda}(x,y) $$
for every $x,y\in\overline V$ whenever
there exists a $d_\lz$-length minimizing geodesic curve $\gz$ in $\overline V$
joining $x$ to $y$ such that $\gz\setminus\{x,y\}\subset V$.
\end{rem}

\begin{proof}[Proof of Lemma \ref{l3.x7}] For $\gz\in\Gamma_\lz(x)$ and
$0\le\theta_1<\theta_2\le \ell_\lz(\gz)$, it holds that
$$d_\lz(\gz(\theta_1),\gz(\theta_2))=\theta_2-\theta_1.$$
\noindent {\it Case 1}: {\it $d_\lz(x,\cdot)$ is differentiable
for all $z\in U\setminus\{x\}$}. By Lemma \ref{l3.x9} (ii) and $H\in \lsc(\overline U\times \rn)$,
we have that $H(z,D_zd_\lz(x,z))\le\lz$ for all $z\in U\setminus\{x\}$. Thus, for any $0\le \theta\le \ell_\lz(\gz)$,
it holds that
$$d_\lz(x,\gz(\theta))=\int_0^\theta D_zd_\lz(x,\gz(s))\cdot \gz'(s)\,ds\le
 \int_0^\theta L_\lz(\gz(s),\gz'(s))\,ds=d_\lz(x,\gz(\theta)).$$
This implies that
$$D_zd_\lz(x,\gz(\theta))\cdot \gz'(\theta)=L_\lz(\gz(\theta),\gz'(\theta))= 1\ \ {\rm{for\ a.e.}} \ \theta\in
[0,\ell_\lz(\gz)].
$$
Set $p_\gz(\theta)= D_zd_\lz(x,\gz(\theta))$ for $\theta\in [0, \ell_\lz(\gz)]$. By the convexity of
$\{p:\ H(\gz(\theta),p)\le\lz\}$, we have that $H(\gz(\theta),p_\gz(\theta))=\lz$ and
there exists $C_\gz(\theta)>0$ such that
$C_\gz(\theta)\gz'(\theta)\in \partial_pH(\gz(\theta),p_\gz(\theta))$ for all $\theta\in [0,\ell_\lz(\gz)]$.
%Notice that $Dd_\lz(x,\cdot)$  is locally uniform bound from above and also away from $0$.
By (A3), we have that  $r_\lz\le |p_\gz(\theta)|\le R_\lz$.
Since $C_\gz(\theta)\gz'(\theta)\cdot p_\gz(\theta)\ge H(\gz(\theta),p_\gz(\theta))=\lz$, we have
that  $C_\gz(\theta)\ge \lz.$
By Corollary \ref{c3.x4}, we have that $C_\gz(\theta)|\gz'(\theta) |\le a_\lz $.
This, together with $\displaystyle |\gz'(\theta)|\ge {|p_\gz(\theta)|}^{-1}$,
implies that  $C_\gz(\theta)\le a_\lz R_\lz$.

\smallskip
\noindent{\it Case 2}: {\it $u(\cdot):=d_\lz(x,\cdot)$ is differentiable almost everywhere in $U\setminus\{x\}$}.
Let $\mathcal N(u)\subset U\setminus\{x\}$ denote the non-differentiable set of $u$.
By Lemma \ref{l3.x8} there exists $g\in[0,\,\ell_\lz(\gz)]\to\rn$ such that
$$\frac d{dt}u( \gz(t))=g(t)\cdot \gz'(t)$$
whenever $\gz$ is differentiable at $t$, and
\begin{equation}\label{e3.x5}g(t)\in\bigcap_{r>0}\mbox{closed convex hull of}\ Du(B(\gz(t),r)\setminus \mathcal N(u)),
\ \forall \ t\in [0,\ell_\lz(\gz)].
\end{equation}
Set $p_\gz(t)=g (t)$ for $t\in[0,\ell_\lz(\gz)]$. We need to show that
\begin{equation}\label{e3.x3.4}
H(\gz(t),p_\gz(t))\le\lz, \ {\mbox{for all}} \ t\in[0,\ell_\lz(\gz)].
\end{equation}
Observe that $H\in\lsc(\overline U\times\rn)$ implies that for any $\ez>0$
there exists  $r>0$ such that
 \begin{equation}\label{e3.x3.5}
 H(\gz(t),p )\le H(x ,p ) +\ez,
 \end{equation}
 whenever $|x-\gz(t)|\le r$ and $\displaystyle|p|\le  2\|Du\|_{L^\fz(B(\gz(t),r))}$.

Choose $\displaystyle p=\sum_{i=1}^m a_ip_i$ for some positive integer $m$,
where  $p_i\in Du(x_i)$, $x_i\in B(\gz(t),r)$ and $\displaystyle
|Du(x_i)|\le \|Du\|_{L^\fz(B(\gz(t),r))}$, $a_i\ge0$ for $1\le i\le m$,
and $\displaystyle
\sum_{i=1}^m a_i=1$.
Since $H(\gz(t),\cdot)$ is convex and $H(x_i,p_i)\le\lz$ for $1\le i\le m$,
it follows from (\ref{e3.x3.5}) that
\begin{eqnarray*}
H(\gz(t),p)&&\le \sum_{i=1}^m a_i H(\gz(t),p_i)\\
&&= \sum_{i=1}^m a_i H(x_i,p_i)+  \sum_{i=1}^m a_i (H(\gz(t),p_i)-H(x_i,p_i))\\
&&\le \lz+ \ez.
\end{eqnarray*}
Sending $\epsilon$ to zero, this implies (\ref{e3.x3.4}).

From (\ref{e3.x3.4}),  we have that
$$d_\lz(x,\gz(\theta))=\int_0^\theta p_\gz(t)\cdot \gz'(t)\,dt\le
 \int_0^\theta L_\lz(\gz(t),\gz'(t))\,ds=d_\lz(x,\gz(\theta)), \ \forall \ \theta\in [0,\ell_\lz(\gz)].
$$
Hence
$$p_\gz(t)\cdot \gz'(t)= L_\lz(\gz(t),\gz'(t)), \ {\mbox{for a.e.}}\ t\in [0, \ell_\lz(\gz)].
$$
From $H(\gz(t),p_\gz(t))\le\lz$ and the definition of $L_\lz(\cdot,\cdot)$, we have
that  $H(\gz(t),p_\gz(t))=\lz $ for almost all $t$.
Now we can follow the same argument as above to find a function $C_\gz$ such that
 $C_\gz(\theta)\gz'(\theta)\in \partial_pH(\gz(\theta),p_\gz(\theta))$
 and $C_\gz(\theta)$ enjoys the desired properties.
\end{proof}

With the help of Lemma \ref{l3.x7},  we prove Theorem \ref{t3.x6} now.
\begin{proof}[Proof of Theorem \ref{t3.x6}]
(i) For any given $\gz\in\Gamma_\lz(x)$, define $\mathcal D_\gz: [0,\,\ell_\lz(\gz)]\to \mathbb R_+$ by
$$\mathcal D_\gz (\theta)= \int^\theta_0\frac{1}{C_\gz(s)}\,ds,$$
where $C_\gz$ is given by Lemma \ref{l3.x7}.
For $t>0$, $\lz>0$, and $x\in U$,
define the set $E^t_\lz(x)$ by
$$E_\lz^t(x):= \Big\{\gz(\mathcal D_\gz^{-1}(t))\ \big| \ \gz\in\Gamma_\lz(x)\
{\mbox{and}}\ \max_{\theta\in[0,\ell_\lz(\gz)]}\mathcal D_\gz(\theta)\ge t
\Big\}.
$$
Then \eqref{e3.x2.3} is equivalent to
\begin{equation}\label{e3.x3.6}
\cl_t(x,\gz(\mathcal D_\gz^{-1}(t)) )= d_\lz(x,\gz(\mathcal D_\gz^{-1}(t)))-\lz t,
\end{equation}
provided $\gz\in\Gamma_\lz(x)$ and $\displaystyle\max_{\theta\in[0,\ell_\lz(\gz)]}\mathcal D_\gz(\theta)\ge t$.

From Lemma \ref{l3.x2}(ii), (\ref{e3.x3.6}) holds if we can show that
\begin{equation}\label{e3.x3.7}
\cl_t(x,\gz(\mathcal D_\gz^{-1}(t)) )\le d_\lz(x,\gz(\mathcal D_\gz^{-1}(t)))-\lz t.
\end{equation}
It follows from Lemma \ref{l3.x7} that for all  $\theta\in[0,\ell_\lz(\gz)]$ and $p\in\rn$,
it holds that
\begin{eqnarray*}
H(\gz(\theta),p)&&\ge
H(\gz(\theta),p_\gz(\theta))+ C_\gz(\theta)\gz'(\theta)(p-p_\gz(\theta))\\
&&=\lz+C_\gz(\theta)\gz'(\theta)(p-p_\gz(\theta)).
\end{eqnarray*}
Thus
\begin{eqnarray*}
L(\gz(\theta),  C_\gz(\theta)\gz'(\theta))&&=\sup_{p\in\rn}\Big\{ p\cdot  C_\gz(\theta)\gz'(\theta)- H(\gz(\theta),p)\Big\}\\
&&\le C_\gz(\theta)\gz'(\theta)\cdot p_\gz(\theta) - \lz = C_\gz(\theta)-\lz.
\end{eqnarray*}
Set $ \bz(\theta)=\gz( \mathcal D_\gz^{-1}(\theta))$ for $0\le\theta\le t$.
 Then we have that
 $$ \bz'(\theta)=   C_\gz\big(\mathcal D_\gz^{-1}(\theta) \big)
 \gz'\big(\mathcal D_\gz^{-1}(\theta)\big),$$
and hence
\begin{eqnarray*}
&&\cl_t\big(x, \gz(\mathcal D_\gz^{-1}(t))\big)
\le\int_0^tL(\bz(\theta),\beta'(\theta))\,d\theta\\
&&= \int_0^tL\big(\gz(\mathcal D_\gz^{-1}(\theta)), C_\gz(\mathcal D_\gz^{-1}(\theta))\gz'(\mathcal D_\gz^{-1}(\theta))\big)\,d\theta\\
&&=  \int_0^{\mathcal D_\gz^{-1}(t)}
L(\gz( \theta), C_\gz( \theta )\gz'( \theta) )C_\gz(\theta)^{-1}\,d\theta\\
&&\leq \int_0^{\mathcal D_\gz^{-1}(t)}\big(C_\gz(\theta)-\lz\big)C_\gz(\theta)^{-1}\,d\theta\\
&&= \mathcal D_\gz^{-1}(t)- \lz \mathcal D_\gz\big(\mathcal D_\gz^{-1}(t)\big)
=\mathcal D_\gz^{-1}(t)-\lz t.
\end{eqnarray*}
This, combined with the fact that $\displaystyle
d_\lz\big(x,\gz(\mathcal D_\gz^{-1}(t))\big)= \mathcal D_\gz^{-1}(t),$
implies (\ref{e3.x3.7}).

\smallskip
\noindent(ii)  For any given $y\in E^t_\lz(x)$, there exists $\gz\in\Gamma_\lz(x)$
such that  $y=\gz\big(\mathcal D_\gz^{-1}(t)\big)$. Moreover, it holds that
 $d_\lz(x,y)=\mathcal D_\gz^{-1}(t)$
 and hence $\mathcal D_\gz(d_\lz(x,y)) =t$, that is,
 $$\int_0^{d_\lz(x,y)}\frac1{C_\gz(s)}\,ds=t.$$
 This, together with $C_\gz(s)\le a_\lz R_\lz$, implies that
 $$d_\lz(x,y)\le a_\lz R_\lz t.$$
Thus we obtain that $E^t_\lz(x)\subseteq \overline{B_{d_\lz}(x,a_\lz R_\lz t)}$.
Hence $E^{<s}_\lz(x)\subset  \overline{B _{d_\lz}(x,a_\lz R_\lz s)}$ for all $s>0$.

On the other hand, if $t>s>0$ and $y\in E_\lz^t(x)$,
we have that
$$0\le \cl_t(x ,y)=d_\lz(x,y)-\lz t,$$
so that $d_\lz(x,y)\ge \lz t> \lz s$. Hence we obtain that
$$
 B_{d_\lz}(x,\lz s)\subset \overline U\setminus E_\lz^{\ge s}(x) = E_\lz^{< s}(x).$$

\noindent(iii) %We first show that $E_\lz^t(x_0)$ is closed. Assume that $x_i\in E_\lz^t(x_0)$ and $x_i\to z$. we try to show that $z\in E_\lz^t(x_0)$.
%Notice that in this case, $\cl_t(x_0,x_i)=d_\lz(x_0,x_i)-\lz t\to d_\lz(x_0,z)-\lz t.$
%It suffices to show that $\cl_t(x_0,x_i)\to\cl_t(x_0,z)$. Since $d_\lz(x_0,z)-\lz t\le \cl_t(x_0,z)$,
%we only need to show that for every $\ez>0$, $\cl_t(x_0,z)\le \cl_t(x_0,x_i)+\ez  $ when $i$ large enough.
%Assume that the curve $\az_i$ joining $x$ and $x_i$ satisfies
%$$\cl_t(x_0,x_i)\ge \int_0^tL(\az_i(\theta),\az_i'(\theta))\,d\theta-\ez.$$
%Let $\gz_i(\theta)=x_i+\theta(x_i-z)$ be the line segment joining $x_i$ and $z$. Then $(\az_i\cup\gz_i)$ joining $x,z$, where
%$ (\az_i\cup\gz_i)(\theta)=\az_i( \theta (1+\dz) )$ when $\theta\in[0,\,t/(1+\dz)]$ and $=\gz_i(\frac{1+\dz}\dz(\theta -\frac t{1+\dz}))$ for $\theta\in(t/(1+\dz),t]$.
%Then when $|x_i-z|\le\dz$, we have
%\begin{eqnarray*}\cl_t(x,z)&&\le\int_0^{t}L((\az_i\cup\gz_i)(\theta),(\az_i\cup\gz_i)'(\theta))\,d\theta\\
%&&
%= \int_0^{t/(1+\dz)}L( \az_i ((1+\dz)\theta),  ( 1+\dz)\az_i '((1+\dz)\theta))\,d\theta+ \\
%&&\quad\int_{t/(1+\dz)}^t L(\gz_i(\frac{1+\dz}\dz(\theta -\frac t{1+\dz})),
%\frac{1+\dz}\dz\gz_i'(\frac{1+\dz}\dz(\theta -\frac t{1+\dz})) )\,d\theta \\&&
%=\frac1{1+\dz}\int_0^{t}L( \az_i (\theta),  ( 1+\dz)\az_i '(\theta))\,d\theta
%+\frac\dz{1 +\dz} \int_0^1 L( \gz_i (\theta), \frac { 1+\dz}\dz(x_i-z) )\,d\theta \\
%&&\le \frac1{1+\dz}\int_0^{t}L( \az_i (\theta),  ( 1+\dz)\az_i '(\theta))\,d\theta+C\dz.
%\end{eqnarray*}
%Since $\az_i'$ is essentially uniformly  bounded, when $\dz$ is small enough, we have
%$$\cl_t(x,z)<(1+\dz)\cl_t(x,x_i)+\ez$$
%as desired. This implies that $z\in E^t_\lz(x_0)$.
To show that $ E^{\le t}_\lz(x) $ is closed, let $x_i\in  E^{\le t}_\lz(x) $,
and $x_i  \to x_\fz\in U$. Without loss of generality, assume that
$x_\fz\ne x$.  Then there exists $t_i\in (0,t]$
such that
$$d_\lz(x,x_i)-\lz t_i=\cl_{t_i}(x,x_i).$$
Moreover, it follows from the proof of (i) that for each $i$, there exists $\gamma_i\in \Gamma_\lambda(x)$
such that $\beta_i(\theta)=\gamma_i\big(\mathcal D_{\gz_i}^{-1}(\theta)\big), 0\le\theta\le t_i,$
satisfies $\beta_i\in \mathcal C\big(0, t_i; x, x_i; \overline U\big)$,
$x_i=\beta_i(t_i)$, and
\begin{equation}\label{e3.x3.66}
\cl_t(x, x_i)=\int_0^{t_i} L\big(\beta_i(\theta), \beta_i'(\theta)\big)\,d\theta
=d_\lambda(x, x_i)-\lambda t_i.
\end{equation}
It follows from Lemmas \ref{l2.w16} and \ref{l3.x7} that
\begin{eqnarray*}
r_\lz |\beta_i(\theta_1)-\beta_i(\theta_2)|&&\le d_\lambda\big(\beta_i(\theta_1), \beta_i(\theta_2)\big)
=d_\lz\Big(\gz_i\big(\mathcal D_{\gz_i}^{-1}(\theta_1)\big),
\gz_i\big(\mathcal D_{\gz_i}^{-1}(\theta_2)\big)\Big)\\
&&=\big|\mathcal D_{\gz_i}^{-1}(\theta_1)-\mathcal D_{\gz_i}^{-1}(\theta_2)\big|
\le \big\|(\mathcal D_{\gz_i}^{-1})'\big\|_{L^\infty([0, t_i])}|\theta_1-\theta_2|\\
&&\le \big\|C_{\gz_i}\big\|_{L^\infty([0, d_\lz(x,x_i)])}|\theta_1-\theta_2|\le
a_\lz R_\lz |\theta_1-\theta_2|
\end{eqnarray*}
holds for any $0\le \theta_1<\theta_2\le t_i$. Hence we conclude that
\begin{equation}\label{e3.x3.660}
\big\|\beta_i'(\cdot)\big\|_{L^\infty([0,t_i])}\le \frac{a_\lz R_\lz}{r_\lz}.
\end{equation}
Since  $x_\fz\ne x$,
there exists $i_0$ such that for $i\ge i_0$ it holds that
$$d_\lz(x_\fz, x_i)<\frac12d_\lz(x_\fz,x) \ {\rm{and}} \
|x_i-x_\fz|<\frac12\min\Big\{|x-x_\fz|,\dist(x_\fz,\partial U)\Big\}.$$
This further implies that
 $$d_\lz(x,x_i)\le d_\lz(x,x_\fz)+d_\lz(x_\fz,x)\
 \ {\rm{and}}\ \ |x-x_i|\ge \frac12|x-x_\fz|.$$
We claim that $\displaystyle\liminf_{i\to\fz}t_i>0$.
For, otherwise, there exists a subsequence $t_{i_k}$ such that
$ t_{i_k}\to0$.  Since
$$d_\lz(x,x_{i_k})=\cl_{t_{i_k}}(x,x_{i_k})+\lz t_{i_k}
\ge\cl_{t_{i_k}}(x,x_{i_k})\ge M\Big(\frac{d_U(x,x_{i_k})}{t_ {i_k}}\Big) d_U(x, x_{i_k}),$$
we obtain that
$$d_\lz(x,x_\fz)+d_\lz(x_\fz,x)\ge \frac12M\Big(\frac{d_U(x,x_\fz)}{2t_{i_k}}\Big)d_U(x,x_\fz)\to\fz
\ {\rm{as}}\ k\to \fz,$$
this is impossible.

By taking a subsequence, we may assume that $t_i\to t_\fz\in (0,\,t]$.
We want to show that
\begin{equation}\label{e3.x3.8}
\cl_{t_\fz}(x,x_\fz)=d_\lz(x,x_\fz)-\lz t_\fz.
\end{equation}
Since $$\cl_{t_i}(x,x_i)=d_\lz(x,x_i)-\lz t_i\to d_\lz(x,x_\fz)-\lz t_\fz\le \cl_{t_\fz}(x,x_\fz),$$
it suffices to prove
$$\liminf_{i\to\fz}\cl_{t_i}(x,x_i)\ge\cl_{t_\fz}(x,x_\fz).$$
Since $s\to \cl_{s}(\cdot,\cdot)$ is monotone decreasing, we may assume that
$t_i\ge t_\fz$.
Let $\gz_i(\theta)=\theta x_\fz+(1-\theta)x_i$, $\theta\in[0,1]$,
be the line segment between $x_i$ and $x_\fz$.
For $0<\delta\le \frac12$, define $(\bz_i\cup\gz_i)\in \mathcal C\big(0, t_\fz; x, x_\fz; \overline U\big)$
by
$$\displaystyle(\bz_i\cup\gz_i)(\theta)=\begin{cases}
\bz_i((1+\dz)\theta ) & \theta\in[0,\,\frac{t_i}{1+\delta}],\\
\gz_i\big(\frac{1+\dz}{(1+\dz){t_\fz}-t_i}(\theta -\frac {t_i}{1+\dz})\big)
& \theta\in [\frac{t_i}{1+\delta},t_\fz].
\end{cases}
$$
Let $i$ be sufficiently large so that $|x_i-x_\fz|\le\dz$
and $t_i\le (1+\frac12\dz)t_\fz$. Then we have that
\begin{eqnarray}\label{e3.x7.0}
\cl_{t_\fz}(x,x_\fz)&&\le\int_0^{{t_\fz}}L((\bz_i\cup\gz_i)(\theta),(\bz_i\cup\gz_i)'(\theta))\,d\theta
\nonumber\\
%&&
%= \int_0^{t_i/(1+\dz)}L( \az_i ((1+\dz)\theta),  ( 1+\dz)\az_i '((1+\dz)\theta))\,d\theta \\
%&&\quad+\int_{t_i/(1+\dz)}^{t_z} L(\gz_i(\frac{1+\dz}{(1+\dz){t_z}-t_i}(\theta -\frac {t_z}{1+\dz})),
%\frac{1+\dz}{(1+\dz){t_z}-t_i}\gz_i'(\frac{1+\dz}{(1+\dz){t_z}-t_i}(\theta -\frac{t_z}{1+\dz})) )\,d\theta \\
&&
=\frac1{1+\dz}\int_0^{t_i}L( \bz_i (\theta),  ( 1+\dz)\bz_i '(\theta))\,d\theta\nonumber\\
&&\quad
+\frac{(1+\dz)t_\fz-t_i}{1 +\dz} \int_0^{1} L\lf( \gz_i (\theta), \frac { 1+\dz}{(1+\dz){t_\fz}-t_i}(x_\fz-x_i)\r )\,d\theta
\nonumber \\
&&\le \frac1{1+\dz}\int_0^{{t_i}}L\big( \bz_i (\theta),  ( 1+\dz)\bz_i '(\theta)\big)\,d\theta+C({t_\fz})\dz,
\end{eqnarray}
where
$$C(t_\fz)=t_\fz\sup\Big\{L(y,q)\ |\  y\in\overline U, \ |q|\le \frac{4}{t_\fz}\Big\}<+\infty.$$
It follows from (\ref{e3.x3.660}) that there exists a compact set $K\subset\overline U\times \rn$ such that
$\beta_i([0,t_i])\subset K$ for all $i\ge 1$. Since $L\in\usc(\overline U\times \rn)$, $L$ is uniformly upper semicontinuous
in $K$. Hence for any $\epsilon>0$, there exists $\delta>0$ such that
\begin{equation}\label{e3.x3.7.1}
L\big(\beta_i(\theta), (1+\delta)\beta_i'(\theta)\big)
\le L\big(\beta_i(\theta), \beta_i'(\theta)\big)+\epsilon
\end{equation}
holds for all $i$ and $\theta\in [0, t_i]$.
This implies
$$\int_0^{{t_i}}L\big( \bz_i (\theta),  ( 1+\dz)\bz_i '(\theta)\big)\,d\theta
\le \int_0^{{t_i}}L\big( \bz_i (\theta),  \bz_i '(\theta)\big)\,d\theta+t_i\epsilon
=\cl_{t_i}(x, x_i)+t_i\epsilon.$$
Putting this into (\ref{e3.x7.0}), we obtain that
$$\cl_{t_\infty}(x,x_\infty)\le \cl_{t_i}(x, x_i)+C(t_\infty)\delta+t_i\epsilon.$$
This clearly implies that
$\displaystyle\liminf_{i\to\fz}\cl_{t_i}(x,x_i)\ge\cl_{t_\fz}(x,x_\fz)$.

Replacing $t_i$ by $t$, the closeness of $E^t_\lambda(x)$ follows from the same argument as above.
Similarly, we can also prove that $ E^{\ge t}_\lz(x)$ is closed. Here, notice that if $x_i\in E^{t_i}_\lz(x)$
for some $t_i\ge t$ and $x_i\to z$ for some $z\in U$, then $t_i$ is bounded. For, otherwise, we would
have that $d_\lz(x,x_i)-\lz t_i  <0<\cl_{t_i}(x,x_i)$. By the same argument as above,
the closeness of  $E^{\geq t}_\lz(x)$ follows.

Finally, the closeness of $ E^{\ge t}_\lz(x)$ implies that
$E^{<t}_\lz(x) $ is relatively open in $\overline U$.
 Due to the relative  closeness of $E^{\leq t}_\lz(x)$, we conclude that the relative boundary of $E^{<t}_\lz(x) $ in $\overline U$ is contained in $E^t_{\lz}(x)$.

\smallskip
\noindent (iv) For $z\in E_\lz^{t+s}(x)$, there exists $\gz\in \Gamma_\lz(x )$ such that
$z=\gz \big(\mathcal D_\gz^{-1}(t+s)\big) $.  Set $y=\gz\big(\mathcal D_\gz^{-1}(t)\big)$.
Then it follows from the definition of $E_\lz^t(x)$ that $y\in E_\lz^t(x)$.
Now we want to show that $z\in E_\lz^s(y)$. To see this, first observe that
\begin{eqnarray*}
&&t+s=\mathcal D_\gz\big(\mathcal D_\gz^{-1}(t+s)\big)=\int_0^{\mathcal D_\gz^{-1}(t+s)}\frac{1}{C_\gz(\theta)}\,d\theta\\
&&=\int_0^{\mathcal D_\gz^{-1}(t)}\frac{1}{C_\gz(\theta)}\,d\theta
+\int_{\mathcal D_\gz^{-1}(t)}^{\mathcal D_\gz^{-1}(t+s)}\frac{1}{C_\gz(\theta)}\,d\theta\\
&&=t+\int_{\mathcal D_\gz^{-1}(t)}^{\mathcal D_\gz^{-1}(t+s)}\frac{1}{C_\gz(\theta)}\,d\theta.
\end{eqnarray*}
This implies that
\begin{equation}\label{e3.x3.10}
\int_{\mathcal D_\gz^{-1}(t)}^{\mathcal D_\gz^{-1}(t+s)}\frac{1}{C_\gz(\theta)}\,d\theta=s.
\end{equation}
Define $\widetilde{\gz}(\theta)=\gz\big(\mathcal D_\gz^{-1}(t)+\theta\big)$ for $\theta\ge 0$. Then
\begin{equation}
\label{e3.x3.12}
\widetilde{\gz}(0)=y\ \ {\rm{and}}\ \ \widetilde{\gz}\big(\mathcal D_\gz^{-1}(t+s)-\mathcal D_\gz^{-1}(t)\big)=z.
\end{equation}
Now we claim that
\begin{equation}
\label{e3.x3.11}
\mathcal D_{\widetilde\gz}^{-1}(s)=\mathcal D_\gz^{-1}(t+s)-\mathcal D_\gz^{-1}(t).
\end{equation}
In fact, by a simple change of variables, (\ref{e3.x3.10}) gives
$$
\int_0^{\mathcal D_\gz^{-1}(t+s)-\mathcal D_\gz^{-1}(t)}\frac{1}{C_{\widetilde\gz}(\theta)}\,d\theta=s.
$$
On the other hand, it follows from the definition that
$$s=\int_0^{\mathcal D_{\widetilde\gz}^{-1}(s)}\frac{1}{C_{\widetilde\gz}(\theta)}\,d\theta.$$
Combining these two identities yields (\ref{e3.x3.11}), and (\ref{e3.x3.12})
and (\ref{e3.x3.11})
yield $z\in E_\lz^s(y)$.

Finally, since $\gz\in\Gamma_\lz(x)$, we have that
\begin{eqnarray*}d_\lz(x,z)&&=d_\lz\big(\gz(0), \gz(\mathcal D_\gz^{-1}(t+s))\big)\\
&&=\mathcal D_\lz^{-1}(t+s)\\
&&=\mathcal D_\gz^{-1}(t)+
\big(\mathcal D_\gz^{-1}(t+s)-\mathcal D_\gz^{-1}(t)\big)\\
&&=\mathcal D_\gz^{-1}(t)+\mathcal D_{\widetilde\gz}^{-1}(s)\\
&&=d_\lz\big(\gz(0), \gz(\mathcal D_\gz^{-1}(t))\big) + d_\lz\big(\widetilde\gz(0), \widetilde\gz(\mathcal D_\gz^{-1}(s))\big)\\
&&=d_\lz(x,y)+d_\lz(y,z).
\end{eqnarray*}
This completes the proof.
\end{proof}

\subsection{Two special types of Hamiltonians $H(x,p)$}

In this subsection,  we consider two special types of $H(x,p)$ and
examine their corresponding geometric structures of the action function $\cl_t(\cdot,\cdot)$.
For convenience, we assume that $U=\rn$.

\smallskip
{\noindent \it Case 1.  $H(x,p)=H(p)$ is independent of $x\in\rn$:
$H(p)$ is convex;  $H(p)\ge H(0)=0$ for all $p\in\rn$ and $\big\{p\in\rn: H(p)=0\big\}$ has
no interior points; and $\displaystyle\lim_{|p|\rightarrow+\infty}H(p)=+\infty$. }

%
%For all $t>0$ and $x,y\in\rn$, we do not know if it is correct that
% $$\cl_t(x,y)=\sup_{\lz>0}[d_\lz(x,y)-\lz t]?$$
% This is equivalent to ask if $$\bigcup_{\lz>0} E^t_\lz(x)=\rn\setminus\{x\}?$$
%In the case $H(x,p)=H(p)$ and $H(x,p)=\langle A(x)p,p\rangle$, this is correct.
We have the following Lemma, which was shown by \cite{acjs}.
Here we also sketch a proof.
\begin{lem} Let $H$ be given as in Case 1 above.
Then  for all $t,\lz>0$ and $x,y\in\rn$, we have that
$  d_\lz(x,y) =L_\lz(y-x) $,
 and
\begin{equation}\label{e3.y1}\cl_t(x,y) =\max_{\lz\ge0}\{d_\lz(x,y)-\lz t\}=tL\lf(\frac{y-x}t\r).
\end{equation}
Consequently, $\cl_t$ enjoys the linearity: if $\gz(\theta)=x+\frac\theta t(y-x)$ for $\theta\in (0,t)$,
then
 $$\cl_\theta(x,\gz(\theta))=  \frac\theta t\cl_t(x,y),\ \forall\ \theta\in (0,t).$$
In particular, it holds that
\begin{equation}\label{e3.y1.1}
E^t_\lz(x)=x+t\bigcup_{\{H(p)=\lz\}}  \partial_p H(p).
\end{equation}

\end{lem}

\begin{proof}
To see  $ d_\lz(x,y) =L_\lz(y-x)$, let $\gz(\theta)=x+\theta(y-x) $, $\theta\in[0,1]$. Then
$\gz\in \mathcal C(0,1; x, y;\rn)$ and $\gz'(\theta)=y-x$. Hence we have that
$$d_\lz(x,y)\le  \int_0^1 L_\lz\lf(  \gz'(\theta) \r)\,d\theta=  L_\lz\lf(  y-x  \r).$$
On the other hand,  for every $\gz\in \mathcal C(0,1;x, y; \rn)$, it follows
from the convexity of $L_\lz(\cdot)$  that
$$ L_\lz\lf( y-x \r)\le \int_0^1 L_\lz(\gz'(\theta))\,d\theta=\ell_\lz(\gz).$$
Taking infimum over all $\gz\in \mathcal C(0,1;x, y; \rn)$, this yields $ L_\lz(y-x)\le d_\lz(x,y)$.

To see \eqref{e3.y1}, take $\widetilde\gz(\theta)=x+\frac{\theta}{t}(y-x)$, $\theta\in[0,t]$,
so that $\widetilde\gz'(\theta)=\frac{y-x}{t}$. Hence we obtain that
\begin{equation}\label{e3.y8}
\cl_t(x,y)\le \int_0^t L\lf( \gz'(\theta)\r)\,d\theta= \int_0^t L\lf( \frac{y-x}t\r)\,d\theta= tL\lf( \frac{y-x}t\r).
\end{equation}
%Conversely,
%for every rectifiable curve $\gz$ joining $x,y$, we have $y-x=\int_0^1\gz'(\theta)\,d\theta$.
%Since $L$ is convex, we have
%$$tL\lf(\frac{y-x}t\r)=tL\lf(\int_0^1\frac1t\gz'(\theta)\,d\theta\r)\le t\int_0^1 L\lf( \frac1t\gz'(\theta)\r)\,d\theta.$$
%Taking infimum over all such curves, we have
%$tL\lf(\frac{y-x}t\r)\le \cl_t(x,y)$, and hence $tL\lf(\frac{y-x}t\r)= \cl_t(x,y)$.
 On the other hand, %by Lemma \ref{l3.x2}(ii), we need to prove
%\begin{equation}\label{e3.y2}
%tL\lf(\frac{y-x}t\r)\le L_\lz  (y-x) -\lz t\end{equation}
%for some $\lz\ge0$.
since $H(p)$ is coercive, for $y\ne x$  there is  $p_*\in\rn$ such that
\begin{equation}\label{e3.y3}L\lf(\frac{y-x}t\r)+H(p_*)=p_*\cdot \frac{y-x}t.\end{equation}
Denote $\lz=H(p_*)$. Then by Lemma \ref{l3.x2} (ii) we have
\begin{equation}\label{e3.y2}
tL\lf(\frac{y-x}t\r)=p_*\cdot(y-x)-tH(p_*)\le L_\lz  (y-x) -\lz t=d_\lz(x,y)-\lz t
\le\cl_t(x,y),
\end{equation}
which together with \eqref{e3.y8} gives \eqref{e3.y1}.

It follows from (\ref{e3.y3}) and (\ref{e3.y2})
that $y\in E^t_\lz(x)$ iff there exists $p_*\in\rn$, with $H(p_*)=\lz$,
such that
$$H(p)\ge H(p_*)+\frac{y-x}{t}\cdot (p-p_*), \ \forall\ p\in\rn.$$
This is equivalent to  $\displaystyle\frac{y-x}{t}\in\partial_p H(p_*)$.
Hence $y\in E_\lz^t(x)$ iff $\displaystyle y-x\in t\bigcup_{\{H(p)=\lz\}}\partial_p H(p)$.
This yields (\ref{e3.y1.1}).
\end{proof}

{\noindent \it Case 2. $H(x,p)=\langle A(x)\cdot p, p\rangle$, $(x,p)\in\rn\times\rn$.
Here $A:\rn\to\mathbb R^{n\times n}$ is symmetric, lower semicontinuous,
and there exists $C\ge1$ such that
\begin{equation}\label{e3.y1.2}
C^{-1}|p|^2\le \langle A(x)\cdot p,p\rangle\le C|p|^2, \ \ \forall\  (x,p)\in\rn\times\rn.
\end{equation}
}

\begin{lem} Let $H(x,p)$ satisfy the properties given in Case 2.
Then  for all $t,\lz>0$ and $x,y\in\rn$, we have that
\begin{equation}\label{e3.y1.3}
d_\lz(x,y) =\sqrt \lz d_1(x,y),
\end{equation}
\begin{equation}\label{e3.y1.4}
\cl_t(x,y) =\max_{\lz\ge0}\Big\{d_\lz(x,y)-\lz t\Big\}=\frac1{4t}d_1^2(x,y),
\end{equation}
and
\begin{equation}\label{e3.y1.5}
E^t_\lz(x)=\partial B_{d_1}(x,2\sqrt \lz t)=\partial B_{d_\lz}(x,2t).
\end{equation}
In particular, $\cl_t$ enjoys the linearity: if $\gz:[0,t]\to  \rn$ is a $d_1$-length minimizing geodesic curve joining $x$
to $y$, then
\begin{equation}\label{linear2}
\cl_\theta(x,\gz(\theta))=  \frac\theta t\cl_t(x,y),\ \forall\ \theta\in (0,t).
\end{equation}

\end{lem}
\begin{proof} %By dual formula of CP, we have
%\begin{eqnarray*}d_\lz(x,y)&&=\sup\{u(x)-u(y):\ H(z,Du(z))\le\lz\ a.\,e.\} \\
%&&=\sqrt\lz\sup\{(\sqrt\lz)^{-1}u(x)-(\sqrt\lz)^{-1}u(y):\ H(z,D[(\sqrt\lz)^{-1}u](z))\le1\ a.\,e.\} \\
%&&=\sqrt\lz\sup\{v(x)-v(y):\ H(z,Dv(z))\le1\ a.\,e.\} \\
%&&=\sqrt\lz d_1(x,y).
%\end{eqnarray*}
Observe that
$$L_\lz(x,q)=\sup_{H(x,p)\le\lz}p\cdot q=\sup_{|p|^2\le\lz}p\cdot(A^{-\frac12}(x)q)
=\sqrt\lz\big|A^{-\frac12}(x)q\big|%=\sqrt\lz\sqrt{A^{-1}(z)q\cdot q}
=\sqrt\lz L_1(x,q).$$
This yields (\ref{e3.y1.3}). From (\ref{e3.y1.3}), we see that
$$ d_\lz(x,y)-\lz t= \sqrt\lz d_1(x,y)-\lz t= -\lf(  \frac{d_1(x,y)}{2\sqrt t}-\sqrt{\lz t} \r)^2+\frac{d^2_1(x,y)}{4  t}.
$$
Hence we obtain that
$$\max_{\lz\ge0}\Big\{d_\lz(x,y)-\lz t\Big\} =\frac{d^2_1(x,y)}{4  t}.
$$
Observe that
\begin{eqnarray*}L(x,q)&&=\sup_{p\in\rn}\Big\{p\cdot q-\langle A(x)\cdot p, p\rangle\Big\}
=\sup_{p\in\rn}\Big\{p\cdot (A^{-\frac12}(x)q)-|p|^2\Big\}\\
&&=\sup_{p\in\rn}\Big\{-\big|p-\frac12(A^{-\frac12}(x)q)\big|^2+\frac14|A^{-\frac12}(x)q|^2\Big\}\\
&&= \frac14\big|A^{-\frac12}(x)q\big|^2.
 \end{eqnarray*}
Let $\gz\in\mathcal D\big(0,t; x, y; \overline U\big)$
be a $\ell_1$-length minimizing geodesic curve,  with constant speed $\frac{d_1(x,y)}{2t}$.
Since
$$\frac14\langle A^{-1}(\gz(s))\cdot\gz'(s), \gz'(s)\rangle
=\Big(\frac{d_1(x,y)}{2t}\Big)^2, \  {\rm{a.e.\ }} s\in [0,t],$$
we have that
\begin{eqnarray*}\cl_t(x,y)&&\le   \int_0^t \frac14A^{-1}(\gz(s))\gz'(s)\cdot\gz'(s) \,ds
 = \frac{d_1^2(x,y)}{4t}. \end{eqnarray*}
This, combined with Lemma \ref{l3.x2} (ii),
implies (\ref{e3.y1.4}).

It follows from  the above argument that $y\in E^t_\lz(x)$ iff
$$\frac1{4t}d_1^2(x,y)=\sqrt\lz d_1(x,y)-\lz t,$$
or, equivalently,
$$\sqrt\lz= \frac1{2t}d_1(x,y).$$
This implies that $y\in E^t_\lz(x)$ iff $d_1(x,y)= 2\sqrt\lz t$. Hence (\ref{e3.y1.5}) is proven.
\end{proof}

Motivated by Lemma 3.12 and Lemma 3.13, we would like to pose

\smallskip
\noindent {\bf Question 3.14.} Assume $H:\overline U\times\rn\to\rr_+$ satisfies (A1), (A2)$_{\rm{weak}}$, (A3), and (1.6),\\
(a) is it true that $\cl_t(x,y)=\sup_{\lz>0}(d_\lz(x,y)-\lz t )$ for all possible $t>0, x, y\in U$? \\
(b) does $\cl_t$ enjoy the linearity (\ref{linear2}) along all $d_\lz$-length minimizing curve $\gamma$ in $U$?

\section{Basic properties of Hamilton-Jacobi flows}
In this section, we will assume that $H(x,p)$ satisfies  (A1), (A2)$_\weak$, (A3), and (\ref{u-coercive}).
For  $r>0$, set $$U_r:=\Big\{x\in U\ | \  {\rm{dist}}(x,\partial U)>r\Big\}.$$

\subsection{Localization of $T^tu$}
\begin{lem}\label{l4.x1}
For each $r>0$, there exists  $  a_0(r)\ge 2$   such that for all $t\in\big(0,\,\frac{r}{a_0(r)}\big)$,
$x\in U_r$ and $y\in B(x,\, t)$,
we have
\begin{eqnarray}\label{e4.x0}
\cl_t(x,y)&=&\inf\Big\{\int_0^tL(\gz(\theta),\gz'(\theta))\,d\theta\ \big|\
\ \gamma\in \mathcal C\big(0,t; x, y; B(x, a_0(r)t)\big)\Big\}.
\end{eqnarray}
\end{lem}
\begin{proof} For $x\in U_r$, $0<t<\frac{r}2 $, and $y\in B(x,t)$,
since $\displaystyle\gamma(\theta)=x+\frac{\theta}{t}(y-x): [0,t]\to U$ joins $x$ to $y$,
we have that
\begin{equation}\label{e4.x01}
 \cl_t(x,y)\le   \int_0^tL\lf(x+\frac\theta t(y-x),\frac1 t(y-x)\r)\,d\theta\le
t\sup_{z\in U_{\frac{r}2}}\sup_{|q|\le 1}   L(z,q):=tC(r)<+\infty.
\end{equation}
For any $a\ge 2$, $t<\frac{r}{a}$, $z\notin B(x,at)$, $y\in B(x,t)$, and  $0<s<t$,  it follows
that $|y-z|> \frac{at}2$. By Lemma \ref{l3.x3}, we have
$$ \cl_s(x,z)\ge   M\lf(\frac{|z-x|}{s}\r){|z-x|} \ge   s   aM(a),  $$
and $$ \cl_{t-s}(z,y)\ge   M\lf(\frac{|y-z|}{t-s}\r)|y-z|\ge \frac{a(t-s)}2M(\frac{a}2). $$
Observe that there exists $ a_0=a_0(r)\ge 2 $ such that
$\frac{a}2 M(\frac{a}2)> C(r)+1$ whenever $a\ge a_0(r)$.
If $\gz\in\mathcal C\big(0,t; x, y; \overline U\big)$ satisfies
$$\cl_t(x,y)+t\ge \int_0^t L(\gamma(\theta), \gamma'(\theta))\,d\theta,$$
then we must have that $\gamma([0,t])\subset B(x,a_0(r)t)$. For, otherwise,
$\gz([0,t])\cap (U\setminus B(x,a_0t))\ne\emptyset$
and there exists $s_*\in(0,t)$ such that $\gz(s_*)\notin B(x,a_0t)$.
Hence we obtain
\begin{eqnarray*}\cl_t(x,y)+t&&\ge\int_0^tL(\gz(\theta),\gz'(\theta))\,d\theta\\
&&= \int_0^{s_*}L(\gz(\theta),\gz'(\theta))\,d\theta+\int_{s_*}^tL(\gz(\theta),\gz'(\theta))\,d\theta\\
&&\ge \cl_{s_*}(x,\gz(s_*)) +\cl_{t-s_*}(\gz(s_*),y)\\
&&\ge s_*aM(a)+(t-s_*)\frac{a}2M(\frac{a}{2})> t(C(r)+1).
\end{eqnarray*}
This contradicts to (\ref{e4.x01}). Hence (\ref{e4.x0}) holds.
\end{proof}

\begin{lem}\label{l4.x2}
(i) For any $\alpha,r>0,$
there exists $\eta_0=\eta_0(\alpha,r)>0$ such that if $0<t<\eta_0$, then
\begin{eqnarray}\label{e4.x2.0}
T^t u(x)=\sup_{y\in B(x,r)}\Big\{u(y)-\cl_t(x,y)\Big\}, \ \forall\  x\in U_r,%\ \mbox{\rm and}\  T_t u(x)=\inf_{y\in B(x,r)\cap  U}\left[u(y)+\cl_t(y,x)\right].
\end{eqnarray}
whenever
  $\osc_U u \le \alpha$.

\noindent(ii) For every $K,\alpha,r>0,$
there exists $a_K\ge 2$    such that if $0<t<\min\big\{\eta_0(\az,\frac{r}2),\frac{r}{a_K}\big\}$, then
\begin{eqnarray}\label{e4.x2.1}
T^t u(x)=\sup_{y\in B(x,a_Kt) }\Big\{u(y)-\cl_t(x,y)\Big\}, \ \forall\  x\in U_r,
  %\ \mbox{\rm and}\  T_t u(x)=\inf_{y\in B(x,r)\cap  U}\left[u(y)+\cl_t(y,x)\right].
\end{eqnarray}
whenever   $\osc _Uu\le\az$ and $\lip\big(u,U_{\frac{r}2}\big)\le K$.
\end{lem}

\begin{proof} (i)
For any $r,\alpha>0$, choose a sufficiently small $\eta_0=\eta_0(\alpha, r)>0$ so that
$$M\big(\frac{r}{\eta_0}\big)>\frac{\alpha}{r}+1.$$
Then for any $0<t<\eta_0$, $x\in U_{r}$ and $y\in U\setminus\overline{ B(x,r)}$, we have
\begin{eqnarray*}
u(y)-\cl_t(x,y) &&\leq u(y)-  tM\big(\big|\frac{y-x}{t}\big|\big)\big|\frac{y-x}{t}\big| \\
&&\leq u(y)-rM\big(\frac{r}{\eta_0}\big)\\
&&\leq u(y)-(\alpha+r)\\
&&< u(y)-\osc_{  U}u -r\\
&&<  u(x)\le T^tu(x).
\end{eqnarray*}
This  implies (\ref{e4.x2.0}).

 (ii) If $\osc u\le\az$ and $\lip(u, U_{\frac{r}2})\le K$, then by (i) we have that,
 for $x\in U_{\frac{r}2}$ and $t<\eta_0(\az,\frac{r}2)$,
 \begin{eqnarray*}
T^t u(x)=\sup_{y\in B(x,\frac{r}2) }\Big\{u(y)-\cl_t(x,y)\Big\}.
\end{eqnarray*}
Let $a_K\ge 2$ be such that
$M\left(a_K\right)>K.$
If $0<t<\min\big\{\eta_0(\az,\frac{r}2),\frac{r}{a_K}\big\}$,
 $x\in U_r$, and $y\in  B(x,\frac{r}2)\setminus  B(x,a_Kt)$,  then we have
\begin{eqnarray*}
u(y)-\cl_t(x,y) &&\leq u(y)-  tM\big(\big|\frac{y-x}{t}\big|\big)\big|\frac{y-x}{t}\big| \\
&&\leq u(y)-u(x)- M\left(a_K\right)|x-y|+u(x)\\
&&\leq  (\lip(u,U_{\frac{r}2})  - K)|x-y|+u(x)\\
&&<  u(x)\le T^tu(x).
\end{eqnarray*}
This implies (\ref{e4.x2.1}).
\end{proof}

\begin{rem}\label{r4.x3}\rm From Lemma \ref{l4.x2}, we can see that if $u\in C(U)$ then
$$\lim_{t\downarrow0}T^t u=\lim_{t\downarrow0}T_t u=u$$
uniformly on any compact subset of $U.$
In fact, for each compact set $K\Subset U$,  $r\in(0,\frac12\dist(K,\partial U)) $ and $0<t<\eta_0(\osc_{U_r}u,r)$, we have
$$u(x)\leq T^t u(x)\leq\sup_{\overline{B}(x,r) } u.$$
The  uniform convergence of $T^tu(x)$ to $u(x)$ on $K$, as $t\to 0$,  follows directly from
the uniform continuity of $u$ on $\widehat{K}:=\big\{x\in U\ |\ {\rm{dist}}(x,K)\le\frac12\dist(K,\partial U)\big\}$.
\end{rem}

\subsection{Semigroup property of $T^tu$}

\begin{lem}\label{l4.x4} For $u\in L^\fz(U)$ and $t,s>0$,
it holds that
\begin{eqnarray}\label{e4.x4.0}
 T^{t+s}u(x) =T^t(T^s u)(x), \  \forall \ x\in U.% \quad\mbox{\rm and}\quad  T_{t+s}u(x) =T_t(T_s u)(x).
 \end{eqnarray}
\end{lem}
\begin{proof}
First, we claim that for any $x, y, z\in U$,
\begin{equation}\label{e4.x4.1}
\cl_{t+s}(x,y)\le \cl_t(x,z)+\cl_s(z,y).
\end{equation}
In fact, let $\alpha\in \mathcal C\big(0,t; x, z; \overline U\big)$
 and $\bz\in\mathcal C\big(0,s; z, y; \overline U\big)$.
Define $\gamma:[0, t+s]\to U$  by
$$\gamma(\theta)=\begin{cases} \alpha(\theta) & 0\le \theta\le t,\\
\beta(\theta-t) & t\le \theta \le t+s.
\end{cases}$$
Then $\gamma\in\mathcal C\big(0, t+s; x, y; \overline U\big)$, and
$$\int_0^{t+s}L(\gamma(\theta), \gamma'(\theta))\,d\theta
=\int_0^{t}L(\alpha(\theta), \alpha'(\theta))\,d\theta
+\int_0^{s}L(\beta(\theta), \beta'(\theta))\,d\theta.$$
Taking infimum over all such curves $\az$ and $\bz$, this yields (\ref{e4.x4.1}).
By the definition of $T^t u$ and (\ref{e4.x4.1}), we have
\begin{eqnarray*}
T^t(T^s u)(x)&&=\sup_{z\in U}\Big\{T^s u(z)-\cl_t(x,z)\Big\}\\
&&=\sup_{z, y\in U}\Big\{u(y)-\cl_s(z,y)-\cl_t(x,z)\Big\}\\
&&\le \sup_{y\in U}\Big\{u(y)-\cl_{t+s} (x,y)\Big\}
=T^{t+s}u(x).
\end{eqnarray*}
To prove the opposite direction of the inequality, recall from the definition of $T^{t+s}u$
that for any $\epsilon>0$ there is a point $y\in  U$ such that
\begin{eqnarray}\label{e4.x1}
T^{t+s} u(x)\le  u(y )- \cl_{ t+s}(x,y ) +\ez.
\end{eqnarray}
By the definition of $\cl_{t+s}(\cdot,\cdot)$,
there is a $\gz\in \mathcal C\big(0,t+s; x, y; \overline U\big)$ such that
$$\cl_{ t+s}(x,y)\ge\int_0^{t+s}L(\gz(\theta),\gz'(\theta))\,d\theta-\ez.$$
Set $z=\gz(t)$ and define $\bz(\theta)=\gz(\theta+t)$ for $\theta\in[0,s]$.
Set $\az(\theta)=\gamma(\theta), \theta\in [0,\,t ]$. Then $\az\in \mathcal C\big(0, t; x, z;\overline U\big)$,
and $\bz\in\mathcal C\big(0,s; z, y; \overline U\big)$. It follows that
$$\cl_{ t+s}(x,y)\ge\int_0^{t }L(\az(\theta),\az'(\theta))\,d\theta+\int_0^{s }L(\bz(\theta),\bz'(\theta))\,d\theta -\ez\ge \cl_t(x,z)+\cl_s(z,y)-\ez.$$
Hence we obtain
\begin{eqnarray*}
T^{t+s}u(x)&&\le u(y)-\cl_t(x,z)-\cl_s(z,y)+2\ez\\
&&\le (T^s u)(z)-\cl_t (x,z)+2\ez\\
&&\le T^t(T^su)(x)+2\ez.
\end{eqnarray*}
Sending $\epsilon\to 0$, this implies
$$T^{t+s}u(x)\le T^t(T^s u)(x).$$
Thus (\ref{e4.x4.0}) is proven.
\end{proof}

\subsection{Lipschitz continuity of $(t,x)\mapsto T^t u(x)$}

%\begin{lem} Let $t,\lz>0 $ and $x\in\rn$.
%Then   we have \begin{equation}\label{e2.x4}T^t[-d_\lz(x,\cdot)](x)\le -d_\lz(x,x)+\lz t\quad \forall x\in\rn , \end{equation}
%and
% for all $x\in\rn$, \begin{equation}\label{e2.x5}T^t[ d_\lz(x,\cdot)](x)\le d_\lz(x,x)+\lz t \quad \forall x\in\rn,  \end{equation}
%where the equality holds if there exists an $\gz\in\Gamma_\lz(x)$ such that $\gz\cap E^t_\lz(x)\ne\emptyset$.
%\end{lem}
%\begin{proof}
%
%By \eqref{e2.x3}, we know that
%$$T^t[-d_\lz(x,\cdot)](x)=\sup_{y}[-d_\lz(x,y)-\cl_t(x,y)]\le \sup_{y}[-d_\lz(x,y)- d_\lz(x,y)+\lz t]\le -d_\lz(x,x)+\lz t. $$
%
%Similarly, by \eqref{e2.x3}, we know that
%$$T^t[ d_\lz(x,\cdot)](x)=\sup_{y}[d_\lz(x,y)- \cl_t(x,y)]\le \sup_{y}[d_\lz(x,y)- d_\lz(x,y)+\lz t]\le d_\lz(x,x)+\lz t. $$
%When $\gz\in\Gamma_\lz(x)$ such that $\gz\cap E^t_\lz(x)\ne\emptyset$, we take $y\in \gz\cap E^t_\lz(x)$ and obtain
%$$d_\lz(x,y)=d_\lz(x,x)+d_\lz(x,y)$$
%and $d_\lz(x,y)-\lz t=\cl_t(x,y)$.  This shows that
%$$T^t[ d_\lz(x,\cdot)](x)\ge[d_\lz(x,y)- \cl_t(x,y)]=  [d_\lz(x,y)- d_\lz(x,y)+\lz t]= d_\lz(x,x)+\lz t $$
%which gives \eqref{e2.x5}.
%
%
%\end{proof}

\begin{lem}\label{l4.x5}
For each $K,\az,r>0$, set $$\eta_1(K,\az,r):=\min\Big\{\eta_0(\az,\frac{r}2),
\frac{r}{a_K}, \frac{r}{a_0(r)}\Big\},$$
where $a_0(r)$ is given by Lemma \ref{l4.x1} and $\eta_0(\az,\frac{r}2),a_K$  are given
by Lemma \ref{l4.x2}. If $u:U\to\rr$ satisfies
$\osc _Uu\le\az$  and $\lip(u,U_{\frac{r}2})\le K$, then
the map $(t,x)\mapsto T^t u(x)$ is  Lipschitz continuous on $  [0,\,\eta_1(K,\az,r)]\times U_r$.
\end{lem}
\begin{proof} We divide it into three steps.

\smallskip
{\noindent \it Step 1.} If  $ 0\leq s<t< \eta_1(K,\az,r) $
and $x\in U_r$, then
\begin{equation}\label{e4.x1.0}
0\le T^tu(x)-T^su(x)\le C(K,\az,r)|t-s|.
\end{equation}
To see this, observe that, by Lemma \ref{l4.x2},
\begin{eqnarray}
T^t u(x) &&=\sup_{y\in B(x,a_Kt) }\Big\{u(y) -\cl_t(x,y)\Big\}.\nonumber
\end{eqnarray}
Now we consider two cases of $s$.

\smallskip
\noindent {\it Case a: $s=0$.} It follows $T^su(x)=u(x)$. Since $\cl_t(x,y)\ge0$, we have
 \begin{eqnarray*}
0\le T^t u(x)-T^su(x) &&\le \sup_{y\in B(x,a_Kt) }\Big\{u(y) -u(x)\Big\} \le a_KKt
\end{eqnarray*}
so that (\ref{e4.x1.0}) holds.

\smallskip
\noindent{\it Case b: $0<s<t$.} Then there exists a point $z\in B(x,a_Kt)$ such that
\begin{eqnarray}\label{e4.x2}
T^{t}u(x)\leq u(z)-\cl_{t}(x,z)+|t-s|.
\end{eqnarray}
Since $t<\frac{r}{a_0(r)}$,
 there exists a Lipschitz $\gamma:[0,t]\to B(x,a_0(r) t)$ joining $x$ to $z$ such that
\begin{eqnarray*}
\cl_{t}(x,z) \geq\int_0^{t}L(\gamma(\theta),\gamma'(\theta))\,d\theta-|t-s|.
\end{eqnarray*}
Since
\begin{eqnarray*}
 \int_0^{t}L(\gamma(\theta),\gamma'(\theta))\,d\theta%&&=  \int_0^{s }L(\gamma(\theta),\gamma'(\theta))\,d\theta+   \int_0^{t-s }L(\gamma(s+\theta),\gamma'(s+\theta))\,d\theta\\
  &&\ge \cl_{s}(x,\gz(s))+\cl_{t-s}(\gz(s),z),
\end{eqnarray*}
we then obtain
\begin{eqnarray*}
\cl_{t}(x,z) \geq\cl_{s}(x,\gz(s))+\cl_{t-s}(\gz(s),z)-|t-s|.
\end{eqnarray*}
Substituting this into \eqref{e4.x2}, we arrive at
\begin{eqnarray*}
T^{t}u(x)&&\leq u(z)-\cl_{s}(x,\gz(s))-\cl_{t-s}(\gz(s),z)+2|t-s|\\
&&=u(\gamma(s))-\cl_{s}(x,\gz(s))+u(z)-u(\gamma(s))-\cl_{t-s}(\gz(s),z)+2|t-s|\\
&&\leq T^s u(x)+u(z)-u(\gamma(s))-\cl_{t-s}(\gz(s),z)+2|t-s|.
\end{eqnarray*}
Since  $\gz(s)\in B(x,a_0(r)t)\subset U_{\frac{r}2}$, we have that
$$|u(z)-u(\gamma(s))|\le K|z-\gamma(s)|.$$
By Lemma \ref{l2.w16}, there exists  $\lambda_K>0$ such that
$$d_{\lambda_{K}}(\gamma(s),z)\geq K|z-\gamma(s)|.$$
Hence we have that
$$d_{\lambda_K}(\gamma(s),z)\ge u(z)-u(\gamma(s)).$$
This and Lemma \ref{l3.x2} (ii) imply that
\begin{eqnarray*}
\cl_{t-s}(\gamma(s),z)&&\geq d_{\lambda_{K}}(\gamma(s),z)-\lambda_K |t-s|
 \geq u(z)-u(\gamma(s))-\lambda_K |t-s|.
\end{eqnarray*}
Hence we obtain that
$$0\le T^{t}u(x)-T^s u(x)\leq(2+\lambda_K)|t-s|,$$
this yields (\ref{e4.x1.0}).

\smallskip
 \noindent{\it Step 2.} For $0<t<\eta_1(\az,K,r)$ and $x,y\in U_{r}$, it holds that
 \begin{equation}\label{e4.x3}
 |T^tu(x)-T^tu(y)|\le C(K,\az,r)|x-y|.
 \end{equation}
Without loss of generality, assume that $T^tu(x)<T^tu(y)$. It suffices to show
\begin{equation}\label{e4.x4}
T^tu(x)\ge T^tu(y)- C(K,\az,r)|x-y|.
\end{equation}
We consider two cases separately.

\smallskip
\noindent{\it Case  c: $y\notin  B(x,t)$.}  From step 1, it holds that
$$T^tu(y)\leq u(y)+(2+\lambda_K)t\leq u(y)+(2+\lambda_K)|x-y|.$$
On the other hand, it is easy to see that
$$T^tu(x)\ge u(x) \ge u(y)-K|x-y|.$$
Hence we obtain that
$$T^tu(x)\ge T^tu(y)- (2+\lambda_K+K)|x-y|,$$
this implies (\ref{e4.x4}) and hence (\ref{e4.x3}).

\smallskip
\noindent{\it Case  d: $y\in B(x,t)$.}
Set $s=t-|y-x|>0$, and let $\gz(\theta)=x+\theta(y-x)/(t-s)$, $\theta\in[0,\,t-s]$, be the line segment between
$x$ and $y$. Then $\gz([0,t-s])\subset B(x,\,t)$  joins $x$ to $y$.
Since  $t-s=|x-y|$ and $t<\frac{r}2$, we have
$$\cl_{t-s}(x,y)\le \int_0^{t-s}L\lf(\gz(\theta), \gz'(\theta)\r)\,d\theta
\le \sup_{z\in U_{\frac{r}2}, |q|\le 1}L(z,q) |x-y|\equiv C(r)|x-y|.$$
By the definition of $T^su(y)$, there exists $z\in U$ such that
$$
T^{s}u(y)\leq u(z)-\cl_{s}(y,z)+|x-y|.
$$
Since $\cl_t(x,z)\le \cl_{t-s}(x,y)+\cl_{s}(y,z) $,
we have that
$$\cl_{s}(y,z)\ge \cl_t(x,z)-  C(r)|x-y|,$$
and hence
$$
T^{s}u(y)\leq u(z)-\cl_t(x,z)+(1+C(r))|x-y|\le T^tu(x)+ (1+C(r))|x-y|.
$$
By step 1,
we have $$T^tu(y)-(2+\lz_K)|x-y|\le T^su(y).$$
Putting these two inequalities together yields
$$T^tu(x) \ge T^tu(y)-(3+C(r)+\lz_K)|x-y|. $$
This implies (\ref{e4.x4}) and (\ref{e4.x3}).

\smallskip
\noindent{\it Step 3.} For $s,t\in[0,\eta_1(K,\az,r)]$ and $x,y\in U_r$, it follows from both step 1 and
step 2 that
\begin{eqnarray*}
|T^tu(x)-T^su(y)|&&\le |T^tu(x)-T^tu(y)|+ |T^tu(y)-T^su(y)|\\
&&\le C(K,\az,r)(|x-y|+|t-s|).
\end{eqnarray*}
This completes the proof of Lemma \ref{l4.x5}.
\end{proof}

 \section{Proof of Theorem 1.5 }\label{s5}
 This section is devoted to the proof of Theorem 1.5, which asserts
 the equivalence between absolute subminimality, comparison with
 intrinsic cones from above, convexity, and pointwise convexity.
Throughout this section, $H(x,p)$ satisfies  (A1), (A2)$_ \weak $, (A3), and (\ref{u-coercive}).
We will  prove (i)$\Rightarrow$(ii), (ii)$\Rightarrow$(iii), (iii)$\Rightarrow$(iv),
 and (vi)$\Rightarrow$(i).

\subsection{Absolute subminimality implies CICA}

\noindent{\it Proof of i) $\Rightarrow$ ii)}. We argue by contradiction. Suppose that
$u\in C(U)$ is an absolute subminimizer  but $u\notin$CICA$(U)$. Then there exist
an open connected $V\Subset U$, $x_0\in U\setminus V$, and $\lambda>0$ such that
\begin{equation}\label{e5.x0}
u(x) >d_\lambda (x_0, x) \ {\rm{in}}\ V; \ u(x)=d_\lambda (x_0,x) \ {\rm{on}}\ \partial V.
\end{equation}
Since $u$ is absolute subminimizing, we have that
\begin{equation}\label{e5.x00}
\big\|H(\cdot, Du)\big\|_{L^\infty(V)}\le \big\|H(\cdot, Dd_\lambda(x_0,\cdot))\big\|_{L^\infty(V)}
\le \lambda,
\end{equation}
where we have used Lemma \ref{l3.x9} (ii) in the last inequality (see also \cite{cp} Proposition 2.10).
For $x_1\in V$, let $\gamma\in {\rm{Lip}}([0,1], U)$ be a $d_\lambda$-length minimizing curve joining
$x_0$ to $x_1$. Then there exists $\theta_0\in (0,1)$
such that $p_0=\gamma(\theta_0)\in \gamma([0,1])\cap \partial V$, and $\gamma: [\theta_0, 1]
\to U$ is a $d_\lambda$-length minimizing curve joining $p_0$ to $x_1$, with
$\gamma((\theta_0, 1])\subset V$. It follows from Remark \ref{r3.x10} that
$$u(x_1)-u(p_0)\le d_\lambda(p_0, x_1).$$
Since $u=d_\lambda(x_0,\cdot)$ on $\partial V$, we have
$u(p_0)=d_\lambda(x_0, p_0)$. Hence
we have
$$u(x_1)\le d_\lambda(x_0,p_0)+d_\lambda(p_0,x_1).$$
On the other hand, from the choice of $p_0$ we  have
$$d_\lambda(x_0,x_1)=d_\lambda(x_0, p_0)+d_\lambda (p_0, x_1).$$
Therefore we would have
$$u(x_1)\le d_\lambda(x_0, x_1).$$
This contradicts to (\ref{e5.x0}). \qed

\subsection{CICA implies convexity criteria} %Proof of (ii)$\Rightarrow$(iii)}

\begin{lem}\label{l5.x1}
Suppose that $u\in \cica(U)$.
Then $u\in C(U)\cap\lip_\loc(U).$
\end{lem}

\begin{proof}
Since $u\in\usc(U)$,  to see $u\in C(U)$  it suffices to show $u\in\lsc(U)$, that is,
$ O_l:=\{x\in U : u(x)>l\}$ is open for each $l\in \rr.$
For $l\in\rr$ and $x\in O_l$, assume that $u$ is not constant on some neighborhood of $x$.
Set $\delta:=\frac13\dist(x,\partial U)$.  There exists a sufficiently large $\lambda>0$ such that
\begin{eqnarray}\label{e5.x1}
0<\max_{\overline{B(x,2\delta)}}u-l\le r_\lz \dz. %\leq\min_{|y-x|\geq \delta}d_{\lambda}(y,x).
\end{eqnarray}
Choose a sufficiently small $\eta\in(0,\dz)$ so that
\begin{eqnarray}\label{e5.x2}
\max_{y\in \overline {B(x,\eta)}}d_{\lambda}(y,x)<u(x)-l.
\end{eqnarray}
We want to show that  $u(y)>l$ whenever $y\in B(x,\eta)$. For, otherwise,
there exists $y_*\in B(x,\eta)$  and $u(y_*)\leq l.$
If $z\in \partial B(x,2\delta),$  then it holds that
$$d_{\lambda}(y_*,z)\ge r_\lz |y_*-z|\ge r_\lz(2\dz-\eta)>r_\lz\delta\ge \max_{\overline{B(x,2\delta)}}u-l.$$
This implies that
$$u(z)\le l+d_\lambda (y_*,z) \ {\rm{on}}\ \partial \big(B(x,2\delta)\setminus\{y_*\}\big).
$$
Since $u\in {\rm{CICA}}(U)$, it follows that
$$u(z)\le l+d_\lambda(y_*,z), \ \forall \ z\in B(x,2\delta),$$
and hence
$$u(x)\le l+d_\lambda(y_*,x),$$
this contradicts to \eqref{e5.x2}. This implies that $O_l$ is open for all $l\in\rr$
and hence $u\in\lsc(U)$.

To prove $u\in\lip_\loc(U)$, let $r>0$, by the continuity of $u$ on $U$, we have
that $\osc_{U_r}u<\fz$.
Choose $\lz>0$ such that $\osc_{U_r}u\le \frac{r_\lz r}{2}$.
For any $x\in U_{2r}$, we have that $B(x,\frac{r}2)\subset U_r$, and
$$u(y)-u(x)\le\osc_{U_r}u\le \frac{r_\lz r}2\le d_\lz(x,y), \ \forall\ y\in\partial B(x,\frac{r}2).$$
Since $u\in\cica(U)$, this yields
$$u(y)-u(x)\leq d_{\lambda}(x,y), \ \forall\ y\in B(x,\frac{r}2).$$
Hence $u\in \lip(U_r)$.
\end{proof}

Now we are ready to prove (ii)$\Rightarrow$(iii). The proof relies heavily
on Theorem \ref{t3.x6}. In contrast with \cite{acjs}, it is unknown that  for $V\Subset U$,
$u\in\cica(U)$ implies $T^tu\in \cica (V)$
for general Hamiltonians $H(x,p)$. The proof involves a subtle application of $u\in\cica(U)$
and some deep analysis of the structure of Hamilton-Jacobi flows.

\smallskip
\noindent {\it Proof of (ii)$\Rightarrow$(iii).}
Let $u\in \cica(U)$. By Lemma \ref{l5.x1}, we may assume that  $u\in\lip_{\loc}(U)$ and $\osc_Uu<\fz$.
For $r>0$, set $\az=\osc_Uu$  and $K=\lip(u,U_{\frac{r}2})$.
By Lemma \ref{l4.x5}, there exists $\eta_1(\az,K,r)>0$
such that the map $(s,x)\mapsto T^s  u(x) $ is Lipschitz on $[0,\eta_1(\az,K, r)]\times U_{ r}$.
Denote by $C(K,\az,r)>0$ the corresponding Lipschitz constant.
Let $a_C\ge 1$ be such that $M(a_C)\ge R_C$.
Set $$\eta(K,\az,r):=\frac1{8}\min\Big\{\eta_1(K,\az,r),\eta_0(\az,\frac{r}4),\frac{r}{4a_C}\Big\}.$$
Now we need

\smallskip
\noindent{\it Claim 1}. {\it For $ s \in[0,\eta(K,\az,r))$, $  h\in(0,\eta(K,\az,r))$,  and  $x\in U_{2r} $,
it holds that
\begin{eqnarray}\label{e5.x3}
S^{+}(T^su)(x)\leq\frac{T^{s+h}  u(x)-T^su(x)}h.
\end{eqnarray}}

Assume Claim 1 for the moment. For  $x\in U_{2r}$, from Lemma \ref{l4.x5}
 the map  $t\mapsto T^t u(x)$ is Lispchitz on $[0,\eta_1(\az,K,r)]$.
Then we have (see also \cite[p.426]{acjs}) that for all $ s,t \in[0,\eta(K,\az,r))$ and $s<t$,
$$
\frac{d}{d\tau}\Big|_{\tau=s}\lf[\frac{T^tu(x)-T^\tau u(x) }{t-\tau}\r]=\frac1{t-s}\lf[-S^{+}(T^su)(x)+ \frac{T^tu(x)-T^su(x) }{t-s}\r]\ge 0$$
whenever the map $\tau\mapsto T^\tau u(x)$ is differentiable at $s$.
This shows that the map  $t\mapsto T^t u(x)$
is convex on $[0,  \eta(K,\az,r))$ for all $x\in U_{2r}$. Hence (iii) holds.

\begin{proof}[Proof of Claim 1.]
 We first observe that for every $x\in U_{2r}$,
 $$\bigcup_{\{w\in B(x,\frac{r}2)\}}\bigcup_{\{\lz<C\}} E^{ <2\eta}_{\lz} (w)  \subset B(x,r).$$
  In fact,  if $y\in E^t_\lz(w)$ for some $w\in B(x,\frac{r}2)$, $t<2\eta$ and  $\lz<C$,
  then $$d_\lz(w,y)=\cl_t(w,y)+ \lz t>\cl_t(w,y).$$
 It follows that $d_U(w,y)<2a_C\eta$. For, otherwise, we would have
  $$\cl_t(w,y)\ge M\lf(\frac{d_U(w,y)}t\r)d_U(w,y)\ge M\lf(\frac{2a_C\eta}t\r)d_U(w,y)
  \ge R_Cd_U(w,y)\ge d_\lambda (w,y).$$
 Therefore, by $a_C\eta<\frac{r}4$ and $|w-x|<\frac{r}2$, we obtain
  $$E^t_\lz(w) \subset \overline U\setminus \Big\{y\in\overline U: d_U(w,y)\ge 2a_C\eta\Big\}
  = B(w,2a_C\eta)\subset B(x,r).$$
Given $s\in[0,\eta)$, $h\in(0,\eta)$, and $x\in U_{2r},$   set
\begin{equation}\label{e5.lambda}
\lambda:=\frac{T^{s+h} u(x)-T^su(x)}{h}\le C.
\end{equation}
We may assume $\lz>0$. For, otherwise, $S^+(T^su)(x)=0\le \lz$.
Choose a sufficiently small $\sz_0\in(0,\eta-s)$ so that for all $\sz<\sz_0$, it holds
\begin{equation}\label{e5.yy00}
\min\lf\{M\lf( \frac{r}{4\sz} \r)\frac{r}4,\,M\lf( \frac {\lz h}{\sz R_\lz}\r)\frac {\lz h}{R_\lz}\r\}
\ge \sup_{z\in B(x,\frac{r}2)}\cl_{s }(x,z)+\lambda.
\end{equation}
For $\sz<\sz_0$ and $s+\sz<\eta<\eta_0(\az,\frac{r}4)$, it follows from Lemma \ref{l4.x2} that
$$T^{s+\sz}u(x)=\sup_{z\in B(x,\frac{r}4)}\Big\{u(z)-\cl_{s+\sz}(x,z)\Big\}.$$
(\ref{e5.x3}) follows, if we can show that for an arbitrary $\ez>0$,
\begin{eqnarray}\label{e5.yy3}
u(z)-\cl_{s+\sz}(x,z)
 \le    T^su(x)+   \lz\sz+ \lz\ez \quad \mbox{for all} \  z\in B(x,\frac{r}4).
\end{eqnarray}
Indeed,  (\ref{e5.yy3}) yields that
$$T^{s+\sz}u(x)\le T^su(x)+\lz (\sz+\ez).$$
Sending $\ez\to0$, this implies that
$T^{s+\sz}u(x)\le T^su(x)+\lz  \sz $ and  hence
$$ S^+(T^su)(x)\le\limsup_{\sz\to0}\frac{T^{s+\sz}u(x)-T^su(x)}{\sz}\le\lz.$$

To show \eqref{e5.yy3}, recall from the definition of $\lz$ given by (\ref{e5.lambda})
that for all $z,w\in U$,
$$
u(z)\le T^su(x)+\lz h+ \cl_{s+h}(x,z)
\le T^su(x)+\lz h+ \cl_{s}(x,w)+\cl_h(w,z).
$$
If $w\in B(x,\frac{r}2)$ and
$z\in E^h_\lz(w)\subset B(x,r)$, then we have
\begin{equation}\label{e5.yy1}
u(z)\le T^su(x)+  \cl_{s}(x,w)+d_\lz(w,z).
\end{equation}
It is clear that (\ref{e5.yy1}) also holds when $z=w$. In particular, (\ref{e5.yy1}) holds
on $\partial \big(E^{\le h}_\lz(w)\setminus\{w\}\big)$.  Since $u\in\cica(U)$,
we conclude that \eqref{e5.yy1} holds for all $z\in E^{\le h}_\lz(w)$.

For $\sz<\sz_0$ and $z\in B(x,\frac{r}4)$,
there is $\gz\in\mathcal C\big(0,s+\sz; x, z; \overline U\big)$ such that
$$\cl_{s+\sz}(x,z)\ge \int_0^{s+\sz}L(\gz(\theta),\gz'(\theta))\;d\theta-\lz \ez,$$
and hence
\begin{equation}\label{e5.yy2}
\cl_{s+\sz}(x,z)\ge\cl_s(x,\gz(s))+\cl_{\sz}(\gz(s),z)-\lz\ez.
\end{equation}
It turns out that $|z-\gz(s)|<\frac{r}4$. For, otherwise, Lemma \ref{l3.x3} and (\ref{e5.yy00})
 imply that
\begin{eqnarray*}
\cl_{\sz}(\gz(s),z)&&\ge M\lf(\frac{ d_U(z,\gz(s))}\sz\r) d_U(z,\gz(s))\ge M\lf( \frac{r}{4\sz} \r)\frac{r}4\\
&&\ge \cl_{s }(x,z)+\lambda\ge \cl_{s+\sz}(x,z)+\lambda.
\end{eqnarray*}
This contradicts to (\ref{e5.yy2}).  Thus $\gz(s)\in B(x,\frac{r}2)$.
Similarly, we also have
$z\in B_{d_\lz}(\gz(s),\lz h)$. For, otherwise, we have
$$d_U(z,\gz(s))\ge \frac{d_\lz(\gz(s),z)}{R_\lz}\ge \frac{\lz h} {R_\lz},
$$
and hence
\begin{eqnarray*}
\cl_{\sz}(\gz(s),z)&&\ge M\big(\frac{ d_U(z,\gz(s))}\sz\big) d_U(z,\gz(s))
\ge M\big( \frac {\lz h}{\sz R_\lz}\big)\frac {\lz h}{R_\lz}\\
&&\ge \cl_{s }(x,z)+\lambda\ge \cl_{s+\sz}(x,z)+\lambda.
\end{eqnarray*}
This also contradicts to (\ref{e5.yy2}).
By Theorem \ref{t3.x6}, we know that $ B_{d_\lz}(\gz(s),\lz h)\subset  E^{<h}_\lz(\gz(s))$.
Therefore we have $z\in E^{<h}_\lz(\gz(s))$. Take $w=\gz(s)$, \eqref{e5.yy1} holds on $E^{\le h}_\lz(\gz(s))$.
Thus we arrive at
$$
u(z)\le T^su(x)+  \cl_{s}(x,\gz(s))+d_\lz(\gz(s),z).$$
This, together with \eqref{e5.yy2} and $\cl_\sz(\gz(s),z)\geq d_\lz(\gz(s),z)-\lz\sz$,
yields that
\begin{eqnarray*}
u(z)-\cl_{s+\sz}(x,z)&&\le T^s u(x) -\cl_\sz(\gz(s),z)+\lz\ez+d_\lz(\gz(s),z)
\leq T^su(x)+   \lz\sz+ \lz\ez.
\end{eqnarray*}
Hence \eqref{e5.yy3} follows.
\end{proof}

\subsection{Convexity criteria implies pointwise convexity criteria}% Proof of (iii)$\Rightarrow$(iv)}

{\it Proof of (iii)$\Rightarrow$(iv).}
Suppose that $u\in C(U)$ is bounded and satisfies the convexity criteria.
To show $u$ also satisfies the pointwise convexity criteria, it suffices to prove
that $S^+u\in {\rm{USC}}(U)$ and $u\in\lip_\loc(U)$.

Since for any $r>0$, there exists $0<\delta=\delta(r)<\frac{r}2$ such that $t\mapsto T^tu(x)$ is convex
on $[0,\delta]$ for any $x\in U_r$, it follows that for any $x\in U_r$,
$\displaystyle S^+_tu(x):=\frac{T^tu(x)-u(x)}{t}: (0, \delta]\to \rr_+$ is a monotone nondecreasing
function. Thus
$$
S^{+}u(x)=\lim_{t\rightarrow 0^+} S^+_tu(x)$$
exists for all $x\in U_r$ and is a upper semicontinuous function in $U_r$. Since $r>0$ is arbitrary,
we conclude that $S^+u\in {\rm{USC}}(U)$.

To see $u\in \textrm{Lip}(U_r)$, first observe that
\begin{eqnarray}\label{e5.x5}
\lambda:=\sup_{x\in U_r} \sup_{0<t\leq  \delta}\frac{T^t u(x)-u(x)}{t}
=\sup_{x\in U_r}\frac{T^{\delta} u(x)-u(x)}{ \delta}\le \frac{2}{\delta}\|u\|_{L^\infty(U_{\frac{r}2})}.
\end{eqnarray}
Hence
\begin{equation}\label{e5.x5.1}
u(y)-u(x)\le \lambda t+\cl_t (x,y), \ \forall x, y\in U_r, \ 0<t\le \delta.
\end{equation}
It follows from (\ref{e5.x5.1}) and  Theorem \ref{t3.x6} that
\begin{equation}\label{e5.x5.2}
u(y)-u(x)\le d_\lambda(x,y), \ \forall\ x\in U_{2r} \ {\rm{and}} \ \forall\ y\in E_\lambda^{<\delta}(x).
\end{equation}
Since $B_{d_\lambda}(x,\lambda\delta)\subset E_\lambda^{<\delta}(x)$, (\ref{e5.x5.2}) yields
\begin{equation}\label{e5.x5.21}
u(y)-u(x)\le d_\lambda(x,y), \ \forall\ x\in U_{2r} \ {\rm{and}} \ \forall\ y\in B_{d_\lambda}(x,\lambda \delta).
\end{equation}
This and Lemma \ref{l3.x9} (ii) imply that $u\in {\rm{Lip}}(U_{2r})$, and
$$\big\|H(\cdot, Du)\big\|_{L^\infty(U_{2r})}\le \lambda.$$
Hence, by (A3) we conclude that there exists $C_r>0$ such that
$$\big\|Du\big\|_{L^\infty(U_{2r})}\le C_r.$$
This completes the proof. \qed

\subsection{Pointwise convexity criteria implies absolute subminimality}%Proof of (iv)$\Rightarrow$(i)}

To prove (iv)$\Rightarrow$(i), we need the following Lemmas.

\begin{lem}\label{l5.x3} For $\az>0$, $K>0$ and $r>0$, let $\eta_0(\alpha, \frac{r}2)>0$ and $a_K$ be given
by Lemma \ref{l4.x2}.
Assume that $ \lip(u,U_{r/2})\le K $  and $\osc_Uu\le \az$.
If $x\in U_{2r}$, $s<\frac{r}2$,  and $\lz\ge\| H(\cdot,Du)\|_{L^{\infty}(B(x,s))} $ and $\lz>0$,
then for   $0<t<\min\big\{\eta_0(\az,\frac{r}2),\frac{r}{a_K}, \frac{sr_\lz}{2R_\lz a_K}\big\}$  and $y\in B(x,\frac{s}4)$,
there holds
\begin{equation}\label{e5.x3.0}
 T^t u(y)-u(y) \leq \lz t.
\end{equation}
\begin{proof} By Lemma \ref{l4.x2},  when $0<t<\min\big\{\eta_0(\az,\frac{r}2),\frac{r}{a_K}\big\}$ and $x\in U_r$, we have
$$ T^t u(y)-u(y) = \sup_{z\in   B(y,a_Kt)}\Big\{u(z)-u(y)-\cl_t(y,z)\Big\}.$$
For each $s<\frac{r}2$ and $x\in U_r$,  if $t<\frac{sr_\lz}{2R_\lz  a_K}$ and $y\in B(x,\frac{s}2)$,
then we have
$$B(y,a_Kt)\subset B_{d_\lz}(y,a_KtR_\lz)\subset B\big(y,\frac{a_KtR_\lz}{r_\lz}\big)\subset B\big(y,\frac{s}2\big)\subset B(x,s).$$
Notice that for every   $z\in B(y, a_K t)$,   there is a $d_\lz$-length
minimizing geodesic curve $\gamma$ joining $y$ to $z$, which is contained in
$B_{d_\lz}(x,a_KtR_\lz)\subset B  (x,s)$. Since $\| H(\cdot,Du)\|_{L^{\infty}(B  (x,s))}\le \lambda$ and $\lz>0$,
it follows from Lemma \ref{l3.x9} and Remark \ref{r3.x10}
that
$$u(z)-u(y)\le d_{\lz }(y,z),$$
and hence
\begin{eqnarray*}
u(z)-u(y)-\cl_t(y,z)&&\le u(z)-u(y)-d_{\lz }(y,z)+ {\lz }  t  \le \lz t .
\end{eqnarray*}
This implies (\ref{e5.x3.0}).
\end{proof}
\end{lem}

\begin{lem}\label{l5.x4} For $ u\in\lip_\loc(U)$, with $\osc_Uu<\fz$, and any open set $V\Subset U$, we have
\begin{equation}\label{e5.x40}
\sup_{x\in V}S^{+}u(x)=\big\|H(x,Du)\big\|_{L^{\infty}(V)}.
\end{equation}
\begin{proof}
For every $x\in V,$   Lemma \ref{l5.x3}  yields
  $$S^{+}u(x)=\limsup_{t\to0}\frac{T^t u(x)-u(x)}{t}\leq \big\|H(\cdot,Du )\big\|_{L^{\infty}(V)}.$$
For the opposite direction of the inequality, it suffices to prove that
$H(x,Du(x) )\le S^+u(x)$ whenever $u$ is differentiable at $x\in V$.
Suppose that $u$ is differentiable at $x\in V$. Then for every $p\in \rr^n$ and  sufficiently small $t>0$,
we have
\begin{eqnarray*}
\frac{T^t u(x)-u(x)}{t}&&=\sup_{y\in U}\left[\frac{u(y)-u(x)}{t}-\frac{1}{t}\cl_t(x,y)\right]
\geq\frac{u(x+tp)-u(x)}{t}-\frac{1}{t}\cl_t(x,x+tp).
\end{eqnarray*}
We claim that
\begin{equation}\label{e5.x41}
\frac{1}{t}\cl_t(x,x+tp)\leq L\left(x,p\right)+\ez_p(t),
\end{equation}
where $\;\epsilon_p(t)\rightarrow0$ as $\;t\rightarrow0.$
In fact, let $\gz(\theta)=x+\theta p$, $\theta\in[0,\,t]$. Then by the upper semicontinuity of $L$ we have
\begin{eqnarray*}
\frac{1}{t}\cl_t(x,x+tp) \leq \frac{1}{t}\int_0^t L\left(x+\theta p,p\right)d\theta\le \sup_{z\in B(x,\,t|p|)}L(z,p)\le L(x,p)+\epsilon_p(t),
\end{eqnarray*}
where$\;\epsilon_p(t)\rightarrow0$  as $\;t\rightarrow0.$
Applying (\ref{e5.x41}),  we arrive at
\begin{eqnarray*}
\frac{T^t u(x)-u(x)}{t}\geq\frac{u(x+tp)-u(x)}{t}- L(x,p)-\epsilon_p(t).
\end{eqnarray*}
Sending $t\rightarrow0$, this yields that
$$S^{+}u(x)\geq Du(x)\cdot p-L(x,p).$$
Taking the supremum over all $p\in \rr^n,$ we conclude
that $$S^{+}u(x)\geq \sup_{p\in\rn}\Big\{Du(x)\cdot p-L(x,p)\Big\}=H(x,Du(x)).$$
This completes the proof.
\end{proof}
\end{lem}
The following lemma on the increasing slope estimate plays a key role in the proof of
``iv) $\Rightarrow$ i)".  It is an extension of
\cite[Lemma 4.6]{acjs} where $H=H(p)$ is assumed.
Since the approach of \cite[Lemma 4.6]{acjs} doesn't seem to be applicable to the case that $H(x,p)$
has $x$-dependence, we provide a different argument.

\begin{lem}\label{l5.x5} Assume that $u\in\lip_\loc(U)$, with $\osc_Uu<\fz$, satisfies the pointwise convexity criteria.
If $x,y\in U$ and $0<t<\delta(x)$ satisfy
\begin{eqnarray}\label{e5.x6}
T^t u(x)=\lim_{y_i\to y}\Big\{u(y_i)-\cl_t(x,y_i)\Big\},
\end{eqnarray}
then
\begin{eqnarray}\label{e5.x7}
\frac{T^t u(x)-u(x)}{t}\leq S^{+}u(y).
\end{eqnarray}
\begin{proof}
For $\ez>0$ and $0<s<t$, there exist  $\gamma_i\in \mathcal C\big(0,t; x, y_i;
\overline U\big)$ such that
$$\cl_t(x,y_i)\ge\int_0^tL(\gz_i(\theta),\gz_i'(\theta))\,d\theta-s\ez.$$
This implies that
$$\cl_t(x,y_i)\ge \cl_{t-s}(x,\gz_i(t-s))+\cl_s(\gz_i(t-s),y_i)-s\ez.$$
Hence it holds that
\begin{eqnarray}\label{e5.x8} \cl_s(\gz_i(t-s),y_i)\le   \cl_t(x,y_i)-\cl_{t-s}(x,\gz_i(t-s))+s\ez.\end{eqnarray}
From (\ref{e5.x8}) and \eqref{e5.x6},  we obtain
\begin{eqnarray}
 &&T^s u(\gamma_i(t-s))-u(\gamma_i(t-s))\nonumber\\
  &&\quad\geq  u(y_i)-\cl_s(\gamma_i(t-s),y_i)-u(\gamma_i(t-s)) \nonumber\\
 &&\quad\geq  \lf[u(y_i)-\cl_t(x,y_i)\right]-\left[u(\gamma_i(t-s))-\cl_{t-s}(x,\gamma_i(t-s))\right]-s\ez \nonumber\\
&&\quad\geq T^t u(x)-T^{t-s}u(x)- 2s\ez.\nonumber
\end{eqnarray}
Since $u$ satisfies the pointwise convexity criteria, we have
$$
T^{t-s}u(x)\leq\frac{s}{t}T^0 u(x)+\frac{t-s}{t}T^t u(x)=\frac{s}{t}  u(x)+\frac{t-s}{t}T^t u(x).
$$
Hence we obtain
\begin{eqnarray}\label{e5.x9}
\frac{T^s u(\gamma_i(t-s))-u(\gamma_i(t-s))}{s}
&&\ge\frac{T^t u(x)-u(x)}{t}-2\ez.
\end{eqnarray}
By the upper semicontinuity of $S^+u$,
there exists $r=r_\epsilon\in(0,\dist(y,\partial U))$ such that
\begin{eqnarray}\label{e5.x10}\sup_{z\in B(y,r)}S^{+}u(z)\le S^+u(y)+\ez.
\end{eqnarray}
Let $s_0>0$ be such that
\begin{equation}\label{e5.x10.1}
M\big(\frac{r}{8s_0}\big)\frac{r}8>
\sup_{z\in B(y,\frac{r}8)}\cl_t(x, z)+s\ez.
\end{equation}
We claim that for every $s\in(0,s_0)$ and sufficiently large $i$,
$\gamma_i(t-s)\in B(y,\frac{r}4)$. For, otherwise,  there exists
a sufficiently large $i$ such that $|y_i-y|\le \frac{r}8$, but $\gamma_i(t-s)\notin B(y,\frac{r}4)$.
In particular, $\gamma_i(t-s)\notin B(y_i,\frac{r}8)$.  By Lemma \ref{l3.x3} and (\ref{e5.x10.1}),
this yields
 $$ \cl_s(\gz_i(t-s),y_i)\ge M\lf(\frac{|y_i-\gz_i(t-s)|}s\r) |y_i-\gz_i(t-s)| \ge M(\frac{r}{8s})\frac{r}8
 >  \cl_t(x,y_i)+s\ez.$$
This contradicts to \eqref{e5.x8}.
From Lemma \ref{l5.x3} and Lemma \ref{l5.x4}, we know that
if $s\in(0,s_0)$ is sufficiently small and   $i$ is sufficiently large, then
\begin{equation}\label{e5.x10.2}
\frac{T^su(\gamma_i(t-s))-u(\gamma_i(t-s))}{s}\leq\|H(x,Du)\|_{L^{\infty}(B(y,r))}+\ez
=\sup_{z\in B(y,r)}S^{+}u(z)+\ez.
\end{equation}
Combining (\ref{e5.x10.2}) with \eqref{e5.x9} and \eqref{e5.x10}, we obtain
\begin{eqnarray}\label{e5.x11}
\frac{T^t u(x)-u(x)}{t}\leq S^{+}u(y)+4\ez.
\end{eqnarray}
This yields  \eqref{e5.x7}.
\end{proof}
\end{lem}

Now we are ready to prove (iv)$\Rightarrow$(i).

%\begin{prop}\label{p4,y1}Suppose $u\in\lip_{\loc}(U)$ is bounded and satisfies the pointwise convexity criterion. Then $u$ is an absolute  subminimizer in $U.$
\begin{proof}[Proof of   (iv)$\Rightarrow$(i). ]
We argue by contradiction.
Suppose that $u$ were not an  absolute subminimizer. Then there exist  $V\Subset U$ and $v\in\lip_{\loc}(V)$, such that $u\geq v$ in $V$, $u=v$ on $\partial V$, and
$$\lambda:=\big\|H(\cdot,Dv)\big\|_{L^{\infty}(V)}<\mu:=\big\|H(\cdot,Du)\big\|_{L^{\infty}(V)}.$$
By Lemma \ref{l5.x4}, this yields that
$$\lambda=\sup_{x\in V}S^{+}v(x)<\sup_{x\in V}S^{+}u(x)=\mu.$$
Select $l\in (\lambda, \mu)$ and define
$$E:=\Big\{x\in\overline{V}: S^{+}u(x)\geq l\Big\}.$$
Since $S^{+}u$ is upper semicontinuous, $E$ is compact and $E\cap V\not=\emptyset$.
It is readily seen that there exists $x_*\in E\cap V$ such that
\begin{eqnarray}\label{e5.x12}
u(x_*)-v(x_*)=\max_{E}(u-v)=:m.
\end{eqnarray}
Since $u$ satisfies the pointwise convexity criteria,
the map $t\mapsto T^t u(x_*)$ is convex on the interval $[0,\delta(x_*)]$ for some $\dz(x_*)>0$.
By Lemma \ref{l5.x3}, if $t>0$ is sufficiently small then
\begin{eqnarray}\label{e5.x13}
\frac{T^t v(x_*)-v(x_*)}{t}<l.
\end{eqnarray}
By Lemma \ref{l4.x2} and Lemma \ref{l3.x5} (iii), there exists
$y_*\in {B\big(x_*,\frac12\dist(x_*,\partial V)\big)}\subset V$ such that
\begin{eqnarray}\label{e5.x14}
T^t u(x_*)=\lim_{y_i\to y_*}\Big\{u(y_i)-\cl_t(x_*,y_i)\Big\}.
\end{eqnarray}
From Lemma \ref{l5.x5}, we  have
\begin{eqnarray}\label{e5.x15}S^{+}u(y_*)\geq\frac{T^tu(x_*)-u(x_*)}{t}\geq S^{+}u(x_*)\geq l.\end{eqnarray}
Hence $y_*\in E.$
On the other hand, by \eqref{e5.x13}, \eqref{e5.x15} and \eqref{e5.x14}, we have
\begin{eqnarray*}
v(y_*)&&=\lim_{i\rightarrow \infty}v(y_i)\leq T^t v(x_*)+\lim_{i\rightarrow\infty}\cl_t(x_*,y_i)\\
&&<lt+v(x_*)+\lim_{i\rightarrow\infty}\cl_t(x_*,y_i)\\
%&&\leq tS^{+}u(x)+v(x)+ \cl_t(x,y) \\
&&\leq T^t u(x_*)-u(x_*)+v(x_*) +\lim_{i\rightarrow\infty}\cl_t(x_*,y_i) \\
&&=\lim_{i\rightarrow\infty} u(y_i)-u(x_*)+v(x_*)+\lim_{i\rightarrow\infty}\big(T^tu(x_*)-[u(y_i)-\cl_t(x_*, y_i)]\big)\\
&&=u(y_*)-u(x_*)+v(x_*).
\end{eqnarray*}
Hence
$$u(y_*)-v(y_*)>  u(x_*)-v(x_*)=m.$$
This contradicts to $y_*\in E$ and  \eqref{e5.x12}. The proof is now complete.
\end{proof}

\section{Proof of Theorem  1.1 and Corollary 1.2}\label{s6}

In this section, we will prove both Theorem 1.1 and Corollary 1.2.
In order to prove Theorem 1.1, we need two lemmas.
The first one is a generalization of the stationary point lemma by
\cite {acjs}, where $H=H(p)$  is independent of $x$.
The proof follows essentially from  \cite[Lemma 2.7]{acjs}
with the action function $tL\lf(\frac{y-x}t\r)$ in \cite{acjs} replaced by the general action function
$\cl_t(x,y)$ defined in this paper. We leave the details to the readers.
\begin{lem} \label{l6.x1}
Assume that $H$ satisfies (A1), (A2)$_\weak$, (A3), and (\ref{u-coercive}).
 Let $\alpha,r>0,$ and $f,g\in C(\overline{U}_r)$ with $\osc_{U}f,\osc_{U}g\leq\alpha.$
 Assume that for some $0<t<\eta_0(\alpha,r)$,
\begin{eqnarray}\label{e6.x1}
T^t f(x)+T_t f(x)-2f(x)\geq0 \geq T^t g(x)+T_t g(x)-2g(x), \ \forall\ x\in U_{2r}.
\end{eqnarray}
Then either
\begin{eqnarray}\label{e6.x2}
\max_{\overline{U}_r}(f-g)=\max_{\overline{U}_r\backslash U_{2r}}(f-g),
\end{eqnarray}
or there exists $x\in U_{2r}$ such that
\begin{eqnarray}\label{e6.x3}
f(x)=T^t f(x)=T_t f(x)\;\;\textrm{and}\;\;g(x)=T^t g(x)=T_t g(x).
\end{eqnarray}
\end{lem}

%\begin{proof}
%Consider the case that \eqref{e6.x2} does not hold. Then the set
%$$E:=\left\{x\in\overline{U}_r:(f-g)(x)=\max_{\overline{U}_r}(f-g)\right\}$$
%is nonempty, closed, and contained in $U_{2r}.$ Define the set
%$$F:=\left\{x\in E:f(x)=\max_{E}f\right\}$$
%which is nonempty, closed, and contained in $E\subseteq U_{2r}.$ Select a point $x\in F.$
%Since $x\in E,$ we have
%\begin{eqnarray}\label{l5,z4}
%g(y)-g(x)\geq f(y)-f(x)\;\;\textrm{for every}\;\;y\in U_r.
%\end{eqnarray}
%Using \eqref{e6.x1}, \eqref{l5,z4} and \eqref{e6.x1} again, in that order, together with $x\in U_{2r},$ and $t<t_0(\alpha,r),$ we deduce
%\begin{eqnarray}\label{l5,z5}
%T^t f(x)-f(x)\geq f(x)-T_t f(x)\geq g(x)-T_t g(x)\geq T^t g(x)-g(x).
%\end{eqnarray}
%Applying \eqref{l5,z4} again, we also get $T^t f(x)-f(x)\leq T^t g(x)-g(x),$ and thus we must have equality in every inequality of \eqref{l5,z5}. That is
%\begin{eqnarray}\label{l5,z6}
%T^t f(x)-f(x)= f(x)-T_t f(x)= g(x)-T_t g(x)= T^t g(x)-g(x).
%\end{eqnarray}
%Select a point $z\in \overline{B}(x,r)\subseteq U_r$ such that
%$$T^t f(x)=f(z)-\cl_t(x,z).$$
%Combining \eqref{l5,z6}, we observe that
%\begin{eqnarray*}
%g(z)&&\leq\cl_t(x,z)+T^t g(x)\\
%&&=\cl_t(x,z)+T^t f(x)+g(x)-f(x)\\
%&&=f(z)+g(x)-f(x).
%\end{eqnarray*}
%Therefore, $f(z)-g(z)\geq f(x)-g(x)$ and thus $z\in E.$ Recalling that $x\in F,$ we see that
%$$f(z)\leq f(x)\leq T^t f(x)=f(z)-\cl_t(x,z)\leq f(z).$$
%Thus $f(x)=T^t f(x).$
%Recalling \eqref{l5,z6}, we have \eqref{e6.x3}.
%\end{proof}
%\end{lem}

The second one is a patching lemma.

\begin{lem}\label{l6.x2}  Assume that $H$ satisfies (A1), (A2), (A3), and (\ref{u-coercive}).
If $u\in \lip_{\loc}(U)\cap C(\overline{U})$ satisfies the pointwise convexity criteria,
then there exists a family of functions $\left\{u_{\sz}\right\}_{\sz>0}$ such that
\begin{enumerate}
  \item[(i)] for every $\sz>0$, $u_{\sz} \in\lip_{\loc}(U)\cap C(\overline{U}) $, $u_{\sz}=u$ on $\partial U$ and $u_{\sz}\leq u$ on $\overline{U}$.
  \item[(ii)] $u_{\sz}\rightarrow u$ uniformly on $\overline{U}$ as $\sz\to 0$.
  \item[(iii)] for every $\sz>0,$   $u_{\sz}$ satisfies the pointwise convexity criteria.
  \item[(iv)] for every $x\in U$ and $\sz>0$, $S^{+}u_{\sz}(x)\geq\sz$.
\end{enumerate}
\end{lem}

The proof of Lemma \ref{l6.x2} will be given by in the subsection 6.1, which is based on careful analysis of
properties of $H(x,p)$, $d_\lz(x,y)$, and $\cl_t(x,y)$. We would like to point out that Lemma \ref{l6.x2}
has been proven  by \cite[Lemma 5.1]{acjs} when $H=H(p)$ has no $x$-dependence.

Since Hamiltonian functions in Theorem 1.1 are assumed to satisfy (A1), (A2), (A3), but not (\ref{u-coercive}),
we can't directly apply Theorem \ref{t1.x4},  Lemma \ref{l6.x1} and Lemma \ref{l6.x2} to prove Theorem 1.1.
To overcome this, we need an elementary lemma.

\begin{lem}\label{l6.x2.0} (i) For any nonnegative $H\in\lsc(\overline U\times\rn)$, $u\in W^{1,\infty}_{\rm{loc}}(U)$
is an absolute subminimizer (or super minimizer)
for $H$ iff it is an absolute subminimizer (or super minimizer) for $H^a$ for any $a\in (1, +\infty)$.\\
(ii) If $H:\overline U\times\rn\to\mathbb R_+$ satisfies (A1), (A2), and (A3), then for any $a\in (1,+\infty)$,
$H^a$ satisfies (A1), (A2), (A3), and (\ref{u-coercive}).
\end{lem}
\begin{proof} (i) is obvious from the definition of absolute subminimizer (or superminimizer).
To see (ii), first observe that
(A3) yields that
$$B(0, r_1)\subset \bigcup_{x\in\overline U}\big\{p:\ H(x,p)\le 2\big\}\subset B(0, R_3).$$
This, combined with $H(x,0)=0$ and convexity of $H(x,\cdot)$ for $x\in\overline U$, further
implies that
$$H(x, p)\ge 2, \ \forall\ x\in\overline U, \ p\in\partial B(0, R_3).$$
Applying the convexity of $H(x,\cdot)$ again, we obtain that
$$2\le H\big(x, R_3\frac{p}{|p|}\big)\le \big(1-\frac{R_3}{|p|}\big)H(x,0) + \frac{R_3}{|p|}H(x,p)=
\frac{R_3}{|p|}H(x,p)$$
for any $x\in\overline U$ and $p\in \rn\setminus B(0, R_3)$.
Therefore we have that
\begin{equation}\label{e6.x101}
H(x,p)\ge \frac{2}{R_3} |p|, \ \forall x\in\overline U, \ p\in\rn \ {\rm{with}}\ |p|\ge R_3.
\end{equation}
This yields that for any $a\in (1,+\infty)$,
\begin{equation}\label{e6.x102}
H^a(x,p)\ge \big(\frac{2}{R_3}\big)^a |p|^a, \ \forall x\in\overline U, \ p\in\rn \ {\rm{with}}\ |p|\ge R_3.
\end{equation}
Hence $H^a$ satisfies (\ref{u-coercive}). It is easy to see that $H^a$ also satisfies (A1), (A2), and (A3).
\end{proof}

\begin{proof}[Proof of Theorem 1.1]
It follows from Lemma \ref{l6.x2.0} that, after replacing $H$ by $H^a$ for any $a>1$, we may simply
assume that $H$ satisfies, in addition to (A1), (A2), and (A3), the uniform coercivity condition (\ref{u-coercive}).
Since $u \in C(\overline{U})$ is an absolute subminimizer,
it follows from Theorem \ref{t1.x4} that $u$ satisfies the pointwise convexity criteria.
Let $\{u_{\sz}\}_{\sz>0}\subseteq\lip_{\loc}(U)\cap C(\overline{U})$
be a family of approximation functions of $u$ given by Lemma \ref{l6.x2}.
From the equivalence between pointwise convexity and convexity property,
for any $r>0$ there exists $\delta=\delta(r, \sigma)>0$ such that
the map $t\mapsto T^t u_\sigma(x)$ is convex on $[0,\delta]$ for all $x\in U_{2r}$.
It follows that when $t>0$ is sufficient small and $0<h<t$, it holds
\begin{eqnarray*}
&&\frac{ T^{2t} u_{\sz}(x)- T^t u_{\sz}(x)}{t}\ge \frac{T^{t+h}u_\sigma(x)-T^t u_\sigma(x)}{h}\\
&&\ge \frac{T^tu_\sigma(x)-u_\sigma(x)}{t}\ge \frac{T^hu_\sigma(x)-u_\sigma(x)}{h},
\ \forall\ x\in U_{2r}.
\end{eqnarray*}
Sending $h$ to zero, this implies
\begin{equation}\label{e6.x4.0}
T^{2t}u_\sigma(x)+u_\sigma(x)-2T^tu_\sigma(x)\ge 0, \ \forall\ x\in U_{2r},
\end{equation}
and
\begin{equation}\label{e6.x4}
\frac{ T^{2t} u_{\sz}(x)- T^t u_{\sz}(x)}{t}\ge S^{+}T^t u_{\sz}(x)\geq S^{+}u_{\sz}(x)\ge \sigma,
\ \forall\ x\in U_{2r}.
 \end{equation}

On the other hand, since $v \in C(\overline{U})$ is an absolute superminimizer for $H$,
Remark \ref{r2.x3}  implies that $\hat v=-v  $ is an absolute subminimizer for $\hat H(x,p)=H(x,-p)$.
By Theorem 1.5, $v$ satisfies the convexity criteria for $\hat H$.
Let $\{\hat v _{\sz}\}_{\sz>0}\subseteq\lip_{\loc}(U)\cap C(\overline{U})$ be a family of
approximation functions to $\hat v$ given by Lemma \ref{l6.x2} for $\hat H$.
Then we may assume that $t\mapsto {\hat T}^t\hat v_\sigma(x)$ is convex on $[0,\delta)$
for all $x\in U_{2r}$. As above, we obtain that when $t>0$ is sufficiently small, it holds
\begin{equation}\label{e6.x5.0}
{\hat T}^{2t}{\hat v_\sigma}(x)+\hat v_\sigma(x)-2{\hat T}^t {\hat v}_\sigma(x)\ge 0, \ \forall \ x\in U_{2r},
\end{equation}
and
\begin{equation}\label{e6.x5.1}
\frac{ \hat T^{2t} \hat v_{\sz}(x)- \hat T^t \hat v_{\sz}(x)}{t}\ge S^{+}\hat T^t \hat v_{\sz}(x)\geq S^{+}\hat v_{\sz}(x)\ge
\sigma, \ \forall \ x\in U_{2r}.
\end{equation}
Set $v_\sz=-\hat v_\sz$. Then $\hat T^{s} \hat v_{\sz}(x)=  -T_{ s} v_\sz(x)$ for any $s>0$ and $x\in U$.
It follows from (\ref{e6.x5.0}) and (\ref{e6.x5.1}) that
\begin{equation}\label{e6.x5.2}
T_{2t}v_\sz(x)+v_\sigma(x)-2T_t v_\sz(x)\le 0, \ \forall\ x\in U_{2r},
\end{equation}
and
\begin{equation}\label{e6.x5.3}\frac{  T_{2t}   v_{\sz}(x)-   T_t   v_{\sz}(x)}{t} \le -\sigma,
\ \forall\ x\in U_{2r}.
\end{equation}
By Lemma \ref{l4.x4}, $T^tT^t u_{\sz}(x)=T^{2t}u_\sz(x)$, and by the definition of $T^t$ and $T_t$
we have that $T_tT^t u_{\sz}(x)\ge u_{\sz}(x)$. Hence  (\ref{e6.x4.0}) yields that
$$T^tT^t u_{\sz}(x)+T_tT^t u_{\sz}(x)-2T^t u_{\sz}(x)\geq 0, \ \forall\ x\in U_{2r}.$$
Similarly, by (\ref{e6.x5.2}) we have
$$T_tT_t v_{\sz}(x)+T^tT_t v_{\sz}(x) -2T_t v_{\sz}(x)\leq0, \ \forall\ x\in U_{2r}.$$
Applying Lemma \ref{l6.x1} with $f=T^t u_{\sz}(x)$ and $g=T_t v_{\sz}(x)$,
we conclude that either
\begin{eqnarray}\label{e6.x6}
\max_{\overline{U_r}}(T^t u_{\sz}-T_t v_{\sz})=\max_{\overline{U_r}\backslash U_{2r}}(T^t u_{\sz}-T_t v_\sz) ,
\end{eqnarray}
or there exists a point $x_0\in U_{2r}$ such that
\begin{eqnarray}\label{e6.x7}
T^t u_{\sz}(x_0)=T^{2t} u_{\sz}(x_0)=T_t T^t u_{\sz}(x_0), \  \ T_t v_{\sz}(x_0)=T^t T_t v_{\sz}(x_0)
=T_{2t} v_{\sz}(x_0).
\end{eqnarray}
It follows from \eqref{e6.x4} and \eqref{e6.x5.3} that \eqref{e6.x7} doesn't hold.
Hence \eqref{e6.x6} holds. Sending $t\to0, \sz\to0, r\to0$ in the order,
(\ref{e6.x6}) implies, with the help of Remark \ref{r4.x3}, that
$$\max_{\overline U}(u-v)=\max_{\partial U}(u-v). $$
In particular, there is a unique absolute minimizer with boundary value $g\in C(\partial U)$.
The proof is complete.
\end{proof}

\begin{proof}[Proof of Corollary 1.2]
It suffices to verify that every Hamiltonian function
$H(x,p)$ in Corollary1.2 satisfies   (A1), (A2)  and (A3) of Theorem 1.1.

It is obvious that $H(x,p)$ satisfies (A1).
 (A2) also holds, since
$$\displaystyle\bigcup_{x\in \overline U}\Big\{p: H(x,p)=0\Big\}=\big\{0\big\}.$$
To see that $H$ satisfies (A3),  first observe that since for each $x\in\overline U$,
$H(x,\cdot)$ is convex and $H(x,p)>0$ for each $p\in \mathbb S^{n-1}$,
we have that $ H(x, tp)$ is strictly increasing with respect to $t\in [0,+\infty)$
and $\displaystyle\lim_{t\to\fz}H(x,tp)=\fz$. Hence for any $\lambda>0$ and $x\in\overline U$,
$$E(x,\lambda):=\big\{p\in \rn: \ H(x,p)<\lambda\big\}$$ is a bounded convex set, which has
a non-empty interior part and contains $0$. For each $\lambda>0$ and $x\in\overline U$, define
$$r(x,\lambda):=\sup\Big\{r>0\ | \ B(0,r)\subset E(x,\lambda)\Big\},
\ R(x,\lambda):=\inf\Big\{R>0\ | \ E(x,\lambda)\subset B(0,R)\Big\}.
$$
It is well-known that
$$0<r(x,\lambda)\le R(x,\lambda)<+\infty, \ \forall\ \lambda>0, \ x\in\overline U.$$
Now we need\\
\noindent{\it Claim}. {\it For any $\lambda>0$, i) $r(\cdot,\lambda)$ is lower semicontinuous
on $\overline U$; and ii) $R(\cdot,\lambda)$ is upper semicontinuous  on $\overline U$.}

To see it,  let $\{x_i\}\subset\overline U$ and $x_0\in\overline U$ be such that
$\displaystyle\lim_{i\rightarrow\infty} x_i=x_0$.  For a given $\lambda>0$,
we want to show $\displaystyle r_0:=\liminf_{i} r(x_i,\lambda)\ge r(x_0,\lambda)$.
We may assume that there exist  a subsequence $\{x_{i_k}\}\subset\{x_i\}$
 such that $\displaystyle\lim_{k} r(x_{i_k},\lambda)=r_0$.
By the definition of $r(x_{i_k}, \lambda)$, there exists $p_k\in\rn\setminus\{0\}$ such that
$|p_k|=r(x_{i_k}, \lambda) \ {\rm{and}}\ H(x_{i_k}, p_k)=\lambda.$ Assume
that $\displaystyle\lim_{k} p_k=p_0$. Since $H\in C(\overline U\times \rn)$, it follows
that $|p_0|=r_0$ and $H(x_0, p_0)=\lambda$. Since $H(x_0,\cdot)$ is convex and $H(x_0,0)=0$, it follows
that for any $t>1$, $H(x_0, tp_0)\ge t\lambda>\lambda$. By the definition of
$r(x_0,\lambda)$ we then have $r(x_0, \lambda)\le r_0.$ Similarly, we can show
that $\displaystyle R_0:=\limsup_{i} R(x_i,\lambda)\le  R(x_0,\lambda)$.

It follows from the Claim that for any $\lambda>0$, there exist $x_1, x_2\in \overline U$ such
that
$$r_\lambda:=r(x_1,\lambda)=\min_{x\in\overline U} r(x,\lambda),
\ R_\lambda:=R(x_2,\lambda)=\max_{x\in\overline U} R(x,\lambda).$$
It is clear that $0<r_\lambda<R_\lambda$, and
$$B(0, r_\lambda)\subset\bigcap_{x\in\overline U} E(x,\lambda)
\subset \bigcup_{x\in\overline U} E(x,\lambda)\subset B(0, R_\lambda).$$
To see $\displaystyle\lim_{\lambda\rightarrow\infty}r_\lambda=+\infty$. We argue by contradiction.
For, otherwise, there exists $\lambda_i\rightarrow\infty$ such that
$\displaystyle\lim_{\lambda_i\rightarrow\infty} r_{\lambda_i}=r_\infty<+\infty$.
This implies that there exist $x_i\in\overline U$ and $p_i\in \rn$ with $|p_i|=r_{\lambda_i}$
such that $H(x_i, p_i)=\lambda_i$. After passing to a subsequence, we may assume
that $x_i\rightarrow x_\infty$ for some $x_\infty\in\overline U$,
and $p_i\rightarrow p_\infty$ for some $p_\infty\in \rn$ with $|p_\infty|=r_\infty$. Since $H\in C(\overline U\times\rn)$,
we arrive at $H(x_\infty, p_\infty)=+\infty$. This is impossible.   Hence
$H$ satisfies (A3) of Theorem 1.1.
\end{proof}

\subsection{Proof of the patching lemma}

In order to prove  Lemma \ref{l6.x2}, we need the following  approximation result.
We point out that this is the only place
of the whole paper that we need to use (A2), instead of (A2)$_{\rm{weak}}$.

\begin{lem}\label{l6.x3} Assume that $H(x,p)$ satisfies  (A1), (A2),  (A3), and (\ref{u-coercive}).
For any $\epsilon>0$, there exists $\lambda=\lambda(\epsilon)>0 $ such that
if $u,v\in C(\overline{U})$ satisfy $u=v$ on $\partial U$, and
\begin{eqnarray}\label{e6.x8}
\sup_{x\in U}\Big\{S^{+}u(x)+S^{+}v(x)\Big\}\leq \lambda,
\end{eqnarray}
then
\begin{equation}\label{e6.x8.1}
\max_{\overline{U}}|u-v|\leq\epsilon.
\end{equation}
\end{lem}
\begin{proof} From (A2) and (A3), there exists a hyperplane $P\subset\rn$
such that
$$\displaystyle\Sigma_0:=\bigcup_{x\in \overline U}\big\{p: H(x,p)=0\big\}\subset P$$
is a compact set containing $0$. Let $q\in\rn$ be a unit normal of $P$. It follows that
$p\cdot q=0$ for all  $p\in \Sigma_0$. Hence
\begin{equation}\label{e6.x9.0}
L_0(x, \pm q)=0, \ \forall\ x\in\overline U.
\end{equation}
We claim that for any $\epsilon>0$, there exists $\lz=\lambda(\epsilon)\in(0,1)$  such that
\begin{eqnarray}\label{e6.x9}
\max\Big\{L_{\lambda}(x,q), \ L_\lambda(x,-q)\Big\}<\frac{\epsilon}{2\diam(U)},\ \forall\ x\in \overline U.
\end{eqnarray}
Assume (\ref{e6.x9}) for the moment. For any $x\in U$, set $r_{x}:=\inf\big\{r>0\ |\ x+rq\notin U\big\}$.
Then $r_{x}\leq \diam(U)$ and $y_{x}:=x\pm r_xq\in \partial U.$
By \eqref{e6.x8} and Lemma \ref{l5.x4}, we have
$$
\sup_{z\in U}S^+u(z)=\big\|H(x,Du)\big\|_{L^{\infty}(U)}<\lambda.
$$
Hence, by Lemma \ref{l3.x9} we have
$$
u(y_x)-u(x) \leq d_\lz(x,y_x), \ \forall x\in U.$$
Let $\gz(t)=x+t(y_x-x)=x+tr_x q$, $t\in[0,\,1]$, be the line segment joining $x$ to $y_x$. Then $\gamma([0,1))\subset U$ and hence (\ref{e6.x9}) implies that
$$d_\lz(x,y_x)\le \int_0^1L_\lz(\gz(t),\gamma'(t))\,dt\le r_x \sup_{t\in[0,1]}L_{\lambda}(\gz(t),q)
\le r_x\frac{\epsilon}{2\diam(U)}\le\frac{\ez}2.$$
In particular, it holds that $u(y_x)-u(x)\leq\frac{\ez}2$ for $x\in U$. Switching the role of $x$ and $y_x$ and
applying (\ref{e6.x9}) with $-q$, we also obtain $u(x)-u(y_x)\le \frac{\epsilon}2$ for $x\in U$.
Hence $|u(x)-u(y_x)|\le \frac{\epsilon}2$ for $x\in U$.
Similarly, it also holds that $|v(x)-v(y_x)|\leq\frac{\ez}2$ for $x\in U$.
Since $u(y_x)=v(y_x)$ for $x\in U$, we obtain (\ref{e6.x8.1}).

We prove (\ref{e6.x9}) by contradiction. It suffices to show (\ref{e6.x9}) when the left hand side
is replaced by $L_\lambda(x,q)$.  Suppose it were false. Then there exists $\epsilon_0>0$ such that
for $\lambda_i\rightarrow 0^+$ there exists $x_i\in\overline U$ so that
$$L_{\lambda_i}(x_i, q)=\max_{\{H(x_i, p)\le\lambda_i\}} p\cdot q>\frac{\epsilon_0}{2\diam U}.$$
Thus there exists $p_i\in\rn$ such that
$$H(x_i, p_i)\le \lambda_i \ {\rm{and}}\ \ p_i\cdot q>\frac{\epsilon_0}{4\diam U}.$$
This and (A3) imply
$$\frac{\epsilon_0}{4\diam U}\le |p_i|\le R_{1}.$$
Assume $x_i\rightarrow \bar x$ and $p_i\rightarrow \bar p$ as $i\rightarrow\infty$.
Then $\bar x\in\overline U$, and
$$\frac{\epsilon_0}{4\diam U}\le |\bar p|\le R_1.$$
By the lower semicontinuity of $H$, we have that $H(\bar x, \bar p)=0$.
But, on the other hand, it holds
$$\bar p\cdot q>\frac{\epsilon_0}{4\diam U}.$$
This contradicts to (\ref{e6.x9.0}).  The proof is now complete. \end{proof}

\begin{proof}[Proof of Lemma \ref{l6.x2}]
For any $\sz>0,$ set
\begin{eqnarray*}
V_\sz:=\Big\{x\in U\ |\ S^{+}u(x)<\sz\Big\}.
\end{eqnarray*}
By the upper semicontinuity of $S^+u$, $V_\sz$ is open. For $x\in\overline{ V_\sz}$,
let ${\mathcal P}(x)$ denote the set of finite ordered lists
$[x_0=x, x_1, \ldots, x_N]\subset \overline {V_\sz}$  such that  \\
(i) $x_N\in\partial V_\sz$, and\\
(ii) for each $i=0,\ldots,N-1$,
there is a $d_\sz$-length minimizing geodesic curve $\gz_{x_i,x_{i+1}}$ joining $x_{i}$
to $x_{i+1}$,  with  $\gz_{x_i,x_{i+1}}\setminus\{x_i,x_{i+1}\}\subset V_\sz.$ Here
$N$ is a positive integer.

Define  $v_\sz : \overline{V}_{\sz} \rightarrow\rr$ by
\begin{eqnarray} \label{e6.x11}
v_\sz(x) = \sup \left\{ u(x_N) - \sum_{i = 0}^{N-1} d_{\sz}(x_{i},x_{i+1})\ \big|\ [x=x_0, \ldots, x_N]\in {\mathcal P}(x)\right\},
\ x\in\overline{V_\sigma}, \end{eqnarray}
and    $u_\sz : \overline U\rightarrow\rr$ by
\begin{eqnarray*}
u_{\sz}(x) = \begin{cases}
v_\sz(x) & \textrm{if}\; x \in \overline{V_\sz}, \\
u(x) & \textrm{if}\; x \in \overline{U} \setminus \overline V_{\sz}.
\end{cases}
\end{eqnarray*}
We want to show that $\displaystyle \{u_\sz\}_{\sz>0}$ satisfies all the desired properties.
It is divided into two claims.

\smallskip
\noindent\emph{Claim 1.
$v_\sz \leq u$ in $V_\sz$, $v_\sz = u$ on $\partial V_\sz $, and $v_\sz\in C(\overline{V_\sz}).$}

\smallskip
\noindent {\it Proof of Claim 1.} For $x\in V_\sigma$, let
$[x_0=x, ..., x_N]\in \cp(x).$  By Lemma \ref{l5.x4}, we have
\begin{eqnarray*}
\|H( \cdot,Du)\|_{L^{\infty}(V_{\sz})}=\sup_{x\in V_\sz}S^+u(x)\le\sz.
\end{eqnarray*}
Hence by Lemma \ref{l3.x9} and Remark \ref{r3.x10}, we  have
\begin{eqnarray*}
u(x_N) - u(x) = \sum_{i=0}^{N-1} (u(x_{i+1}) - u(x_i)) \leq \sum_{i = 0}^{N-1} d_\sz(x_{i}, x_{i+1}),
\end{eqnarray*}
this implies that $v_\sz(x) \leq u(x).$
When $x\in\partial V_\sz$,  since $[x]\in \cp(x)$, it follows that  $v_\sz (x) \geq u(x)$.
Hence $v_\sz (x)= u(x)$ on $\partial V_\sz$.

To show that $v_\sz$ is continuous on $\overline{V_\sz}$,
let $x,y\in V_\sz$  be such that
there exists a $d_\lz$-length minimizing geodesic curve $\gz_{x,y}\subset V_\sz$ joining $x$ to $y$,
which exists as long as $|x-y|$ is sufficiently small.
 For any $\epsilon>0,$ choose $[x_0=x,\ldots, x_N]\in \cp(x)$ such that
 \begin{equation}\label{e6.x12}
v_\sz(x)-\epsilon\le u(x_N) - \sum_{i = 0}^{N-1} d_\sz(x_{i},x_{i+1}).
\end{equation}
Then
 $[y_0,y_1, \ldots, y_{N+1}]:=[y,x,\ldots, x_N]\in \cp(y)$. Hence,
\begin{eqnarray*}
v_\sz(y)
 \ge u(x_N) -d_\sz(y,x) -\sum_{i = 0}^{N-1} d_\sz(x_{i},x_{i+1})
  \geq v_\sz(x)-\epsilon -d_\sz(y,x).
\end{eqnarray*}
Since $\epsilon>0$ is arbitrary, we conclude that
\begin{equation}\label{e6.x13}
v_\sz(x)-v_\sz(y)\leq d_\sz(y,x)\le R_\lz d_U(x,y),
\end{equation}
provided $x, y\in V_\sigma$ has sufficiently small $|x-y|$.
It follows from (\ref{e6.x13}) that $v_\sz$ is continuous on $V_\sz.$
Since $u\in C(\overline{U})$, $v_\sz=u$ on $\partial V_\sz$ and $v_\sz\le u$ on $\overline{V_\sz},$
the continuity of $v_\sz$ on $\overline{V_\sz} $ follows from the lower semicontinuity of
 $v_\sz$ on $\partial V_\sz$  when approaching from $V_\sz.$  In fact, for
 $y\in\partial V_\sz$, let $y_j\in V_\sz$ such that $y_j\rightarrow y$ as $j\rightarrow \infty.$
 Let $x_j\in \partial V_\sz$ such that
 $$d_\sz(y_j,x_j)=\min_{z\in \partial V_\sz} d_\sz(y_j,z).$$
Then $x_j\rightarrow y$. For each $j$, there exists a $d_\sz$-length minimizing geodesic curve $\gz_{y_j,x_j}$
joining $y_j$ to $x_j$, with $\gz_{y_j,x_j}\setminus\{x_j\}  \subset V_\sz$.
From \eqref{e6.x13} and $v_\sz=u$ on $\partial V_\sz,$ we have
$$u(x_j)=v_\sz(x_j)\le v_\sz(y_j)+d_\sz(y_j,x_j),$$
this yields  that $\displaystyle\liminf_{j\rightarrow\infty} v_\sz(y_j)\ge u(y) = v(y) $.

\medskip
\noindent \noindent\emph{Claim 2. For any $x \in V_\sz$ and sufficiently small $t>0,$
\begin{equation}\label{e6.x14}
T^tu_{\sz}(x) -u_{\sz}(x) = \sz t;
\end{equation}
and, for any $x\in U\backslash V_{\sz}$ and sufficiently small $t>0$,
\begin{equation}\label{e6.xx14}T^t u_{\sz}(x)=T^t u(x).\end{equation}}

\noindent{\it Proof of Claim 2.}
To see \eqref{e6.x14}, for $x \in V_\sz$ let $0 < r < \dist(x, \partial V_\sz)$ be
such that $B(x,r)\subset B_\sz (x, R_\sz r)\subset V_\sz$.
  By Lemma \ref{l4.x2}, there exists  $t_0 = t_0(\osc_U u_{\sz}, r)>0$
  such that if $0<t<t_0$ then
\begin{eqnarray*}
T^t u_{\sz}(x) = \sup_{y\in  B(x,r)} \Big\{v_\sz(y) - \cl_t\left(x,y \right) \Big\}.
\end{eqnarray*}
For each $y\in B(x,r)$,  there is a $d_\lz$-length minimizing geodesic curve $\gz_{x,y}\subset V_\sz$ joining $x$
to $y$. Hence \eqref{e6.x13} holds for any $y\in B(x,r)$. This, combined with Lemma \ref{l3.x2} (ii), implies that
\begin{eqnarray}\label{e6.x15}
T^t u_{\sz}(x)%&&=\sup_{y\in  B(x,r)}\left[ v_\sz(y) - L_t\left(x,y\right)\right]\\
&&\leq \sup_{y\in  B(x,r)} \Big\{ d_\sz(x,y) + v_\sz(x) - \cl_t \left(x,y \right)\Big\}
 \leq v_\sz(x) + \sz t
 = u_\sz(x)+\sz t .
\end{eqnarray}
On the other hand, we can choose $t>0$ so small that Theorem \ref{t3.x6} (ii) and (iii)
imply that $\partial E^{<t}_{ \sz} (x)   \subset B(x,r)$.
Fix a small $\epsilon>0,$ and select $[x_0=x,\ldots,x_{N}]\in\cp{(x)}$
such that \eqref{e6.x12} holds.
It is clear that there is a point $y_*\in \gz_{x_{j_*},x_{j_*+1}}\cap\partial E_{\sz}^{<t}(x) $ for some $j_*\in \{0,1,\ldots,N-1\}.$
Then $[y_*,x_{j_*+1},\ldots, x_{N}]\in\cp{(y_*)}$ and it follows that
$$
v_{\sz}(y_*)\geq u(x_{N})-\sum_{i=j_*+1}^{N-1}d_{\sz}(x_{i},x_{i+1})-d_{\sz}(y_*,x_{j_*+1}).
$$
From this inequality and \eqref{e6.x12}, we deduce
$$
v_{\sz}(y_*)-v_{\sz}(x)+\epsilon\geq\sum_{i=0}^{j_*}d_{\sz}(x_{i},x_{i+1})-d_{\sz}(y_*,x_{j_*+1}).
$$
Since $y_*\in \gamma_{x_{j_*}, x_{j_*+1}}$, which is $d_\sz$-length minimizing,  it follows
that
$$d_{\sz}(x_{j_*},x_{j_*+1})-d_{\sz}(y_*,x_{j_*+1})=d_{\sz}(x_{j_*},y_*).$$
Thus we have
\begin{eqnarray*}
v_{\sz}(y_*)-v_{\sz}(x)+\epsilon&&\ge
d_{\sz}(x_{j_*},y_*)+d_{\sz}(x_{j_*-1},x_{j_*})+\cdots+d_{\sz}(x,x_1) \geq d_{\sz}(x,y_*).
\end{eqnarray*}
From Theorem \ref{t3.x6}, we have that $y_*\in \partial E_{\sz}^{<t}(x) \subseteq E_{\sz}^t(x) $.
Hence we obtain
\begin{eqnarray}\label{e6.x16}
T^t u_{\sz}(x)-u_{\sz}(x)+\epsilon&&\geq v_{\sz}(y_*)-v_{\sz}(x)-\cl_t(x,y_*)+\epsilon\nonumber\\
&&\geq d_{\sz}(x,y_*)-\cl_t(x,y_*)=\sz t.
\end{eqnarray}
Since $\epsilon>0$ is arbitrary, \eqref{e6.x14} follows from (\ref{e6.x15}) and (\ref{e6.x16}).

For \eqref{e6.xx14},  since $T^t u\geq T^t u_{\sz}$ follows from $u\geq u_{\sz}$,
it suffices to prove $T^t u(x)\le T^t u_{\sz}(x)$ for $x\in U\backslash V_{\sz}$.
For $x\in U\backslash V_{\sz}$, since $u$ enjoys the pointwise convexity criteria,
the map $t\mapsto T^t u(x)$ is convex on the interval $[0,\delta(x)] $ for some  $\delta(x)>0$.
Choose $0<r<\dist(x,\partial U)$ and  $0<t<\min\left\{\eta_0\left(\osc_{U}u_{\sz},r\right),\delta(x)\right\}.$
There exist $y\in \overline{B(x,r)}$ and $y_i\in B(x,r)$ such that $y_i\rightarrow y$ as $i\rightarrow\infty$, and
$$T^t u(x)=\lim_{y_i\to y}\Big\{u(y_i)-\cl_t(x,y_i)\Big\}.$$
By the increasing slope estimate \eqref{e5.x7}, we have
$$S^{+}u(y)\geq\frac{T^t u(x)-u(x)}{t}\geq S^{+} u(x)\ge\sigma.$$
This implies that $y\in U\backslash V_{\sz}.$ Since $u_{\sz}(y)=u(y)$ and both $u_\sz$ and $u$
are continuous at $y$,  we have
\begin{eqnarray*}
T^t u(x)&&=\lim_{y_i\to y}\Big\{u(y_i)-\cl_t(x,y_i)\Big\}\\
&&\le \limsup_{y_i\to y}\Big\{u_\sz(y_i)-\cl_t(x,y_i)\Big\}+
\lim_{y_i\to y}\Big\{u_\sz(y_i)-u ( y_i )\Big\} \\
&&= \limsup_{y_i\to y}\Big\{u_\sz(y_i)-\cl_t(x,y_i)\Big\}
\leq T^t u_{\sz}(x).
\end{eqnarray*}
This clearly implies (\ref{e6.xx14}).

For $ x\in V_\sz$,  \eqref{e6.x14} implies that $S^+u_\sz=\sz$;
while for $x\in U\setminus V_\sz$, (\ref{e6.xx14}) implies that $S^{+}u_{\sz}(x)=S^+u(x)\ge\sz$. This yields (iv).
From (\ref{e6.x13}) and Lemma \ref{l3.x9} (ii), we have that $v_\sigma\in {\rm{Lip}}_{\rm{loc}}(V_\sigma)$ and
$$H(x, Dv_\sigma(x))\le \sigma \ {\rm{for\ a.e.}}\ x\in V_\sigma.$$
By (A3), this implies that $|Dv_\sigma(x)|\le R_\sigma$ for a.e. $x\in V_\sigma$. This, combined with $u\in
{\rm{Lip}}_{\rm{loc}}(U)$, implies that $u_\sigma\in {\rm{Lip}}_{\rm{loc}}(U)$. This together with Claim 1
then implies (i).

Since $S^+u(x)\le \sigma$ and $S^+u_\sigma(x)=\sigma$ for $x\in V_\sigma$,
Lemma \ref{l6.x3} implies that for every $\ez>0$ there exists $\sz=\sigma(\epsilon)>0$
such that $$\big\|u-u_\sz\big\|_{L^\fz(U)}= \big\|u-u_\sz\big\|_{L^\fz(V_\sz)}<\ez.$$
In particular,  $u_\sz\to u$ uniformly in $U$, as $\sz\to0$.  This gives (ii).

Finally, for every $x\in U,$ the convexity of the map $t\mapsto T^t u_{\sz}(x)$ follows
from \eqref{e6.x14} when $ x\in V_\sz$; and
\eqref{e6.xx14} when $x\in U\setminus V_\sz$.
Moreover, it follows from Claim 2 and the upper semicontinuity of $S^+u$ in $U$
that $S^+u_\sigma$ is also upper semicontinuous.
Therefore $u_\sz$ satisfies the pointwise convexity criteria (iii). The proof is now complete.
\end{proof}

%\smallskip
\noindent{\bf Acknowledgement}. The first and third authors are partially supported by a grant from NSF of China (No. 11522102).
The second author is partially supported by NSF grant 1522869.

\noindent Qianyun Miao and Yuan Zhou

\noindent
Department of Mathematics, Beijing University of Aeronautics and Astronautics, Beijing 100191, P. R. China

\noindent{\it E-mail }:  \texttt{miaoqianyun@buaa.edu.cn} and \texttt{yuanzhou@buaa.edu.cn}

\medskip

\noindent Changyou Wang

\noindent  Departmet of Mathematics, Purdue University,
150 N. University Street, West Lafayette, IN 47907, USA.

\noindent {\it E-mail}: \texttt{wang2482@purdue.edu }  %@@!!@@!!

%\bigskip

%\noindent  Yuan Zhou

%\noindent
%Department of Mathematics, Beijing University of Aeronautics and Astronautics, Beijing 100191, P. R. China

%\noindent{\it E-mail }:  \texttt{yuanzhou@buaa.edu.cn}

%\end{CJK*}

\end{document}